%% file: Perturbation_Analysis_of_LICQ.tex
\documentclass[11pt]{article}
\usepackage[top=1in, bottom=1in, left=1in, right=1in]{geometry}

\usepackage[linktocpage,colorlinks,linkcolor=blue,anchorcolor=blue,citecolor=blue,urlcolor=blue,pagebackref]{hyperref}
\usepackage{euscript,mdframed}
\usepackage{microtype,todonotes,relsize}
\usepackage{amsmath,amsthm}
\usepackage{algpseudocode}
\usepackage{bbding}
\usepackage{pifont}
\usepackage{diagbox}

\newtheorem{statement}{Statement}

\usepackage{accents}
\input{_macros}
\allowdisplaybreaks
\usepackage[authoryear,round]{natbib}
\renewcommand*{\backref}[1]{\ifx#1\relax \else Page #1 \fi}
\renewcommand*{\backrefalt}[4]{%
  \ifcase #1 \footnotesize{(Not cited.)}%
  \or        \footnotesize{(Cited on page~#2.)}%
  \else      \footnotesize{(Cited on pages~#2.)}%
  \fi
}

\newcommand*{\colorboxed}{}
\def\colorboxed#1#{%
  \colorboxedAux{#1}%
}
\newcommand*{\colorboxedAux}[3]{%
  \begingroup
    \colorlet{cb@saved}{.}%
    \color#1{#2}%
    \boxed{%
      \color{cb@saved}%
      #3%
    }%
  \endgroup
}
\numberwithin{equation}{section}
\usepackage{xcolor} 
\usepackage{marginnote}

\newcommand{\todol}[2][]{{%
 \let\marginpar\marginnote
 \reversemarginpar
 \renewcommand{\baselinestretch}{0.8}%
 \todo[color=yellow]{#2}}}

\title{On the Nature of Regularity Assumptions in Bilevel Optimization with Constrained Lower-level Problem}

\author{Xiaotian Jiang, Chang He, Mingyi Hong, Shuzhong Zhang\footnote{X.~Jiang and S.~Zhang are with Department of Industrial and System Engineering, University of Minnesota, Minneapolis, MN, USA. E-mails: \texttt{jian0851@umn.edu}, \texttt{zhangs@umn.edu}. 
C.~He is with Research Institute for Interdisciplinary Sciences, Shanghai University of Finance and Economics, Shanghai, China. E-mail: \texttt{ischanghe@gmail.com}. M.~Hong is with Department of Electrical and Computer Engineering, University of Minnesota, Minneapolis, MN, USA. E-mail: \texttt{mhong@umn.edu}.}}

\date{\today}

\begin{document}
\maketitle

\begin{abstract}
In this paper, we study the regularity assumptions commonly adopted in bilevel optimization with constrained lower-level problems, including the linear independence constraint qualification, the strict complementary slackness condition, and the second-order sufficient condition. These conditions are typically required to hold for the lower-level problem at every upper-level variable $x$. We first show that the requirement that these conditions hold at every
upper-level variable $x$ is strong, in the sense that it is non-prevalent: there exist problems for which no sufficiently small perturbation of the lower-level objective and constraints can make the conditions hold at every $x$. To establish the result, we prove rigidity theorems showing that certain structural quantities of the lower-level problem must remain invariant across all $x$ whenever these conditions hold everywhere. We then construct explicit counterexamples in which these invariants differ between two values of $x$. In contrast, we show that the weaker requirement, that these conditions hold at almost every $x$, is a weak assumption, in the sense that it is prevalent: with probability one over a random perturbation of the lower-level objective and constraints, each condition holds at almost every $x$. We further analyze the gap between the two requirements. Although the ``every $x$'' and ``almost every $x$'' versions differ only on a measure-zero set, we show that this difference introduces fundamental difficulties in both theory and computation for bilevel optimization.
\end{abstract}

\section{Introduction}\label{sec:intro}

\noindent BiLevel Optimization (BLO) has emerged as a fundamental modeling paradigm for hierarchical decision-making problems across machine learning, operations research, and game theory. It has been applied widely in modern learning systems, including hyperparameter optimization \citep{maclaurin2015gradient,franceschi2018bilevel}, meta-learning \citep{finn2017model,franceschi2018bilevel}, reinforcement learning \citep{zeng2024demonstrations,hong2023ttsa}, and more recently machine unlearning \citep{reisizadeh2025blur} and large-scale model alignment \citep{li2024getting,li2025joint}. In such applications, BLO provides a principled framework for representing two-stage interactions in which a decision at the upper level depends implicitly on the solution to a lower-level optimization problem. Mathematically, a BLO problem consists of an upper-level objective that must be optimized over a set of solutions determined by a lower-level problem, often described as
\begin{align}\label{eq:blo_intro_basic}
\min_{x\in\mathcal{X}\subset\mathbb{R}^n} \ f(x,y^\ast(x))
\qquad\operatorname{s.t.}\qquad
y^\ast(x)\in S(x)\coloneq \argmin_{y\in\mathcal{Y}(x)\subset\mathbb{R}^m} g(x,y),
\end{align}
where $f$ is the upper-level objective and $g$ defines the lower-level problem, and $\mathcal{Y}(x)\subset\mathbb{R}^m$ denotes the lower-level feasible set.
In many applications, the lower-level feasible set $\mathcal{Y}(x)$ depends explicitly on the upper-level variable $x$. Constraints of this form are referred to as coupled constraints. In contrast, if the feasible set does not depend on $x$, the lower-level constraints are called uncoupled. In this work, we focus on the case where 
$$
\mathcal{Y}(x)\coloneq \{y\in\mathbb{R}^m : h_i(x,y)\le 0,\ i=1,\ldots,k\}
$$
is specified by a family of inequality constraint functions that may depend on $x$.

Much recent research has focused on BLO problems with functional constrained lower-level structures \citep{lu2024firstorderpenaltymethodsbilevel,kornowski2024firstordermethodslinearlyconstrained,jiang2024barrier,shen2025single}. The theoretical analysis of these models commonly requires regularity assumptions on the lower-level problem such as the linear independence constraint qualification~\citep{pmlr-v202-khanduri23a,xu2023efficient,kwon2023penalty,jiang2024primal}, strict complementary slackness condition~\citep{fliege2021gauss,giovannelli2021inexact,nolasco2025sensitivity}, and second-order sufficient conditions~\citep{hatz2013regularizing,giovannelli2021inexact,nolasco2025sensitivity} hold for the lower-level problem at every upper-level variable $x$. We recall how these conditions are defined at a given upper-level variable $x$. 

\begin{definition}\label{def:licq}
Linear Independence Constraint Qualification (LICQ) is said to hold at a feasible point 
$(x,y)$ of the lower-level problem at upper-level variable $x$ if the active constraint gradients with 
respect to the lower-level variable, $\{\nabla_y h_i(x,y): i\in \mathcal{J}(x,y)\}$, are linearly 
independent, where $\mathcal{J}(x,y)=\{i:h_i(x,y)=0\}$.
\end{definition}

\begin{definition}\label{def:scsc}
strict complementary slackness condition (SCSC) is said to hold at a KKT point $(y^\ast,\lambda^\ast)$ 
of the lower-level problem at upper-level variable $x$ if, for every constraint index $i$, either $h_i(x,y^\ast)<0$ and 
$\lambda_i^\ast=0$, or $h_i(x,y^\ast)=0$ and $\lambda_i^\ast>0$.
\end{definition}

\begin{definition}\label{def:sosc}
Second-Order Sufficient Conditions (SOSC) are said to hold at a KKT point $(y^\ast,\lambda^\ast)$ 
of the lower-level problem at upper-level variable $x$ if
\[
d^\top \nabla_{yy}^2 \mathcal{L}(x,y^\ast,\lambda^\ast)\, d > 0
\quad \text{for all nonzero } d \in \mathcal{C}(x, y^\ast),
\]
where $\mathcal{L}(x,y,\lambda) = g(x,y) + \sum_{i=1}^k \lambda_i h_i(x,y)$, and the critical cone $\mathcal{C}(x, y^\ast)$ is defined as
\[
\mathcal{C}(x, y^\ast)
\coloneq
\Big\{ d \in \mathbb{R}^m :
\nabla_y h_i(x,y^\ast)^\top d = 0 \ \text{if } \lambda_i^\ast > 0;
\quad
\nabla_y h_i(x,y^\ast)^\top d \le 0 \ \text{if } \lambda_i^\ast = 0,
\ \forall i\in \mathcal{J}(x,y^\ast)
\Big\}.
\]
\end{definition}

\begin{remark}[Uniform SOSC]\label{remark:uniform_sosc}

Beyond the standard pointwise SOSC, existing works \citep{giovannelli2021inexact, nolasco2025sensitivity} often assume a strengthened uniform version: there exists $\sigma > 0$, independent of $x$, such that $d^\top \nabla_{yy}^2 \mathcal{L}(x,y^\ast,\lambda^\ast)\, d \geq \sigma \|d\|^2$ for all $d \in \mathcal{C}(x, y^\ast)$.
In this paper, the prevalence results in Section~\ref{section:prevalence} concern the pointwise version (Definition~\ref{def:sosc}), while the rigidity and gap analysis in Sections~\ref{sec:nonprevalence} and \ref{section:obstructionforbilevel} require the uniform version. The distinction between the two and its implications for bilevel optimization are discussed in Section~\ref{subsec:discussion}.
\end{remark}

In the bilevel setting, these conditions are imposed on the lower-level problem viewed as a parametric optimization problem in $y$ with upper-level variable $x$. Specifically, when we say ``LICQ holds for the lower-level problem at upper-level variable $x$,'' we mean that LICQ holds at every feasible point $y \in \mathcal{Y}(x)$. Similarly, ``SCSC holds for the lower-level problem at upper-level variable $x$'' means that SCSC holds at every KKT point of the lower-level problem at $x$, and ``SOSC holds for the lower-level problem at upper-level variable $x$'' means that the uniform version (Remark~\ref{remark:uniform_sosc}) holds at every local minimizer of the lower-level problem at $x$.

These three conditions are standard in the optimization literature and are commonly assumed in works on bilevel optimization with constrained lower-level problems. In the bilevel setting, the specific combination of regularity conditions required depends on the structure of the lower-level problem. The following three combinations are commonly adopted in the literature:
\begin{itemize}[itemsep=1.5pt, topsep=1.5pt]\label{three_combinations}
    \item Strongly convex lower-level with LICQ~\citep{pmlr-v202-khanduri23a,xu2023efficient,kornowski2024firstordermethodslinearlyconstrained,jiang2024primal,shen2025single};
    \item Strongly convex lower-level with LICQ and SCSC~\citep{kwon2023penalty,jiang2025efficient};
    \item Non-strongly-convex lower-level with LICQ, SCSC, and uniform SOSC~\citep{giovannelli2021inexact,nolasco2025sensitivity}.
\end{itemize}
We formalize these requirements as the following statement.

\begin{statement}[``Every $x$'' regularity]\label{stmt:every_x}
For every $x \in \mathcal{X}$, LICQ (Definition~\ref{def:licq}), SCSC (Definition~\ref{def:scsc}), and uniform SOSC (Remark~\ref{remark:uniform_sosc}) hold for the lower-level problem.
\end{statement}

In this work, we aim to justify the validity of these conditions within the context of BLO, testing the three standard combinations listed above. A natural question that arises at this point is:

\begin{center}
\fbox{%
\parbox{0.92\linewidth}{%
\centering
\textit{\textbf{Question~1:} How strong is Statement~\ref{stmt:every_x}?}
}%
}
\end{center}
To address this question, we need a notion for quantifying the strength of an assumption. 
In this work, we adopt the perspective of prevalence to evaluate this strength.

\begin{definition}[Prevalence (informal)]\label{def:prevalent}
A property is said to be prevalent within a given function space if there exists a finite-dimensional family of perturbations such that, for any element in this space, a random perturbation from this family makes the perturbed element satisfy the property with probability one.
\end{definition}

The formal definition is given in Definition~\ref{def:prevalence-formal}. In this work, the function space is taken to be the space of smooth functions. In other words, if a property is prevalent, then after any sufficiently small perturbation, 
almost every perturbed problem satisfies the property. Such a property should therefore be viewed 
as a weak assumption. To contrast with this notion, we now introduce non-prevalent property.

\begin{definition}[Non-prevalence (informal)]\label{def:non-prevalent}
A property is said to be non-prevalent within a given function space if there exists an element in this space such that, for any finite-dimensional family of perturbations, the set of perturbations under which the perturbed element does not satisfy this property has positive measure.
\end{definition}

Formally, a property is non-prevalent if it is not prevalent in the sense of Definition~\ref{def:prevalence-formal}. In practice, we focus on a stronger form: there exist an element $v$ and a constant $\epsilon > 0$ such that every perturbation with norm less than $\epsilon$ yields a perturbed element $v + p$ that does not satisfy the property (see Section~\ref{sec:nonprevalence}). Such a property should therefore be regarded as a strong assumption. Based on the above definition, Question~1 can be restated as follows.

\begin{center}
\fbox{%
\parbox{0.92\linewidth}{%
\centering
\textit{\textbf{Question~1.1:} Is Statement~\ref{stmt:every_x} prevalent?}
}%
}
\end{center}
Regrettably, as a primary contribution of this work, we provide a negative answer to Question 1.1. The non-prevalence of ``every-$x$‘’ type regularity assumptions has also been observed in other bilevel settings~\citep{jiang2025correspondence}. 
Nonetheless, ``almost-every-$x$'' type conditions have already been adopted in the literature as a workable basis for algorithm design~\citep{pmlr-v202-khanduri23a, jiang2025correspondence}, which motivates us to formalize the following version of our regularity requirement:

\begin{statement}[``Almost every $x$'' regularity]\label{stmt:almost_every_x}
For almost every $x \in \mathcal{X}$, LICQ (Definition~\ref{def:licq}), SCSC (Definition~\ref{def:scsc}), and SOSC (Definition~\ref{def:sosc}) hold for the lower-level problem.
\end{statement}

This motivates a two-fold exploration of Question 1.1, again 
focusing on the three standard combinations listed above:
\begin{center}
\fbox{%
\parbox{0.92\linewidth}{%
\centering
\textit{\textbf{Question~2:} Is Statement~\ref{stmt:almost_every_x} prevalent? Why is Statement~\ref{stmt:every_x} non-prevalent?}
}%
}
\end{center}
Our analysis reveals a clear separation: Statement~\ref{stmt:almost_every_x} is prevalent~(Section~\ref{section:prevalence}), while Statement~\ref{stmt:every_x} is not~(Section~\ref{sec:nonprevalence}). We establish the prevalence of the former by utilizing Sard’s Theorem and transversality theory. To understand why the ``at every $x$'' requirement fails to be prevalent, we identify certain persistent structural patterns or invariants that cannot be eliminated by any $C^2$ perturbation of the lower-level objective and constraints whose norm is sufficiently small. These inherent obstructions offer a critical perspective for understanding the fundamental non-prevalence of such strong assumptions.

Having established the prevalence of Statement~\ref{stmt:almost_every_x} and identified the structural obstructions that prevent Statement~\ref{stmt:every_x} from being prevalent, it is natural to ask how different Statement~\ref{stmt:almost_every_x} really is from Statement~\ref{stmt:every_x}. Although these two statements differ only on a measure-zero set of upper-level variables, it is well known that in optimization, properties that differ only on a measure-zero set can still belong to fundamentally different problem classes. For example, a smooth function and a nonsmooth function may differ only on a measure-zero set of nondifferentiable points, yet they require entirely different algorithmic treatments. A similar phenomenon may arise here: the difference between Statement~\ref{stmt:every_x} and Statement~\ref{stmt:almost_every_x}, despite occurring only on a measure-zero set of upper-level variables, can lead to fundamentally different outcomes in a BLO problem with a constrained lower-level problem. This motivates the following question:
\begin{center}
\fbox{%
\parbox{0.92\linewidth}{%
\centering
\textit{\textbf{Question~3:} What is the gap between Statement~\ref{stmt:every_x} and Statement~\ref{stmt:almost_every_x}?}
}%
}
\end{center}
We will show that, despite differing only on a measure-zero set of 
upper-level variables, the two requirements lead to fundamentally different consequences: failure on this measure-zero set can cause the Lagrange multiplier to lose uniqueness and continuity, the hyperfunction to lose differentiability, or a branch of local minimizers to terminate abruptly. In all cases, the sensitivity formulas used for hypergradient computation become ill-conditioned in an entire neighborhood of each failure point (Section~\ref{section:obstructionforbilevel}).

\subsection{Related work}

\noindent\textbf{Regularity Assumptions in Bilevel Optimization with Constrained Lower-Level Problems:} A number of recent works on bilevel optimization with constrained lower-level problems rely on regularity assumptions, including LICQ, SCSC, and SOSC. \citet{giovannelli2021inexact} introduced a hyperfunction-based method to solve~\eqref{eq:blo_intro_basic} with nonconvex lower-level objective $g$ and nonconvex functional lower-level constraints $h$. To ensure the differentiability of hyperfunction, they assume $y^\ast(x)$ is unique for any $x$, and LICQ, SCSC, and SOSC hold at the lower-level solution $y^\ast(x)$ for every upper-level variable $x$. \citet{xu2023efficient} study bilevel problems~\eqref{eq:blo_intro_basic} with strongly convex lower-level objective $g$ and convex functional lower-level constraints $h$. They analyze the hypergradient at differentiable points and study the Clarke subdifferential of the hyperfunction at non-differentiable points. In their analysis, they assume that LICQ holds at the lower-level solution $y^\ast(x)$ for every upper-level variable $x$. \citet{kwon2023penalty} study problem \eqref{eq:blo_intro_basic} with nonconvex lower-level objective $g$ and convex lower-level constraint functions $h$. They introduce a penalized hyper-objective and analyze its relationship with the original bilevel problem~\eqref{eq:blo_intro_basic}. To ensure sufficient regularity of the lower-level solution set of \eqref{eq:blo_intro_basic}, they assume that for every upper-level variable $x$, there exists at least one lower-level solution that satisfies LICQ and SCSC. \citet{pmlr-v202-khanduri23a} consider \eqref{eq:blo_intro_basic} with strongly convex lower-level objective $g$ and uncoupled linear constraints $h$ of the form $\{y:Ay\leq b\}$. Under some technical assumptions, they show that adding a single random linear perturbation $q^\top y$ to the lower-level objective $g$ ensures that the resulting hyperfunction is differentiable at all iterates generated by a gradient-based algorithm with probability one. In their analysis, they also assume that LICQ holds at the lower-level solution $y^\ast(x)$ for every upper-level variable $x$. \citet{kornowski2024firstordermethodslinearlyconstrained} consider \eqref{eq:blo_intro_basic} with strongly convex lower-level objective $g$ and two kinds of lower-level constraints $h$: coupled linear equality constraints $\{y:Ax+By=b\}$ and coupled linear inequality constraints $\{y:Ax+By\leq b\}$. They require that LICQ holds for every feasible point $y$ and for every upper-level variable $x$ of the lower-level problem. This LICQ assumption ensures that the hypergradient is computable, and the Lagrange multipliers remain Lipschitz continuous with respect to the upper-level variable $x$ under the inequality constrained case. \citet{jiang2024primal} consider a Lagrangian-based penalty reformulation for bilevel problems~\eqref{eq:blo_intro_basic} with strongly convex lower-level objective $g$ and convex functional lower-level constraints $h$. To guarantee that the Lagrange multipliers associated with the lower-level problem remain well-defined and continuous with respect to the upper-level variable $x$, they assume that LICQ holds for every upper-level variable $x$ and for all feasible lower-level solutions $y$. \citet{nolasco2025sensitivity} study problem \eqref{eq:blo_intro_basic} with a convex lower-level objective $g$ and convex lower-level constraint functions $h$. In order to ensure that the lower-level solution $y^\ast(x)$ and its associated multipliers $\lambda^\ast(x)$ are differentiable with respect to the upper-level variable $x$, they assume that LICQ, SCSC, and SOSC hold at $y^\ast(x)$ for every $x$. \citet{shen2025single} consider problem \eqref{eq:blo_intro_basic} with a strongly convex lower-level objective $g$ and linear lower-level constraint functions of the form $\{y:Ay+Bx\leq b\}$. To ensure the existence and uniqueness of the lower-level Lagrange multipliers, they assume that LICQ holds at every feasible point $y$ for every upper-level variable $x$.\\

\noindent\textbf{Genericity of Regularity Assumptions in Single Level Constrained Optimization:} 

The genericity of regularity conditions in nonlinear optimization has been studied
under different notions of genericity. 
One notion of genericity is measure-theoretic and is formulated through
finite-dimensional perturbations. In the single-level setting,
\citet{spingarnrockafellar1979} considered linear perturbations of the objective
and constant perturbations of the constraints, and showed that,
for almost every perturbation parameter in the Lebesgue sense, LICQ, SCSC, and
SOSC hold. This notion of genericity is the same as the prevalence notion used in this paper.

A different notion of genericity is formulated in smooth function 
spaces: a property is called generic if the set of smooth functions 
satisfying it is open and dense in the corresponding strong $C^\infty$ 
topology. \citet{jongen2000nonlinear} developed a transversality-based theory establishing that, for an open and dense set of smooth problem data, LICQ holds at every feasible point and every critical point is nondegenerate. In particular, every local minimizer of such a problem satisfies LICQ, SCSC, and SOSC. For parametric programs $\min_y g(x,y)$ s.t.\ $h(x,y) \leq 0$ with 
parameter $x \in \mathbb{R}^n$, building on the transversality-based 
viewpoint of \cite{jongen2000nonlinear}, \citet{still2019genericity} 
shows that, for generic smooth problem data, the singular set where any of LICQ, SCSC, or SOSC fails is 
confined to strata of codimension at least $m+1$ in $(x,y)$-space, 
where $m$ is the dimension of $y$.

However, for parameter-dependent problems, a measure-theoretic genericity theory based on concrete finite-dimensional perturbations remains missing. Moreover, \cite{jongen2000nonlinear} and \cite{still2019genericity} show that, for generic smooth problem data,
degeneracies, if present, occur only on small subsets of the parameter-state
space. They do not show that degeneracies are unavoidable for a robust class of problems.
In other words, they do not construct robust examples for which no sufficiently small perturbation
can make the every-\(x\) regularity property hold.

Beyond the classical NLP setting, related genericity and stability results 
have been established for several structured problem classes. These include 
linear conic and semidefinite programming~\citep{alizadeh1997complementarity,
dur2017genericity}; semi-algebraic, definable, and tame optimization, with 
emphasis on constraint qualifications, active manifolds, and continuity of 
set-valued maps~\citep{daniilidis2011continuity,drusvyatskiy2016generic,
lee2017generic,bolte2018qualification}; and bilevel programs, where 
genericity of partial calmness has been established via singularity-theoretic 
classifications of generalized critical points~\citep{ke2022generic}.

\subsection{Contributions}

In this work, we study the strength and limitations of the regularity assumptions. Our main contributions can be summarized as follows:

\begin{itemize}[itemsep=1.5pt, topsep=1.5pt]

\item We prove that Statement~\ref{stmt:every_x} is non-prevalent in the sense of Definition~\ref{def:non-prevalent}, and hence is a strong assumption. The argument proceeds in two stages. First, we establish rigidity theorems showing that if a regularity condition holds at every $x$, then certain structural quantities must remain invariant across all $x \in \mathcal{X}$: LICQ implies the stratification of the feasible set $\mathcal{Y}(x)$ is preserved (Theorem~\ref{thm:main_licq}); LICQ and SCSC together imply that no continuous KKT trajectory can migrate across different strata of $\mathcal{Y}(x)$ (Theorem~\ref{thm:main_scsc}); and LICQ, SCSC, and uniform SOSC together imply the number of local minimizers on each stratum is constant (Theorem~\ref{thm:main_sosc}). Second, we construct explicit counterexamples in which these structural invariants differ between two different $x$, and prove perturbation stability results (Theorems~\ref{thm:stab_strat_perturb},~\ref{thm:stab_scsc},~\ref{thm:stab_strat_lmins_perturb_noM}) guaranteeing that these differences persist under any sufficiently small perturbation. The rigidity theorems then certify that the corresponding regularity condition must fail at some $x$ for any perturbed problem (Counterexamples~\ref{counterexample:licq2},~\ref{counterexample:licq3},~\ref{counterexample:scsc1},~\ref{counterexample:sosc}). Notably, Theorem~\ref{thm:main_sosc} is different in nature from classical Robinson-type stability results~\cite{robinson1980strongly}. Robinson-type results are local: they provide a local continuation of a given solution or KKT point in a neighborhood of itself. Theorem~\ref{thm:main_sosc} is global: it asserts that the stratum-wise count of local minimizers is constant across the entire upper-level variable space $\mathcal{X}$, which in particular rules out the emergence of new local minimizers at any point of $\mathcal{X}$. Establishing this global statement requires, beyond local continuation, a closedness property of the graph of local minimizers.

\item We show that Statement~\ref{stmt:almost_every_x} is prevalent in the sense of Definition~\ref{def:prevalent}, and hence is a weak assumption. Specifically, after applying an arbitrarily small constant perturbation to the constraint functions and a linear perturbation to the objective, each of these three conditions holds at almost every $x$ with probability one (Theorems~\ref{thm:LICQ},~\ref{thm:SCSC},~\ref{thm:SOSC}). This fills a gap in the measure-theoretic genericity theory for parameter-dependent problems based on concrete finite-dimensional perturbations. The single-level case follows immediately by taking the upper-level variable space to be a singleton, recovering the genericity results presented in \citet{spingarnrockafellar1979}.

\item We analyze the gap between Statement~\ref{stmt:every_x} and Statement~\ref{stmt:almost_every_x} by examining the structural consequences of regularity failure on a measure-zero set. We show that LICQ failure causes the Lagrange multiplier to lose uniqueness and continuity (Section~\ref{subsec:role_licq}); SCSC failure under strongly convex lower-level causes the hyperfunction to lose differentiability, turning the bilevel problem from smooth into nonsmooth (Section~\ref{subsec:role_scsc}); and SCSC or uniform SOSC failure under non-strongly-convex lower-level can cause a branch of local minimizers to terminate abruptly, making the solution mapping discontinuous (Section~\ref{subsec:role_sosc}). In all three cases, the sensitivity formulas~\eqref{eq:sensitivity_formula}\eqref{eq:complementarity_sensitivity} used for hypergradient computation become ill-conditioned not only at the failure points but throughout an entire neighborhood surrounding them;

\item We show that LICQ, SCSC, and uniform SOSC together imply uniform local quadratic growth over the upper-level variable space (Theorem~\ref{prop:uniform_quadratic_growth}). This uniform growth condition is the structural foundation underlying the global smooth continuation of local minimizer branches (Theorem~\ref{thm:global_continuation}). We further clarify the distinct roles played by each condition. Uniform SOSC controls curvature within the critical cone. SCSC provides a uniform lower bound on the multipliers at active constraints, ensuring quadratic growth in directions outside the critical cone.

\end{itemize}

\section{Non-Prevalence of LICQ, SCSC, and SOSC: ``Every-$x$'' Case}\label{sec:nonprevalence}

In this section, we address Question~1.1 from the introduction: whether Statement~\ref{stmt:every_x}, which requires LICQ, SCSC, and uniform SOSC to hold at every $x$ under the three settings listed in Section~\ref{sec:intro}, is prevalent. We show that the answer is negative. For each of LICQ, SCSC, and uniform SOSC, we construct explicit counterexamples demonstrating that no sufficiently small $C^2$ perturbation of the lower-level objective and constraints can make the corresponding condition hold at all upper-level variables $x$ simultaneously.

Our approach has two stages. In the first stage, we prove rigidity theorems: if a regularity condition holds at every $x$, then certain structural features of the lower-level problem must remain invariant across all $x \in \mathcal{X}$. The three rigidity results below correspond to the three settings commonly adopted in the bilevel literature, respectively: (i) strongly convex lower-level with LICQ, (ii) strongly convex lower-level with LICQ and SCSC, and (iii) non-strongly-convex lower-level with LICQ, SCSC, and uniform SOSC (see Remark~\ref{remark:uniform_sosc}):
\begin{itemize}
    \item \textbf{LICQ for every $x$} implies that the feasible set $\mathcal Y(x)$, viewed as a manifold with generalized boundary (Figure~\ref{fig:MGB}), has a natural stratification (indexed by active constraint patterns) that is preserved across all $x \in \mathcal X$ up to a stratification-preserving diffeomorphism (Theorem~\ref{thm:main_licq}); this result is independent of the convexity of the lower-level objective.

    \item \textbf{LICQ and SCSC for every $x$} further imply that the active constraint index set along any continuous KKT trajectory is constant (Theorem~\ref{thm:main_scsc}); this result does not require strong convexity, although strong convexity automatically provides such a trajectory via the unique minimizer. Here, a KKT trajectory refers to a continuous map $x \mapsto y^*(x)$ such that $y^*(x)$ is a KKT point of the lower-level problem at each $x$.

    \item \textbf{LICQ, SCSC, and uniform SOSC for every $x$} further imply that the number of local minimizers on each stratum of $\mathcal Y(x)$ stays constant across all $x \in \mathcal X$ (Theorem~\ref{thm:main_sosc}).
\end{itemize}

In the second stage, we use these rigidity theorems to construct concrete counterexamples. We exhibit specific problems in which the relevant structural quantity (the stratification type, the active set along a KKT trajectory, or the stratum-wise count of local minimizers) differs between two upper-level variables $x_1$ and $x_2$. To ensure that this mismatch is robust, we establish perturbation stability results (Theorems~\ref{thm:stab_strat_perturb}, \ref{thm:stab_scsc}, \ref{thm:stab_strat_lmins_perturb_noM}) showing that each structural quantity is preserved at any fixed $x$ under sufficiently small $C^2$ perturbations of the lower-level data. Combined with the rigidity theorems, these stability results imply that the corresponding regularity condition must fail at some $x$ for every problem in a $C^2$-neighborhood of our counterexamples.

The rigidity theorems also provide necessary conditions for the ``every-$x$'' requirement that serve as practical screening criteria. To verify whether the ``every-$x$'' requirement is achievable for a given problem, one can first examine the original (unperturbed) problem: if the relevant structural quantity, such as the stratification type of $\mathcal{Y}(x)$ or the number of local minimizers on each stratum, differs between two different points $x_1$ and $x_2$, then the corresponding regularity condition cannot hold at every $x$. Moreover, the stability results guarantee that this conclusion is robust: since these quantities do not change under small perturbations at each fixed $x$, the same obstruction persists for all nearby problems.

\subsection{Perturbation Setup and Technical Preliminaries}
\label{subsection:technical_preliminaries}

Before presenting the main results, we clarify the perturbation setup,
the compact-local notion of smallness, and several standing definitions and
assumptions.

\paragraph{Perturbation setup.}
For a
lower-level problem
\begin{align*}
    \min_y \ g(x,y) \quad \operatorname{s.t.} \quad h_i(x,y)\le 0,\quad i=1,\ldots,k,
\end{align*}
we consider perturbed data of the form
\begin{align}\label{eq:pertubedproblem}
    \widetilde g(x,y)\coloneq g(x,y)+s(x,y),\qquad
    \widetilde h_i(x,y)\coloneq h_i(x,y)+r_i(x,y),\quad i=1,\ldots,k,
\end{align}
where \(s\) and \(r_i\)'s are smooth perturbations. In the LICQ-only setting,
only the constraints are relevant, and one may take \(s\equiv 0\). We next describe the compact-local \(C^2\)-norm used throughout this section to measure perturbation size.
\begin{definition}[Compact-local \(C^2\)-norm]\label{def:localc2norm}
Let \(K\subset \mathbb R^m\) be compact, and let
\(\phi:\mathbb R^m\to\mathbb R\) be twice continuously differentiable on a
neighborhood of \(K\). We define
\[
    \|\phi\|_{C^2(K)}
    \coloneq 
    \max\Big\{
        \sup_{y\in K}|\phi(y)|,\,
        \sup_{y\in K}\|\nabla_y \phi(y)\|,\,
        \sup_{y\in K}\|\nabla^2_{yy}\phi(y)\|_{\rm op}
    \Big\},
\]
where \(\|\cdot\|_{\rm op}\) denotes the operator norm. For a function
\(\phi(x,y)\), we use \(\|\phi(x,\cdot)\|_{C^2(K)}\) to denote the above norm
with respect to the lower-level variable \(y\), with \(x\) fixed.
\end{definition}

Throughout this section, the size of a perturbation is understood in this
compact-local sense. Namely, after a compact set \(K\subset\mathbb R^m\) is
chosen, a perturbation is said to be small if its \(C^2(K)\)-norm is small. For
the perturbed lower-level data
\[
    \widetilde g(x,y)\coloneq g(x,y)+s(x,y),\qquad
    \widetilde h_i(x,y)\coloneq h_i(x,y)+r_i(x,y),
\]
this means that, for fixed \(x\), the following term
\[
    \|s(x,\cdot)\|_{C^2(K)}
    +\max_{i=1,\ldots,k}\|r_i(x,\cdot)\|_{C^2(K)}
\]
is sufficiently small. This formulation does not require the perturbation to be
globally small on all of \(\mathbb R^m\). It is therefore compatible with
standard perturbations such as linear perturbations of the objective: a linear
function does not have a finite global \(C^2\)-norm, but its \(C^2\)-norm on any fixed compact set is finite and can be
made arbitrarily small by scaling.

\paragraph{Inner and outer compact sets.}
In the stability results (Theorem~\ref{thm:stab_strat_perturb}, \ref{thm:stab_scsc}, \ref{thm:stab_strat_lmins_perturb_noM}), we use two compact sets \(K_{\rm in}\) and
\(K_{\rm out}\). We briefly explain their roles here.
After fixing an upper-level value \(x\), we choose compact sets
\[
    \mathcal Y(x)\subset \operatorname{int}(K_{\rm in}),\ \text{and} \ 
    K_{\rm in}\subset \operatorname{int}(K_{\rm out}).
\]
The set \(K_{\rm in}\) is the region where the local geometry of the lower-level
problem is compared. For example, we only consider the relevant feasible-set strata, KKT points,
or local minimizers inside \(K_{\rm in}\). The larger set
\(K_{\rm out}\) is a region on which the perturbation is controlled:
\[
    \|s(x,\cdot)\|_{C^2(K_{\rm out})}
    +\max_{i=1,\ldots,k}\|r_i(x,\cdot)\|_{C^2(K_{\rm out})}\le \epsilon .
\]
The use of two compact sets is technical. To use the stability results of \citet{jongen2000nonlinear}, which are stated
in the strong \(C^2\) topology (see proof of Theorem~\ref{thm:stab_strat_perturb} in Appendix~\ref{proof:thm:stab_strat_perturb}), under our compact-local \(C^2(K_{\rm out})\)
control, we use a cutoff argument to localize the perturbation. For this purpose, the perturbation is required to be small on a slightly larger
compact set \(K_{\rm out}\), while the structural comparison is made on the
smaller compact set \(K_{\rm in}\). The strict inclusion
\(K_{\rm in}\subset \operatorname{int}(K_{\rm out})\) leaves room for choosing a
smooth cutoff function that equals one near \(K_{\rm in}\) and vanishes outside
\(K_{\rm out}\). Thus \(K_{\rm in}\) is the region of comparison, while
\(K_{\rm out}\) is the region where the perturbation is controlled.

For clarity and to emphasize the insight, the counterexamples in this section
are stated without explicitly displaying the auxiliary compact sets
\(K_{\rm in}\) and \(K_{\rm out}\). The fully localized version of the argument is given in
Appendix~\ref{sec:local_version}.

\paragraph{KKT trajectories and assumption.}
We will use the following terminology. A KKT trajectory is a continuous map $x \mapsto y^\ast(x)$ defined on a subset of $\mathcal{X}$ such that $y^\ast(x)$ is a KKT point of the lower-level problem at each $x$. A local minimizer branch is a continuous map $x \mapsto y^\ast(x)$ such that $y^\ast(x)$ is a local minimizer of the lower-level problem at each $x$. A continuation of a local minimizer $y_0$ at $x_0 \in \mathcal{X}$ is a local minimizer branch $y^\ast(\cdot)$ passing through $y_0$, meaning $y^\ast(x_0) = y_0$. We impose the following assumption for the results in Sections~\ref{sec:nonprevalence} and~\ref{section:obstructionforbilevel}.
\begin{assumption}\label{assumption:general2} The following conditions hold.
    \begin{enumerate}
        \item The lower-level objective $g(x,y)$ and constraints $h_i(x,y)$ are all smooth, i.e., $g,h_i\in C^\infty$.\label{assumption:regularity1}
        \item The lower-level feasible set $\mathcal{Y}(x)$ is nonempty for every $x\in\mathcal{X}$.\label{assumption:regularity2}
    \end{enumerate}
\end{assumption}
The first condition ensures the regularity needed for the Implicit Function Theorem, the differentiability of KKT systems, and the topological tools from \citet{jongen2000nonlinear} used throughout the proofs. In principle, finite-order differentiability suffices, but we assume $C^\infty$ for simplicity.

\subsection{Non-prevalence of LICQ}\label{subsection:licq}

We begin by recalling some terminology from \citet{jongen2000nonlinear} that will be used in the rigidity and stability results below; formal definitions are given in Appendix~\ref{sec:appendix_def}. A manifold with generalized boundary (MGB) is a set that is locally diffeomorphic to $\mathbb{R}^q \times \mathbb{R}^p_+$ for some integers $q, p \geq 0$. Informally, it is a set with smooth interior and well-behaved corners where multiple boundary faces meet. A feasible set defined by smooth inequality constraints satisfying LICQ is a typical example of an MGB \citep[page~87]{jongen2000nonlinear}. Every MGB carries a natural stratification: a decomposition into layers called strata, each consisting of points that share the same local boundary structure. For a constraint set satisfying LICQ, this stratification corresponds to the partition by active constraint patterns together with connectivity: each stratum is a connected component of the set $S_{\mathcal{J}}(x) \coloneq  \{y : \mathcal{J}(x,y) = \mathcal{J}\}$ for some index set $\mathcal{J}$, where $\mathcal{J}(x,y) = \{i : h_i(x,y) = 0\}$. Two feasible points lie in the same stratum if and only if they have the same active index set and belong to the same connected component of $S_{\mathcal{J}}$. For example, in Figure~\ref{fig:MGB}, the strata are the interior ($\mathcal{J} = \emptyset$), the four boundary arcs (each with a distinct singleton $\mathcal{J}$), and the four vertices (each with a distinct two-element $\mathcal{J}$). Two MGBs are said to share the same stratification if there exists a diffeomorphism between them that maps each stratum of one onto a corresponding stratum of the other.

\begin{figure}
    \centering
\tikzset{every picture/.style={line width=0.75pt}} 

\tikzset{every picture/.style={line width=0.75pt}} 

\begin{tikzpicture}[x=0.75pt,y=0.75pt,yscale=-1,xscale=1]

\draw    (56.27,154.4) -- (220.27,154.4) ;
\draw [shift={(222.27,154.4)}, rotate = 180] [color={rgb, 255:red, 0; green, 0; blue, 0 }  ][line width=0.75]    (10.93,-3.29) .. controls (6.95,-1.4) and (3.31,-0.3) .. (0,0) .. controls (3.31,0.3) and (6.95,1.4) .. (10.93,3.29)   ;
\draw    (132.27,244.4) -- (132.27,74.4) ;
\draw [shift={(132.27,72.4)}, rotate = 90] [color={rgb, 255:red, 0; green, 0; blue, 0 }  ][line width=0.75]    (10.93,-3.29) .. controls (6.95,-1.4) and (3.31,-0.3) .. (0,0) .. controls (3.31,0.3) and (6.95,1.4) .. (10.93,3.29)   ;
\draw    (100,101) .. controls (176.09,83.26) and (233.09,116.26) .. (236.09,120.26) ;
\draw    (86.09,227.26) .. controls (126.09,197.26) and (112.09,79.26) .. (152.09,49.26) ;
\draw    (56.27,154.4) .. controls (96.27,124.4) and (165.09,226.26) .. (205.09,196.26) ;
\draw    (169.09,246.26) .. controls (152.09,226.26) and (144.09,121.26) .. (184.09,91.26) ;
\draw  [fill={rgb, 255:red, 0; green, 0; blue, 0 }  ,fill opacity=1 ] (112.09,164.59) .. controls (112.09,163.21) and (113.21,162.09) .. (114.59,162.09) .. controls (115.97,162.09) and (117.09,163.21) .. (117.09,164.59) .. controls (117.09,165.97) and (115.97,167.09) .. (114.59,167.09) .. controls (113.21,167.09) and (112.09,165.97) .. (112.09,164.59) -- cycle ;
\draw  [fill={rgb, 255:red, 0; green, 0; blue, 0 }  ,fill opacity=1 ] (153.09,190.59) .. controls (153.09,189.21) and (154.21,188.09) .. (155.59,188.09) .. controls (156.97,188.09) and (158.09,189.21) .. (158.09,190.59) .. controls (158.09,191.97) and (156.97,193.09) .. (155.59,193.09) .. controls (154.21,193.09) and (153.09,191.97) .. (153.09,190.59) -- cycle ;
\draw  [fill={rgb, 255:red, 0; green, 0; blue, 0 }  ,fill opacity=1 ] (125.09,96.59) .. controls (125.09,95.21) and (126.21,94.09) .. (127.59,94.09) .. controls (128.97,94.09) and (130.09,95.21) .. (130.09,96.59) .. controls (130.09,97.97) and (128.97,99.09) .. (127.59,99.09) .. controls (126.21,99.09) and (125.09,97.97) .. (125.09,96.59) -- cycle ;
\draw  [fill={rgb, 255:red, 0; green, 0; blue, 0 }  ,fill opacity=1 ] (174.09,98.59) .. controls (174.09,97.21) and (175.21,96.09) .. (176.59,96.09) .. controls (177.97,96.09) and (179.09,97.21) .. (179.09,98.59) .. controls (179.09,99.97) and (177.97,101.09) .. (176.59,101.09) .. controls (175.21,101.09) and (174.09,99.97) .. (174.09,98.59) -- cycle ;
\draw [line width=1.5]    (114.59,164.59) .. controls (121.04,132.62) and (122.04,116.12) .. (127.59,96.59) ;
\draw [line width=1.5]    (127.59,96.59) .. controls (147.04,95.12) and (163.04,96.12) .. (179.09,98.59) ;
\draw [line width=1.5]    (155.59,193.09) .. controls (155.04,158.12) and (156.04,125.12) .. (176.59,98.59) ;
\draw [line width=1.5]    (114.59,164.59) .. controls (125.04,171.12) and (139.04,182.12) .. (155.59,190.59) ;

\draw (152,27.4) node [anchor=north west][inner sep=0.75pt]  [color={rgb, 255:red, 0; green, 0; blue, 0 }  ,opacity=1 ]  {$h_{1}^{-1}( 0)$};
\draw (29,123.4) node [anchor=north west][inner sep=0.75pt]  [color={rgb, 255:red, 0; green, 0; blue, 0 }  ,opacity=1 ]  {$h_{2}^{-1}( 0)$};
\draw (236,96.4) node [anchor=north west][inner sep=0.75pt]  [color={rgb, 255:red, 0; green, 0; blue, 0 }  ,opacity=1 ]  {$h_{3}^{-1}( 0)$};
\draw (171,232.4) node [anchor=north west][inner sep=0.75pt]  [color={rgb, 255:red, 0; green, 0; blue, 0 }  ,opacity=1 ]  {$h_{4}^{-1}( 0)$};
\draw (227,151.4) node [anchor=north west][inner sep=0.75pt]    {$y_{1}$};
\draw (117,64.4) node [anchor=north west][inner sep=0.75pt]    {$y_{2}$};
\draw (135.09,131.99) node [anchor=north west][inner sep=0.75pt]    {$\mathcal{Y}$};

\end{tikzpicture}
    \caption{A two-dimensional feasible set $\mathcal{Y}$ defined by four inequality constraints $h_1, h_2, h_3, h_4 \leq 0$. The curves $h_i^{-1}(0)$ are the constraint boundaries. Under LICQ, the feasible set is an MGB whose natural stratification consists of three types of strata: the interior (points where no constraint is active), the boundary arcs (one-dimensional pieces where exactly one constraint is active), and the vertices (zero-dimensional points where exactly two constraints are active). In this example, there are one interior component, four boundary arcs, and four vertices.}
    \label{fig:MGB}
\end{figure}
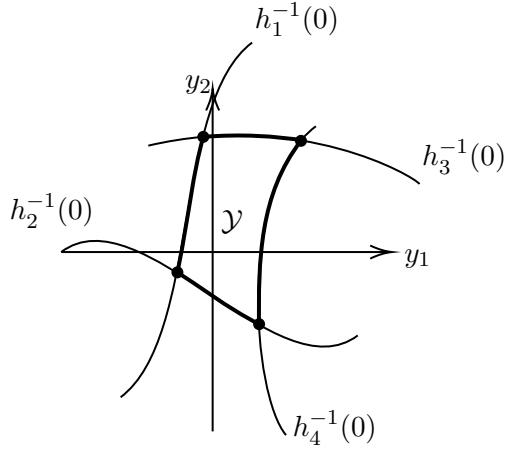

We are now ready to state the rigidity theorem for LICQ. It says that if LICQ holds for the lower-level problem at every $x$, then the stratification of $\mathcal{Y}(x)$ cannot change as $x$ varies. As a technical requirement, we assume $\mathcal{X}$ to be a connected MGB. Intuitively, ``connected'' means $\mathcal{X}$ consists of a single piece rather than several disjoint components; the formal definition is recalled in Appendix~\ref{sec:appendix_def}. This connectedness requirement is mild and is satisfied by most upper-level variable spaces arising in practice, such as boxes, balls, convex sets, or smooth manifolds with a single component. In particular, $\mathcal{X}$ is a connected MGB whenever it is defined by smooth inequality constraints satisfying LICQ over a connected set.

\begin{theorem}[Rigidity of stratification]\label{thm:main_licq}
    Suppose Assumption~\ref{assumption:general2} holds, the upper-level feasible set $\mathcal{X}$ is a connected manifold with generalized boundary (MGB), and the lower-level feasible set $\mathcal{Y}(x)$ is uniformly bounded. Furthermore, assume that, at every $x \in \mathcal{X}$, LICQ holds for the lower-level problem. Then, for any $x \in \mathcal{X}$, the lower-level feasible set $\mathcal{Y}(x)$ shares the same natural stratification as a manifold with generalized boundary (MGB). Specifically, for any $x_1, x_2 \in \mathcal{X}$, there exists a stratification-preserving diffeomorphism between $\mathcal{Y}(x_1)$ and $\mathcal{Y}(x_2)$.
\end{theorem}

The proof builds on \citet[Theorem~10.4.5]{jongen2000nonlinear}, which 
establishes the corresponding diffeomorphism result for feasible sets 
depending on a single parameter. Extending this to our multi-parameter 
setting requires the connected MGB structure of $\mathcal{X}$: for any 
$x_1, x_2 \in \mathcal{X}$, we use local path-connectedness to connect 
them by a piecewise smooth path $\gamma$, apply the one-parameter result 
on each smooth segment of $\gamma$, and compose the resulting 
diffeomorphisms to obtain a diffeomorphism between $\mathcal{Y}(x_1)$ 
and $\mathcal{Y}(x_2)$. The stratification preservation then follows 
from \citep[Corollary~3.1.29]{jongen2000nonlinear}. The full proof is 
given in Appendix~\ref{proof:thm:main_licq}.

Theorem~\ref{thm:main_licq} establishes a necessary condition for the ``every $x$'' LICQ requirement: the stratification of $\mathcal{Y}(x)$ must remain invariant across all $x \in \mathcal{X}$. To show that this requirement is non-prevalent, our strategy is to construct a specific problem where the stratification of $\mathcal{Y}(x)$ strictly differs between two upper-level variables $x_1$ and $x_2$. However, to ensure that the mismatch between $x_1$ and $x_2$ persists even after perturbing the lower-level constraints, we first need to establish that the stratification at any fixed upper-level variable $x$ is stable under small perturbations. The following theorem provides exactly this stability guarantee.

\begin{theorem}[Stability of stratification under perturbation]\label{thm:stab_strat_perturb}
    Fix any \(x\in\mathcal X\). Suppose Assumption~\ref{assumption:general2} holds,
    the lower-level feasible set \(\mathcal Y(x)\) is bounded, and LICQ holds for the
    lower-level problem at \(x\). For any compact sets \(K_{\rm in},K_{\rm out}\subset\mathbb R^m\)
    such that
    \(\mathcal Y(x)\subset \operatorname{int}(K_{\rm in})\) and
    \(K_{\rm in}\subset \operatorname{int}(K_{\rm out})\), there exists a constant \(\epsilon(x,K_{\rm in},K_{\rm out})>0\) such that the following holds: for any \(C^2\) perturbations of the lower-level constraint functions, as in \eqref{eq:pertubedproblem}, satisfying
    \[
    \max_i\|r_i(x,\cdot)\|_{C^2(K_{\rm out})}\le \epsilon(x,K_{\rm in},K_{\rm out}),
    \]
    (see Definition~\ref{def:localc2norm} for the compact-local \(C^2\)-norm), the set \(\widetilde{\mathcal Y}(x)\cap K_{\rm in}\), where
    \[
    \widetilde{\mathcal Y}(x)\coloneq \{y:h_i(x,y)+r_i(x,y)\le 0,\ i=1,\ldots,k\},
    \]
    has the same natural stratification as \(\mathcal Y(x)\). Specifically, there
    exists a stratification-preserving diffeomorphism between \(\mathcal Y(x)\)
    and \(\widetilde{\mathcal Y}(x)\cap K_{\rm in}\).
\end{theorem}

The proof combines the structural stability theorem of \citet[Theorem~6.3.1]{jongen2000nonlinear} with a cutoff argument. \citet[Theorem~6.3.1]{jongen2000nonlinear} yields a stratification-preserving diffeomorphism under any sufficiently small perturbation in the strong $C^2$-topology of $\mathbb{R}^m$. Since our perturbation is only assumed small on $K_{\mathrm{out}}$ rather than globally, we multiply each $r_i$ by a smooth cutoff equal to one near $K_{\mathrm{in}}$ and supported in $K_{\mathrm{out}}$, producing a globally small perturbation to which the cited theorem applies. The full proof is given in Appendix~\ref{proof:thm:stab_strat_perturb}.

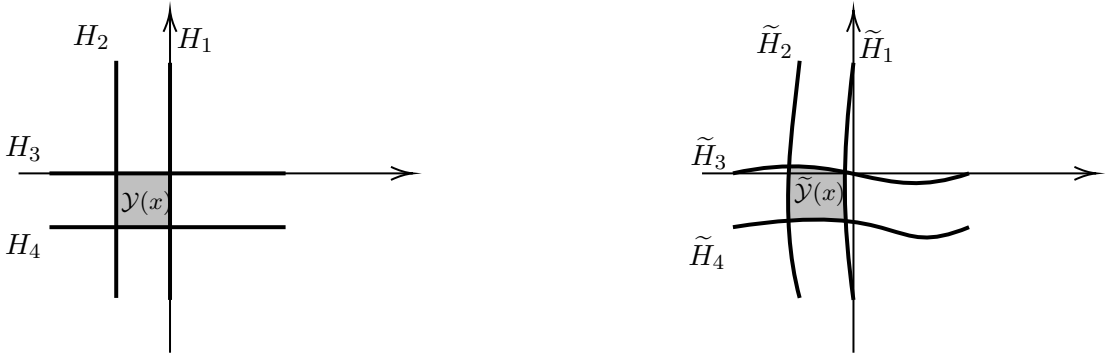
\begin{figure}
    \centering

\tikzset{every picture/.style={line width=0.75pt}} 

\begin{tikzpicture}[x=0.75pt,y=0.75pt,yscale=-1,xscale=1]

\draw    (432.27,129.4) -- (628.09,129.4) ;
\draw [shift={(630.09,129.4)}, rotate = 180] [color={rgb, 255:red, 0; green, 0; blue, 0 }  ][line width=0.75]    (10.93,-3.29) .. controls (6.95,-1.4) and (3.31,-0.3) .. (0,0) .. controls (3.31,0.3) and (6.95,1.4) .. (10.93,3.29)   ;
\draw    (508.27,219.4) -- (508.27,49.4) ;
\draw [shift={(508.27,47.4)}, rotate = 90] [color={rgb, 255:red, 0; green, 0; blue, 0 }  ][line width=0.75]    (10.93,-3.29) .. controls (6.95,-1.4) and (3.31,-0.3) .. (0,0) .. controls (3.31,0.3) and (6.95,1.4) .. (10.93,3.29)   ;
\draw [line width=1.5]    (481.27,72.9) .. controls (478,104.26) and (470,150.26) .. (481.27,191.9) ;
\draw [line width=1.5]    (447.77,156.4) .. controls (537,142.26) and (525,173.26) .. (566.27,156.4) ;
\draw [line width=1.5]    (508.27,73.9) .. controls (502,124.26) and (503,154.26) .. (508.27,192.9) ;
\draw [line width=1.5]    (447.77,129.4) .. controls (507,116.26) and (520,145.26) .. (566.27,129.4) ;
\draw  [color={rgb, 255:red, 255; green, 255; blue, 255 }  ,draw opacity=0 ][fill={rgb, 255:red, 0; green, 0; blue, 0 }  ,fill opacity=0.25 ] (476,125.26) .. controls (477.06,125.48) and (503.94,128.04) .. (504,129.26) .. controls (504.06,130.48) and (504.94,153.37) .. (504,153.26) .. controls (503.06,153.14) and (476.39,153.14) .. (476,152.26) .. controls (475.61,151.37) and (474.94,125.04) .. (476,125.26) -- cycle ;
\draw    (89.27,129.4) -- (285.09,129.4) ;
\draw [shift={(287.09,129.4)}, rotate = 180] [color={rgb, 255:red, 0; green, 0; blue, 0 }  ][line width=0.75]    (10.93,-3.29) .. controls (6.95,-1.4) and (3.31,-0.3) .. (0,0) .. controls (3.31,0.3) and (6.95,1.4) .. (10.93,3.29)   ;
\draw    (165.27,219.4) -- (165.27,49.4) ;
\draw [shift={(165.27,47.4)}, rotate = 90] [color={rgb, 255:red, 0; green, 0; blue, 0 }  ][line width=0.75]    (10.93,-3.29) .. controls (6.95,-1.4) and (3.31,-0.3) .. (0,0) .. controls (3.31,0.3) and (6.95,1.4) .. (10.93,3.29)   ;
\draw [color={rgb, 255:red, 0; green, 0; blue, 0 }  ,draw opacity=1 ][line width=1.5]    (165.27,73.9) -- (165.27,192.9) ;
\draw [color={rgb, 255:red, 0; green, 0; blue, 0 }  ,draw opacity=1 ][line width=1.5]    (104.77,129.4) -- (223.27,129.4) ;
\draw [color={rgb, 255:red, 0; green, 0; blue, 0 }  ,draw opacity=1 ][line width=1.5]    (138.27,72.9) -- (138.27,191.9) ;
\draw [color={rgb, 255:red, 0; green, 0; blue, 0 }  ,draw opacity=1 ][line width=1.5]    (104.77,156.4) -- (223.27,156.4) ;
\draw  [color={rgb, 255:red, 0; green, 0; blue, 0 }  ,draw opacity=1 ][fill={rgb, 255:red, 0; green, 0; blue, 0 }  ,fill opacity=0.25 ] (138.27,129.4) -- (165.09,129.4) -- (165.09,156.4) -- (138.27,156.4) -- cycle ;

\draw (509.27,55.8) node [anchor=north west][inner sep=0.75pt]  [color={rgb, 255:red, 0; green, 0; blue, 0 }  ,opacity=1 ]  {$\widetilde H_1$};
\draw (425,109.4) node [anchor=north west][inner sep=0.75pt]  [color={rgb, 255:red, 0; green, 0; blue, 0 }  ,opacity=1 ]  {$\widetilde H_{3}$};
\draw (458,53.4) node [anchor=north west][inner sep=0.75pt]  [color={rgb, 255:red, 0; green, 0; blue, 0 }  ,opacity=1 ]  {$\widetilde H_{2}$};
\draw (424,159.4) node [anchor=north west][inner sep=0.75pt]  [color={rgb, 255:red, 0; green, 0; blue, 0 }  ,opacity=1 ]  {$\widetilde H_{4}$};
\draw (167.27,54.8) node [anchor=north west][inner sep=0.75pt]  [color={rgb, 255:red, 0; green, 0; blue, 0 }  ,opacity=1 ]  {$H_{1}$};
\draw (81,109.4) node [anchor=north west][inner sep=0.75pt]  [color={rgb, 255:red, 0; green, 0; blue, 0 }  ,opacity=1 ]  {$H_{3}$};
\draw (115,53.4) node [anchor=north west][inner sep=0.75pt]  [color={rgb, 255:red, 0; green, 0; blue, 0 }  ,opacity=1 ]  {$H_{2}$};
\draw (81,158.4) node [anchor=north west][inner sep=0.75pt]  [color={rgb, 255:red, 0; green, 0; blue, 0 }  ,opacity=1 ]  {$H_{4}$};
\draw (139.27,136.8) node [anchor=north west][inner sep=0.75pt]  [font=\footnotesize]  {$\mathcal{Y}( x)$};
\draw (477,129.66) node [anchor=north west][inner sep=0.75pt]  [font=\footnotesize]  {$\widetilde{\mathcal{Y}}( x)$};

\end{tikzpicture}
    \caption{Left: a feasible set $\mathcal{Y}(x)$ defined by four inequality constraints, where $H_i \coloneq  (h_i(x, \cdot))^{-1}(0)$ denotes the boundary of the $i$-th constraint. The feasible set is bounded by four arcs meeting at four corners. Right: the perturbed feasible set $\widetilde{\mathcal{Y}}(x)$ obtained by a small $C^2$ perturbation of the constraints, with perturbed boundaries $\widetilde{H}_i$. Although the constraint boundaries are deformed, the stratification is preserved: the perturbed set still has one interior region, four boundary arcs, and four corner vertices.}
    \label{fig:stablethm}
\end{figure}

Figure~\ref{fig:stablethm} illustrates this stability: although the constraint boundaries may deform under perturbation, the stratification of the feasible set is preserved. We now apply Theorems~\ref{thm:main_licq} and~\ref{thm:stab_strat_perturb} to construct counterexamples showing that the ``every-$x$'' LICQ requirement is non-prevalent. The argument is as follows. We find a problem where the stratification of $\mathcal{Y}(x)$ differs between two upper-level variables $x_1$ and $x_2$. By Theorem~\ref{thm:stab_strat_perturb}, the stratification at each fixed $x$ is preserved under small perturbations. So the mismatch between $x_1$ and $x_2$ persists after any small perturbation. By Theorem~\ref{thm:main_licq}, this means LICQ cannot hold at every $x$ for the perturbed problem.

A common cause for such a mismatch is the presence of an upper-level variable $x$ at which some feasible point has more active constraints than the number of lower-level variables, making linear independence impossible. The following counterexample constructs exactly this scenario.

\begin{counterexample}\label{counterexample:licq2}
Consider the lower-level feasible set in $\mathbb{R}^2$ defined by
\begin{align*}
    h_1(x,y) &\coloneq  y_1 \leq 0, \qquad h_2(x,y) \coloneq  -y_1 - 2 \leq 0, \qquad h_3(x,y) \coloneq y_2 \leq 0, \\
    h_4(x,y) &\coloneq  -y_2 - 2, \qquad h_5(x,y) \coloneq  y_1 + y_2 + x \leq 0,
\end{align*}
where $y = (y_1, y_2) \in \mathbb{R}^2$ and the upper-level variable $x \in [-1, 1]$. There exists $\epsilon > 0$ such that the following holds: for any $C^2$ perturbations $r_i(x,y)$, $i = 1, \ldots, 5$, with
\begin{align*}
    \|r_i\| \leq \epsilon, \quad \|\nabla_y r_i\| \leq \epsilon,\quad \text{and}\quad \|\nabla^2_{yy} r_i\| \leq \epsilon \qquad i = 1, \ldots, 5,
\end{align*}
and the perturbed constraints $\widetilde{h}_i(x, y) \coloneq  h_i(x, y) + r_i(x, y)$, there exist $x^\ast \in [-1, 1]$ and a feasible point $y^\ast \in \mathbb{R}^2$ of the perturbed problem such that LICQ fails at $(x^\ast, y^\ast)$.
 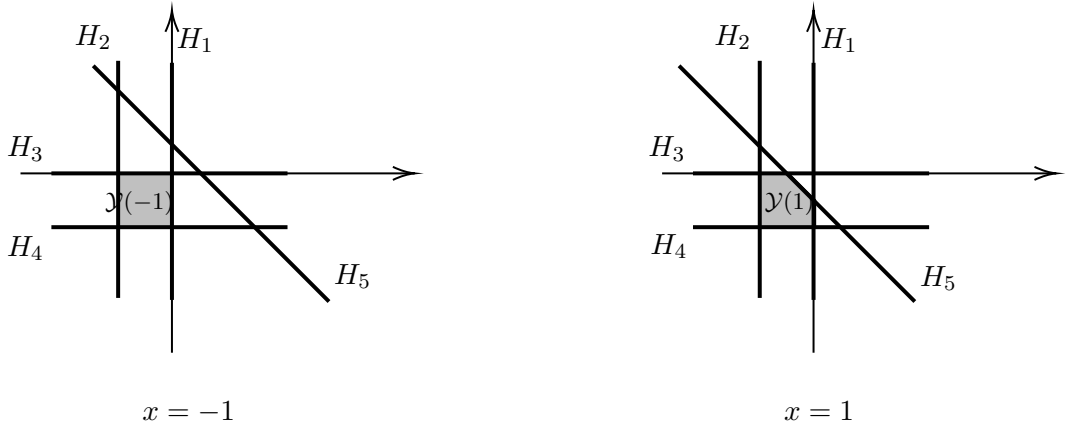
\begin{figure}
        \centering

\tikzset{every picture/.style={line width=0.75pt}} 

\begin{tikzpicture}[x=0.75pt,y=0.75pt,yscale=-1,xscale=1]

\draw    (36.27,134.4) -- (232.09,134.4) ;
\draw [shift={(234.09,134.4)}, rotate = 180] [color={rgb, 255:red, 0; green, 0; blue, 0 }  ][line width=0.75]    (10.93,-3.29) .. controls (6.95,-1.4) and (3.31,-0.3) .. (0,0) .. controls (3.31,0.3) and (6.95,1.4) .. (10.93,3.29)   ;
\draw    (112.27,224.4) -- (112.27,54.4) ;
\draw [shift={(112.27,52.4)}, rotate = 90] [color={rgb, 255:red, 0; green, 0; blue, 0 }  ][line width=0.75]    (10.93,-3.29) .. controls (6.95,-1.4) and (3.31,-0.3) .. (0,0) .. controls (3.31,0.3) and (6.95,1.4) .. (10.93,3.29)   ;
\draw [color={rgb, 255:red, 0; green, 0; blue, 0 }  ,draw opacity=1 ][line width=1.5]    (112.27,78.9) -- (112.27,197.9) ;
\draw [color={rgb, 255:red, 0; green, 0; blue, 0 }  ,draw opacity=1 ][line width=1.5]    (51.77,134.4) -- (170.27,134.4) ;
\draw [color={rgb, 255:red, 0; green, 0; blue, 0 }  ,draw opacity=1 ][line width=1.5]    (72.77,80.4) -- (191.09,198.72) ;
\draw [color={rgb, 255:red, 0; green, 0; blue, 0 }  ,draw opacity=1 ][line width=1.5]    (85.27,77.9) -- (85.27,196.9) ;
\draw [color={rgb, 255:red, 0; green, 0; blue, 0 }  ,draw opacity=1 ][line width=1.5]    (51.77,161.4) -- (170.27,161.4) ;
\draw    (358.27,134.4) -- (554.09,134.4) ;
\draw [shift={(556.09,134.4)}, rotate = 180] [color={rgb, 255:red, 0; green, 0; blue, 0 }  ][line width=0.75]    (10.93,-3.29) .. controls (6.95,-1.4) and (3.31,-0.3) .. (0,0) .. controls (3.31,0.3) and (6.95,1.4) .. (10.93,3.29)   ;
\draw    (434.27,224.4) -- (434.27,54.4) ;
\draw [shift={(434.27,52.4)}, rotate = 90] [color={rgb, 255:red, 0; green, 0; blue, 0 }  ][line width=0.75]    (10.93,-3.29) .. controls (6.95,-1.4) and (3.31,-0.3) .. (0,0) .. controls (3.31,0.3) and (6.95,1.4) .. (10.93,3.29)   ;
\draw [color={rgb, 255:red, 0; green, 0; blue, 0 }  ,draw opacity=1 ][line width=1.5]    (434.27,78.9) -- (434.27,197.9) ;
\draw [color={rgb, 255:red, 0; green, 0; blue, 0 }  ,draw opacity=1 ][line width=1.5]    (373.77,134.4) -- (492.27,134.4) ;
\draw [color={rgb, 255:red, 0; green, 0; blue, 0 }  ,draw opacity=1 ][line width=1.5]    (366.77,80.4) -- (485.09,198.72) ;
\draw [color={rgb, 255:red, 0; green, 0; blue, 0 }  ,draw opacity=1 ][line width=1.5]    (407.27,77.9) -- (407.27,196.9) ;
\draw [color={rgb, 255:red, 0; green, 0; blue, 0 }  ,draw opacity=1 ][line width=1.5]    (373.77,161.4) -- (492.27,161.4) ;
\draw  [color={rgb, 255:red, 0; green, 0; blue, 0 }  ,draw opacity=1 ][fill={rgb, 255:red, 0; green, 0; blue, 0 }  ,fill opacity=0.25 ] (85.27,134.4) -- (112.09,134.4) -- (112.09,161.4) -- (85.27,161.4) -- cycle ;
\draw  [fill={rgb, 255:red, 0; green, 0; blue, 0 }  ,fill opacity=0.25 ] (407.27,134.4) -- (421.66,134.4) -- (435.09,147.83) -- (435.09,161.26) -- (407.27,161.26) -- cycle ;

\draw (113.27,59.8) node [anchor=north west][inner sep=0.75pt]  [color={rgb, 255:red, 0; green, 0; blue, 0 }  ,opacity=1 ]  {$H_{1}$};
\draw (28,114.4) node [anchor=north west][inner sep=0.75pt]  [color={rgb, 255:red, 0; green, 0; blue, 0 }  ,opacity=1 ]  {$H_{3}$};
\draw (192.2,178.23) node [anchor=north west][inner sep=0.75pt]  [color={rgb, 255:red, 0; green, 0; blue, 0 }  ,opacity=1 ]  {$H_{5}$};
\draw (96,247.4) node [anchor=north west][inner sep=0.75pt]    {$x=-1$};
\draw (62,58.4) node [anchor=north west][inner sep=0.75pt]  [color={rgb, 255:red, 0; green, 0; blue, 0 }  ,opacity=1 ]  {$H_{2}$};
\draw (28,164.4) node [anchor=north west][inner sep=0.75pt]  [color={rgb, 255:red, 0; green, 0; blue, 0 }  ,opacity=1 ]  {$H_{4}$};
\draw (436.27,59.8) node [anchor=north west][inner sep=0.75pt]  [color={rgb, 255:red, 0; green, 0; blue, 0 }  ,opacity=1 ]  {$H_{1}$};
\draw (350,114.4) node [anchor=north west][inner sep=0.75pt]  [color={rgb, 255:red, 0; green, 0; blue, 0 }  ,opacity=1 ]  {$H_{3}$};
\draw (486.2,179.23) node [anchor=north west][inner sep=0.75pt]  [color={rgb, 255:red, 0; green, 0; blue, 0 }  ,opacity=1 ]  {$H_{5}$};
\draw (418,247.4) node [anchor=north west][inner sep=0.75pt]    {$x=1$};
\draw (383,58.4) node [anchor=north west][inner sep=0.75pt]  [color={rgb, 255:red, 0; green, 0; blue, 0 }  ,opacity=1 ]  {$H_{2}$};
\draw (351,163.4) node [anchor=north west][inner sep=0.75pt]  [color={rgb, 255:red, 0; green, 0; blue, 0 }  ,opacity=1 ]  {$H_{4}$};
\draw (77.27,141.8) node [anchor=north west][inner sep=0.75pt]  [font=\footnotesize]  {$\mathcal{Y}( -1)$};
\draw (408.27,141.8) node [anchor=north west][inner sep=0.75pt]  [font=\footnotesize]  {$\mathcal{Y}( 1)$};

\end{tikzpicture}
        \caption{The feasible set $\mathcal{Y}(x)$ (shaded) in Counterexample~\ref{counterexample:licq2} at $x = -1$ (left) and $x = 1$ (right), where $H_i \coloneq  (h_i(x, \cdot))^{-1}(0)$ denotes the boundary of the $i$-th constraint. At $x = -1$, $H_5$ does not intersect the rectangle bounded by $H_1, \ldots, H_4$, so the feasible set is that rectangle, with four vertices and four boundary edges. At $x = 1$, $H_5$ cuts off a corner, producing a pentagon with five vertices and five boundary edges. The two feasible sets have different stratifications.}
        \label{fig:licq_counterexample1}
    \end{figure}

    Figure~\ref{fig:licq_counterexample1} shows the feasible set at $x = -1$ and $x = 1$. At $x = -1$, the feasible set $\mathcal{Y}(-1)$ is a square, and the constraint boundaries $H_1$ and $H_3$ intersect at a feasible vertex $(0,0)$. At $x = 1$, the boundary $H_5$ cuts off this corner, making the intersection $(0,0)$ infeasible and producing a pentagon. This topological transition strictly requires that at some intermediate upper-level variable (specifically $x=0$ for the unperturbed problem), the moving boundary $H_5$ must pass exactly through the intersection of $H_1$ and $H_3$. At this critical moment, three constraints are simultaneously active at a single point in $\mathbb{R}^2$, forcing their gradients to be linearly dependent and causing LICQ to fail. Since the stratifications at $x=-1$ and $x=1$ are completely different (four vertices versus five), Theorem~\ref{thm:stab_strat_perturb} guarantees this mismatch persists under any small perturbation, and Theorem~\ref{thm:main_licq} ensures LICQ cannot hold at every $x$ for any nearby problem.
    
\end{counterexample}

\begin{remark}\label{remark:licq_lp_screening}
The failure of LICQ in Counterexample~\ref{counterexample:licq2} is driven by a vertex switching between feasible and infeasible as $x$ varies. This suggests a screening criterion for lower-level problems with affine constraints. Suppose $\mathcal{Y}(x)$ is uniformly bounded and each $h_i(x,\cdot)$ is affine. If $m$ active surfaces $\{y : h_i(x_0, y) = 0\}$ intersect at a single feasible point at some $x_0 \in \mathcal{X}$, then for LICQ to hold at every $x$, the same $m$ surfaces must intersect at a single feasible point for all $x \in \mathcal{X}$. If this fails at any $x$, no small perturbation of the constraints can restore LICQ everywhere.
\end{remark}

\begin{remark}[LICQ at the minimizer only]\label{remark:licq_optimal}
The non-prevalence conclusion remains valid even when LICQ is only required at 
the lower-level minimizer $y^\ast(x)$ rather than at every feasible point. 
Consider Counterexample~\ref{counterexample:licq2} with the lower-level objective 
$g(x,y) = -(y_1+y_2)$, so the minimizer $y^\ast(x)$ lies at the upper-right corner 
of $\mathcal{Y}(x)$. At $x = -1$, the constraint $H_5$ is inactive near this 
corner, and the minimizer $y^\ast(-1) = (0,0)$ has only $h_1$ and $h_3$ active; 
their gradients are linearly independent, so LICQ holds. As $x$ increases from 
$-1$ to $1$, the boundary $H_5$ translates toward the lower-left and eventually 
cuts off this corner, producing the pentagonal feasible set. At the intermediate 
value $x = 0$, $H_5$ passes through $(0,0)$, so $h_1$, $h_3$, and $h_5$ are 
simultaneously active at the minimizer. Three gradients in $\mathbb{R}^2$ cannot 
be linearly independent, so LICQ fails at $y^\ast(0)$. Sufficiently small 
perturbations preserve both the location of the minimizer near this corner and 
the continuous motion of $H_5$ across it, so the failure of LICQ at the minimizer 
persists under perturbation. Hence the ``every $x$'' requirement remains 
non-prevalent.
\end{remark}

In Counterexample~\ref{counterexample:licq2}, the failure of LICQ is caused by having too many active constraints at a single point: three constraints are simultaneously active at a point in $\mathbb{R}^2$, making linear independence of their gradients impossible. The next counterexample illustrates a different mechanism. Here the number of active constraints never exceeds the dimension of the lower-level variable, so the issue is not an excess of active constraints. Instead, LICQ fails because two constraint surfaces become tangent at certain upper-level variable values, causing their gradients to become parallel. As the upper-level variable $x$ varies, the two constraint boundaries transition from transversal intersection to tangency, which changes the number of intersection points. Both mechanisms, however, are manifestations of the same underlying phenomenon: a change in the stratification of the feasible set as $x$ varies.

\begin{counterexample}\label{counterexample:licq3}
Consider the lower-level feasible set in $\mathbb{R}^2$ defined by
\begin{align*}
    h_1(x,y) \coloneq  y_1^2 + y_2^2 - 1 \leq 0, \qquad h_2(x,y) \coloneq  y_1^2 + \left(\frac{y_2}{x}\right)^2 - \frac{1}{2} \leq 0,
\end{align*}
where $y = (y_1, y_2) \in \mathbb{R}^2$ and the upper-level variable $x \in [1, 2]$. There exists $\epsilon > 0$ such that the following holds: for any $C^2$ perturbations $r_i(x,y)$, $i = 1, 2$, with
\begin{align*}
    \|r_i\| \leq \epsilon, \quad \|\nabla_y r_i\| \leq \epsilon,\quad \text{and}\quad \|\nabla^2_{yy} r_i\| \leq \epsilon \qquad i = 1, 2,
\end{align*}
and the perturbed constraints $\widetilde{h}_i(x, y) \coloneq  h_i(x, y) + r_i(x, y)$, there exist $x^\ast \in [1, 2]$ and a feasible point $y^\ast \in \mathbb{R}^2$ of the perturbed problem such that LICQ fails at $(x^\ast, y^\ast)$.

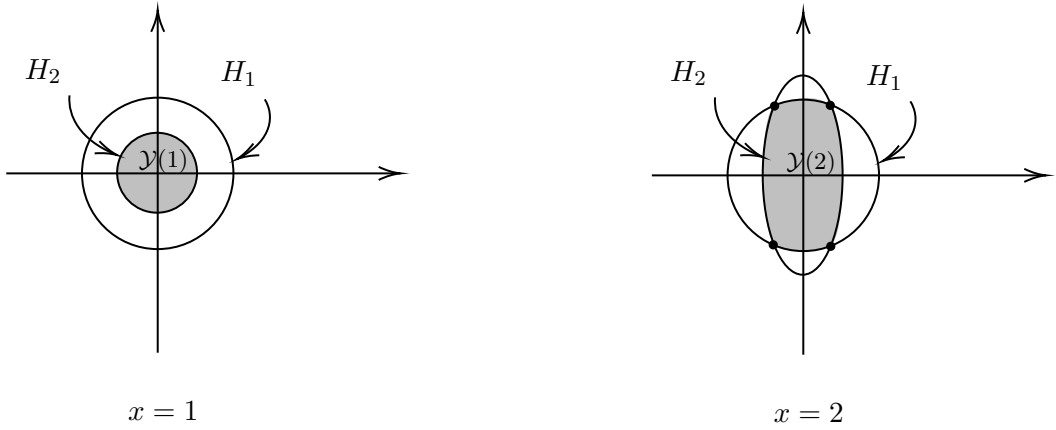
\begin{figure}
    \centering

\tikzset{every picture/.style={line width=0.75pt}} 

\begin{tikzpicture}[x=0.75pt,y=0.75pt,yscale=-1,xscale=1]

\draw    (36.27,134.4) -- (232.09,134.4) ;
\draw [shift={(234.09,134.4)}, rotate = 180] [color={rgb, 255:red, 0; green, 0; blue, 0 }  ][line width=0.75]    (10.93,-3.29) .. controls (6.95,-1.4) and (3.31,-0.3) .. (0,0) .. controls (3.31,0.3) and (6.95,1.4) .. (10.93,3.29)   ;
\draw    (112.27,224.4) -- (112.27,54.4) ;
\draw [shift={(112.27,52.4)}, rotate = 90] [color={rgb, 255:red, 0; green, 0; blue, 0 }  ][line width=0.75]    (10.93,-3.29) .. controls (6.95,-1.4) and (3.31,-0.3) .. (0,0) .. controls (3.31,0.3) and (6.95,1.4) .. (10.93,3.29)   ;
\draw   (74.27,134.4) .. controls (74.27,113.41) and (91.28,96.4) .. (112.27,96.4) .. controls (133.25,96.4) and (150.27,113.41) .. (150.27,134.4) .. controls (150.27,155.39) and (133.25,172.4) .. (112.27,172.4) .. controls (91.28,172.4) and (74.27,155.39) .. (74.27,134.4) -- cycle ;
\draw  [fill={rgb, 255:red, 0; green, 0; blue, 0 }  ,fill opacity=0.25 ] (91.9,134.08) .. controls (91.9,123.01) and (100.88,114.03) .. (111.95,114.03) .. controls (123.02,114.03) and (132,123.01) .. (132,134.08) .. controls (132,145.16) and (123.02,154.13) .. (111.95,154.13) .. controls (100.88,154.13) and (91.9,145.16) .. (91.9,134.08) -- cycle ;
\draw    (166,97.26) .. controls (170.85,105.02) and (170.05,117.48) .. (153.57,126.44) ;
\draw [shift={(152,127.26)}, rotate = 333.43] [color={rgb, 255:red, 0; green, 0; blue, 0 }  ][line width=0.75]    (10.93,-3.29) .. controls (6.95,-1.4) and (3.31,-0.3) .. (0,0) .. controls (3.31,0.3) and (6.95,1.4) .. (10.93,3.29)   ;
\draw    (68,95.26) .. controls (67.04,105.87) and (73.52,117.42) .. (90.15,124.5) ;
\draw [shift={(92,125.26)}, rotate = 201.25] [color={rgb, 255:red, 0; green, 0; blue, 0 }  ][line width=0.75]    (10.93,-3.29) .. controls (6.95,-1.4) and (3.31,-0.3) .. (0,0) .. controls (3.31,0.3) and (6.95,1.4) .. (10.93,3.29)   ;
\draw    (360.27,135.4) -- (556.09,135.4) ;
\draw [shift={(558.09,135.4)}, rotate = 180] [color={rgb, 255:red, 0; green, 0; blue, 0 }  ][line width=0.75]    (10.93,-3.29) .. controls (6.95,-1.4) and (3.31,-0.3) .. (0,0) .. controls (3.31,0.3) and (6.95,1.4) .. (10.93,3.29)   ;
\draw    (436.27,225.4) -- (436.27,55.4) ;
\draw [shift={(436.27,53.4)}, rotate = 90] [color={rgb, 255:red, 0; green, 0; blue, 0 }  ][line width=0.75]    (10.93,-3.29) .. controls (6.95,-1.4) and (3.31,-0.3) .. (0,0) .. controls (3.31,0.3) and (6.95,1.4) .. (10.93,3.29)   ;
\draw   (398.27,135.4) .. controls (398.27,114.41) and (415.28,97.4) .. (436.27,97.4) .. controls (457.25,97.4) and (474.27,114.41) .. (474.27,135.4) .. controls (474.27,156.39) and (457.25,173.4) .. (436.27,173.4) .. controls (415.28,173.4) and (398.27,156.39) .. (398.27,135.4) -- cycle ;
\draw    (490,98.26) .. controls (494.85,106.02) and (494.05,118.48) .. (477.57,127.44) ;
\draw [shift={(476,128.26)}, rotate = 333.43] [color={rgb, 255:red, 0; green, 0; blue, 0 }  ][line width=0.75]    (10.93,-3.29) .. controls (6.95,-1.4) and (3.31,-0.3) .. (0,0) .. controls (3.31,0.3) and (6.95,1.4) .. (10.93,3.29)   ;
\draw    (392,96.26) .. controls (391.04,106.87) and (397.52,118.42) .. (414.15,125.5) ;
\draw [shift={(416,126.26)}, rotate = 201.25] [color={rgb, 255:red, 0; green, 0; blue, 0 }  ][line width=0.75]    (10.93,-3.29) .. controls (6.95,-1.4) and (3.31,-0.3) .. (0,0) .. controls (3.31,0.3) and (6.95,1.4) .. (10.93,3.29)   ;
\draw   (415.9,135.26) .. controls (415.9,107.64) and (424.88,85.26) .. (435.95,85.26) .. controls (447.02,85.26) and (456,107.64) .. (456,135.26) .. controls (456,162.87) and (447.02,185.26) .. (435.95,185.26) .. controls (424.88,185.26) and (415.9,162.87) .. (415.9,135.26) -- cycle ;
\draw  [color={rgb, 255:red, 0; green, 0; blue, 0 }  ,draw opacity=0 ][fill={rgb, 255:red, 0; green, 0; blue, 0 }  ,fill opacity=0.25 ] (422.03,100.59) .. controls (430.83,96.59) and (442.22,94.19) .. (450,100.26) .. controls (457.78,106.32) and (458.42,167.86) .. (449.23,170.99) .. controls (440.05,174.13) and (428.98,174.32) .. (422,170.26) .. controls (415.02,166.19) and (413.23,104.59) .. (422.03,100.59) -- cycle ;
\draw  [fill={rgb, 255:red, 0; green, 0; blue, 0 }  ,fill opacity=1 ] (448,100.26) .. controls (448,99.25) and (448.81,98.44) .. (449.82,98.44) .. controls (450.82,98.44) and (451.63,99.25) .. (451.63,100.26) .. controls (451.63,101.26) and (450.82,102.07) .. (449.82,102.07) .. controls (448.81,102.07) and (448,101.26) .. (448,100.26) -- cycle ;
\draw  [fill={rgb, 255:red, 0; green, 0; blue, 0 }  ,fill opacity=1 ] (420.03,100.59) .. controls (420.03,99.59) and (420.85,98.78) .. (421.85,98.78) .. controls (422.85,98.78) and (423.67,99.59) .. (423.67,100.59) .. controls (423.67,101.6) and (422.85,102.41) .. (421.85,102.41) .. controls (420.85,102.41) and (420.03,101.6) .. (420.03,100.59) -- cycle ;
\draw  [fill={rgb, 255:red, 0; green, 0; blue, 0 }  ,fill opacity=1 ] (419.37,170.26) .. controls (419.37,169.25) and (420.18,168.44) .. (421.18,168.44) .. controls (422.19,168.44) and (423,169.25) .. (423,170.26) .. controls (423,171.26) and (422.19,172.07) .. (421.18,172.07) .. controls (420.18,172.07) and (419.37,171.26) .. (419.37,170.26) -- cycle ;
\draw  [fill={rgb, 255:red, 0; green, 0; blue, 0 }  ,fill opacity=1 ] (448.23,170.99) .. controls (448.23,169.99) and (449.05,169.18) .. (450.05,169.18) .. controls (451.05,169.18) and (451.87,169.99) .. (451.87,170.99) .. controls (451.87,172) and (451.05,172.81) .. (450.05,172.81) .. controls (449.05,172.81) and (448.23,172) .. (448.23,170.99) -- cycle ;

\draw (142.27,76.8) node [anchor=north west][inner sep=0.75pt]  [color={rgb, 255:red, 0; green, 0; blue, 0 }  ,opacity=1 ]  {$H_{1}$};
\draw (96,247.4) node [anchor=north west][inner sep=0.75pt]    {$x=1$};
\draw (44,75.4) node [anchor=north west][inner sep=0.75pt]  [color={rgb, 255:red, 0; green, 0; blue, 0 }  ,opacity=1 ]  {$H_{2}$};
\draw (466.27,78.8) node [anchor=north west][inner sep=0.75pt]  [color={rgb, 255:red, 0; green, 0; blue, 0 }  ,opacity=1 ]  {$H_{1}$};
\draw (420,248.4) node [anchor=north west][inner sep=0.75pt]    {$x=2$};
\draw (368,76.4) node [anchor=north west][inner sep=0.75pt]  [color={rgb, 255:red, 0; green, 0; blue, 0 }  ,opacity=1 ]  {$H_{2}$};
\draw (101.27,120.8) node [anchor=north west][inner sep=0.75pt]  [font=\footnotesize]  {$\mathcal{Y}( 1)$};
\draw (426.27,121.8) node [anchor=north west][inner sep=0.75pt]  [font=\footnotesize]  {$\mathcal{Y}( 2)$};

\end{tikzpicture}
   \caption{The feasible set $\mathcal{Y}(x)$ (shaded) in Counterexample~\ref{counterexample:licq3} at $x = 1$ (left) and $x = 2$ (right), where $H_i \coloneq  (h_i(x, \cdot))^{-1}(0)$ denotes the boundary of the $i$-th constraint. At $x = 1$, the ellipse $H_2$ lies entirely inside the circle $H_1$, so the feasible set is a disk with no vertices. At $x = 2$, $H_2$ protrudes beyond $H_1$, and the two constraint boundaries intersect transversally at four points (marked), producing a feasible set with four vertices and four boundary arcs. The two feasible sets have different stratifications.}
    \label{fig:licq_counterexample2}
\end{figure}

Figure~\ref{fig:licq_counterexample2} shows the feasible set at $x = 1$ and $x = 2$. At $x = 1$, the ellipse $h_2^{-1}(1, 0)$ lies entirely inside the circle $h_1^{-1}(1, 0)$, so the feasible set is the disk bounded by the ellipse. It has no vertices and a single boundary arc. At $x = 2$, the ellipse extends beyond the circle, and the two constraint boundaries intersect transversally at four points (marked). The feasible set now has four vertices and four boundary arcs. The two feasible sets have different stratifications (zero vertices and one boundary arc versus four vertices and four boundary arcs). 
By Theorem~\ref{thm:stab_strat_perturb}, this difference persists under any small perturbation, implying that LICQ cannot hold at every $x$ for any nearby problem.

\end{counterexample}

\subsection{Non-prevalence of SCSC}

Throughout this section and Section~\ref{subsection:sosc}, to simplify the exposition, we treat each active-set manifold $S_{\mathcal{J}}(x) \coloneq  \{y : \mathcal{J}(x,y) = \mathcal{J}\}$ as a single stratum, rather than further refining it into connected components. This convention suffices for our counterexample constructions; the same arguments go through under the standard stratification by connected components.

The rigidity theorem for LICQ (Theorem~\ref{thm:main_licq}) concerns the topology of the feasible set: if LICQ holds at every $x$, the stratification of $\mathcal{Y}(x)$ is constant. We now show that adding SCSC imposes a further rigidity on the behavior of KKT points within this stratification. Specifically, if both LICQ and SCSC hold at every $x$, then no continuous trajectory of KKT points can move from one stratum to another as $x$ varies. In other words, the active constraint pattern at any KKT point is locked in place.

\begin{theorem}[Rigidity of active set]\label{thm:main_scsc}
    Suppose Assumption~\ref{assumption:general2} holds, the upper-level feasible set $\mathcal{X}$ is a connected manifold with generalized boundary (MGB), and the lower-level feasible set $\mathcal{Y}(x)$ is uniformly bounded for all $x \in \mathcal{X}$. Furthermore, assume that, at every $x \in \mathcal{X}$, LICQ and SCSC hold for the lower-level problem. Then, for any continuous trajectory of KKT points $y^\ast(x)$ parameterized by $x$, the active constraint index set remains constant. Consequently, no KKT point can migrate across different strata of $\mathcal{Y}(x)$ as $x$ varies.
\end{theorem}

Theorem~\ref{thm:main_scsc} shows that under the ``every $x$'' SCSC assumption, the active constraint pattern of a KKT point is invariant. The proof idea is as follows: along any continuous KKT trajectory, SCSC ensures that the active set is locally constant. By the connectedness of $\mathcal{X}$, it is globally constant. The full proof is given in Appendix~\ref{proof:thm:main_scsc}.

As in the LICQ case, we now turn to perturbation stability. Unlike the stratification in the LICQ case, the active constraint pattern at a KKT point is not stable under general perturbations, because perturbing the objective may destroy the KKT point itself. Stability can be restored if additional nondegeneracy conditions are imposed (such as the nonsingularity of the KKT matrix in Theorem~\ref{thm:stab_strat_lmins_perturb_noM}). In the setting of strongly convex lower-level with LICQ and SCSC, the lower-level problem has a unique global minimizer $y^\ast(x)$ that persists under small perturbations, and the active constraint pattern at $y^\ast(x)$ is stable, as the following theorem makes precise.

\begin{theorem}[Stability of the active set under perturbation]\label{thm:stab_scsc}
    Fix any \(x\in\mathcal X\). Suppose Assumption~\ref{assumption:general2} holds,
    the lower-level feasible set \(\mathcal Y(x)\) is bounded, and each lower-level constraint function \(h_i(x,\cdot)\), \(i=1,\ldots,k\), is convex. LICQ and SCSC hold for the lower-level problem, and \(g(x,\cdot)\) is strongly convex.
    Let \(y^*(x)\) denote the unique global minimizer and let
    \(\mathcal J(x,y^*(x))\) denote its active constraint index set. For any compact
    sets \(K_{\rm in},K_{\rm out}\subset\mathbb R^m\) such that
    \(\mathcal Y(x)\subset \operatorname{int}(K_{\rm in})\) and
    \(K_{\rm in}\subset \operatorname{int}(K_{\rm out})\), there exists
    \(\epsilon(x,K_{\rm in},K_{\rm out})>0\) such that the following holds: for any
    \(C^2\) perturbations \(s(x,\cdot)\) and \(r_i(x,\cdot)\), as in \eqref{eq:pertubedproblem}, satisfying
    \[
    \|s(x,\cdot)\|_{C^2(K_{\rm out})}
    +
    \max_i\|r_i(x,\cdot)\|_{C^2(K_{\rm out})}
    \le \epsilon(x,K_{\rm in},K_{\rm out}),
    \]
    (see Definition~\ref{def:localc2norm} for the compact-local \(C^2\)-norm), the perturbed problem restricted to \(K_{\rm in}\) has a unique minimizer
    \(\widetilde y^*(x)\), and $\widetilde{\mathcal J}(x,\widetilde y^*(x))=\mathcal J(x,y^*(x))$, where $\widetilde{\mathcal J}(x,\widetilde y^*(x))\coloneq \{i:\widetilde h_i(x,\widetilde y^*(x))=0\}$
    denotes the active constraint index set of the perturbed problem at \(\widetilde y^*(x)\).
\end{theorem}

This theorem guarantees that under small perturbations, the unique solution may shift slightly, but its active constraint pattern remains exactly the same. It is a special case of Theorem~\ref{thm:stab_strat_lmins_perturb_noM}; the full proof is given in Appendix~\ref{proof:thm:stab_scsc}. We construct our counterexample in the setting of strongly convex lower-level with LICQ and SCSC.

\begin{counterexample}\label{counterexample:scsc1}
    Consider the lower-level problem
\begin{align*}
    \min_{y\in\mathbb{R}^2}\ g(x,y)=(y_1-x)^2+y_2^2\quad \operatorname{s.t.}\quad h(x,y)=\|y\|^2\leq 1,
\end{align*}
where $y \in \mathbb{R}^2$ and the upper-level variable $x \in [0, 2]$. There exists $\epsilon > 0$ such that the following holds: for any $C^2$ perturbation $s(x,y)$ and $C^2$ perturbation $r(x,y)$ with
\begin{align*}
    \|s\|, \, \|\nabla_y s\|, \, \|\nabla^2_{yy} s\|, \, \|r\|, \, \|\nabla_y r\|, \, \|\nabla^2_{yy} r\| \leq \epsilon,
\end{align*}
and the perturbed problem
\begin{align*}
    \min_{y\in\mathbb{R}^2}\ \widetilde{g}(x,y)\coloneq (y_1-x)^2+y_2^2+s(x,y)\quad
    \operatorname{s.t.}\quad \widetilde{h}(x,y)\coloneq \|y\|^2-1+r(x,y)\leq 0,
\end{align*}
there exist $x^\ast \in [0, 2]$ and a KKT point $y^\ast \in \mathbb{R}^2$ of the perturbed problem such that SCSC fails at $(x^\ast, y^\ast)$. Since $\epsilon$ is small, the perturbed problem remains strongly convex with a convex feasible set, so it admits a unique global minimizer $y^\ast(x)$ that depends continuously on $x$.
    \begin{figure}
        \centering

\tikzset{every picture/.style={line width=0.75pt}} 

\begin{tikzpicture}[x=0.75pt,y=0.75pt,yscale=-1,xscale=1]

\draw  (28,156.91) -- (219.72,156.91)(123.91,60.28) -- (123.91,252) (212.72,151.91) -- (219.72,156.91) -- (212.72,161.91) (118.91,67.28) -- (123.91,60.28) -- (128.91,67.28)  ;
\draw  [fill={rgb, 255:red, 0; green, 0; blue, 0 }  ,fill opacity=0.25 ] (99.95,156.91) .. controls (99.95,143.67) and (110.68,132.94) .. (123.91,132.94) .. controls (137.15,132.94) and (147.88,143.67) .. (147.88,156.91) .. controls (147.88,170.14) and (137.15,180.87) .. (123.91,180.87) .. controls (110.68,180.87) and (99.95,170.14) .. (99.95,156.91) -- cycle ;
\draw  [color={rgb, 255:red, 0; green, 0; blue, 0 }  ,draw opacity=1 ][fill={rgb, 255:red, 0; green, 0; blue, 0 }  ,fill opacity=1 ] (121.03,156.91) .. controls (121.03,155.31) and (122.32,154.02) .. (123.91,154.02) .. controls (125.51,154.02) and (126.8,155.31) .. (126.8,156.91) .. controls (126.8,158.5) and (125.51,159.8) .. (123.91,159.8) .. controls (122.32,159.8) and (121.03,158.5) .. (121.03,156.91) -- cycle ;
\draw  (231,156.91) -- (422.72,156.91)(326.91,60.28) -- (326.91,252) (415.72,151.91) -- (422.72,156.91) -- (415.72,161.91) (321.91,67.28) -- (326.91,60.28) -- (331.91,67.28)  ;
\draw  [fill={rgb, 255:red, 0; green, 0; blue, 0 }  ,fill opacity=0.25 ] (302.8,156.91) .. controls (302.8,143.59) and (313.59,132.79) .. (326.91,132.79) .. controls (340.24,132.79) and (351.03,143.59) .. (351.03,156.91) .. controls (351.03,170.23) and (340.24,181.03) .. (326.91,181.03) .. controls (313.59,181.03) and (302.8,170.23) .. (302.8,156.91) -- cycle ;
\draw  [color={rgb, 255:red, 0; green, 0; blue, 0 }  ,draw opacity=1 ][fill={rgb, 255:red, 0; green, 0; blue, 0 }  ,fill opacity=1 ] (348.14,156.91) .. controls (348.14,155.31) and (349.44,154.02) .. (351.03,154.02) .. controls (352.63,154.02) and (353.92,155.31) .. (353.92,156.91) .. controls (353.92,158.5) and (352.63,159.8) .. (351.03,159.8) .. controls (349.44,159.8) and (348.14,158.5) .. (348.14,156.91) -- cycle ;

\draw (103,251) node [anchor=north west][inner sep=0.75pt]  [color={rgb, 255:red, 74; green, 74; blue, 74 }  ,opacity=1 ]  {$y^{\ast }( 0)$};
\draw (306,251) node [anchor=north west][inner sep=0.75pt]  [color={rgb, 255:red, 0; green, 0; blue, 0 }  ,opacity=1 ]  {$y^{\ast }( 2)$};

\end{tikzpicture}
        \caption{The feasible set $\mathcal{Y}(x)$ (shaded) and the unique minimizer $y^\ast(x)$ (marked) in Counterexample~\ref{counterexample:scsc1} at $x = 0$ (left) and $x = 2$ (right). At $x = 0$, the minimizer lies in the interior with no active constraint. At $x = 2$, the minimizer lies on the constraint boundary. The minimizer migrates across strata as $x$ varies.}
        \label{fig:scsc_counterexample}
    \end{figure}
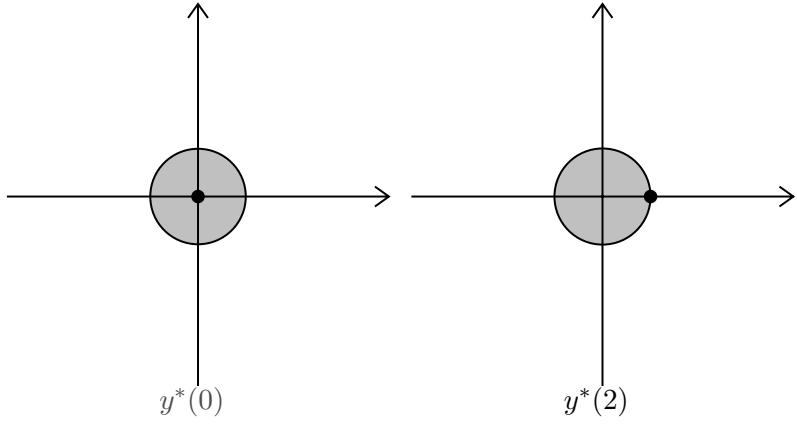

Figure~\ref{fig:scsc_counterexample} shows the feasible set and the unique minimizer at $x = 0$ and $x = 2$. At $x = 0$, the minimizer $y^\ast(0) = (0,0)$ lies in the interior of $\mathcal{Y}(0)$, so no constraint is active. At $x = 2$, the minimizer $y^\ast(2) = (1,0)$ lies on the constraint boundary $\{y : h(x,y) = 0\}$. As $x$ increases from $0$ to $2$, the minimizer moves from the interior to the boundary, crossing from one stratum to another. By Theorem~\ref{thm:stab_scsc}, the active constraint pattern at each fixed $x$ is preserved under small perturbations. By Theorem~\ref{thm:main_scsc}, this stratum crossing is incompatible with SCSC holding at every $x$.

\end{counterexample}

\subsection{Non-prevalence of uniform SOSC}\label{subsection:sosc}

The rigidity theorems for LICQ and SCSC concern the topology of the feasible set and the active set along KKT trajectories, respectively. We now turn to the third combination: LICQ, SCSC, and uniform SOSC, the standard setting for non-strongly-convex lower-level problems. Adding uniform SOSC yields a further rigidity: if LICQ, SCSC, and uniform SOSC all hold at every $x$, then the number of local minimizers on each stratum remains constant as $x$ varies over $\mathcal{X}$. Before stating the rigidity theorem, we recall the natural stratification of $\mathcal Y(x)$ indexed by active sets. For any $y \in \mathcal Y(x)$, let $\mathcal{J}(x,y) \coloneq  \{j \in \{1, \ldots, k\} : h_j(x, y) = 0\}$ denote its active index set, and define
\[
\mathcal S_\mathcal{J}(x) \coloneq  \{y \in \mathcal Y(x) : \mathcal{J}(x,y) = \mathcal{J}\}, \qquad \mathcal{J} \subseteq \{1, \ldots, k\}.
\]
The collection $\{\mathcal S_\mathcal{J}(x)\}_\mathcal{J}$ partitions $\mathcal Y(x)$ into strata.

\begin{theorem}[Rigidity of stratum-wise minimizer count]\label{thm:main_sosc}
    Suppose Assumption~\ref{assumption:general2} holds, the upper-level feasible set
    $\mathcal{X}$ is a connected manifold with generalized boundary (MGB), and the
    lower-level feasible set $\mathcal{Y}(x)$ is uniformly bounded for all
    $x \in \mathcal{X}$. Assume that, for every $x \in \mathcal{X}$, LICQ, SCSC,
    and uniform SOSC hold for the lower-level problem, and the lower-level problem
    admits finitely many local minimizers. Then, for every
    $\mathcal J\subseteq\{1,\ldots,k\}$ and any $x_1,x_2\in\mathcal X$, we have
    \begin{align*}
    &\#\bigl\{\text{local minimizers of the lower-level problem at }x_1
    \text{ in }\mathcal S_{\mathcal J}(x_1)\bigr\} \\
    =&\#\bigl\{\text{local minimizers of the lower-level problem at }x_2
    \text{ in }\mathcal S_{\mathcal J}(x_2)\bigr\}.
    \end{align*}
\end{theorem}

The proof proceeds as follows. First, LICQ and uniform SOSC imply the
nonsingularity of the KKT Jacobian, so the Implicit Function Theorem gives a
locally unique \(C^1\) continuation of each local minimizer. SCSC keeps the
active set, and hence the stratum, fixed along this continuation. This yields
lower semicontinuity of the number of local minimizers on each stratum, denoted
\(N_{\mathcal J}(x)\). Upper semicontinuity follows from the closedness of the
graph of local minimizers (Lemma~\ref{lem:closedness_minimizers}). Therefore,
\(N_{\mathcal J}(x)\) is a continuous integer-valued function on the connected
set \(\mathcal X\), and hence is constant. The full proof is given in
Appendix~\ref{proof:thm:main_sosc}.

Theorem~\ref{thm:main_sosc} tells us that if LICQ, SCSC, and uniform SOSC hold at every upper-level variable $x$, and the total number of local minimizers is finite, then the number of local minimizers on each stratum is constant across $\mathcal{X}$. To build counterexamples, we again need a perturbation stability result ensuring that this count cannot be changed by small perturbations at a fixed $x$. However, uniform SOSC alone is not sufficient for this purpose. Uniform SOSC guarantees nondegeneracy only at local minimizers, but says nothing about other KKT points such as saddle points or local maximizers on the feasible set. When the objective is perturbed, a degenerate non-minimizing KKT point can undergo a bifurcation, splitting into multiple KKT points, some of which may be new local minimizers. This would change the stratum-wise count of local minimizers even though the original local minimizers themselves persist; see \citet[Example~4.1]{jiang2025correspondence} for a concrete illustration. To rule out this phenomenon, we require a stronger condition: the KKT matrix must be nonsingular at every KKT point, not just at local minimizers. This nondegeneracy guarantees that the stratum-wise count of local minimizers is preserved under small perturbations. We will specifically construct our counterexample to satisfy this condition at the chosen upper-level variables $x_1$ and $x_2$. 

\begin{theorem}[Stability of stratum-wise minimizer count under perturbation]\label{thm:stab_strat_lmins_perturb_noM}
Fix any \(x\in\mathcal X\). Suppose Assumption~\ref{assumption:general2} holds, the
lower-level feasible set \(\mathcal Y(x)\) is bounded, LICQ and SCSC hold for the
lower-level problem, and the KKT matrix
\[
    \begin{pmatrix}
        \nabla^2_{yy}L(x,y,\lambda) & \nabla_y h_{\mathcal{J}(x,y)}(x,y)^\top\\
        \nabla_y h_{\mathcal{J}(x,y)}(x,y) & 0
    \end{pmatrix}
\]
is nonsingular at every KKT pair \((y,\lambda)\). For any compact sets
\(K_{\rm in},K_{\rm out}\subset\mathbb R^m\) satisfying $
\mathcal Y(x)\subset \operatorname{int}(K_{\rm in})$ and $ 
K_{\rm in}\subset \operatorname{int}(K_{\rm out}),$
there exists a constant \(\epsilon(x,K_{\rm in},K_{\rm out})>0\) such that the
following holds: for any \(C^2\) perturbations of the lower-level objective and
constraint functions satisfying
\[
\|s(x,\cdot)\|_{C^2(K_{\rm out})}
+
\max_i\|r_i(x,\cdot)\|_{C^2(K_{\rm out})}
\le \epsilon(x,K_{\rm in},K_{\rm out}),
\]
(see Definition~\ref{def:localc2norm} for the compact-local \(C^2\)-norm), letting \(\widetilde{\mathcal S}_{\mathcal J}(x)\) denote the strata of the
perturbed feasible set, for every \(\mathcal J\subseteq\{1,\dots,k\}\), we have
\begin{align*} &\#\bigl\{\text{local minimizers of the original problem in }\mathcal S_{\mathcal J}(x)\bigr\}\\ =& \#\bigl\{\text{local minimizers of the perturbed problem in }\widetilde{\mathcal S}_{\mathcal J}(x)\cap K_{\rm in}\bigr\}. \end{align*}
\end{theorem}

The proof has five steps. First, the boundedness of $\mathcal{Y}(x)$ together with the nonsingularity of the KKT matrix ensures that every KKT point is isolated, so the total number of KKT points is finite. Second, a quantitative fixed-point argument shows that each original KKT point admits a unique nearby KKT point of the perturbed problem. Third, we rule out new KKT points away from the original ones by establishing a uniform positive lower bound on the KKT residual outside their neighborhoods. Fourth, the nonsingularity of the KKT matrix together with the second-order necessary condition prevents a non-minimizing KKT point from becoming a local minimizer after perturbation. Finally, SCSC ensures that each perturbed KKT point lies on the same stratum as the original, and SOSC at each original local minimizer is preserved under small perturbations, so it remains a local minimizer. The full proof is given in Appendix~\ref{proof:thm:stab_strat_lmins_perturb_noM}.

We now apply Theorems~\ref{thm:main_sosc} and~\ref{thm:stab_strat_lmins_perturb_noM} to construct a counterexample showing that the ``every-$x$'' uniform SOSC requirement is non-prevalent. The idea is the same as before. We find a problem where the number of local minimizers on a particular stratum differs between two upper-level variables $x_1$ and $x_2$. By Theorem~\ref{thm:stab_strat_lmins_perturb_noM}, this count is preserved under small perturbations at each fixed $x$. So the mismatch persists after any small perturbation. By Theorem~\ref{thm:main_sosc}, this means uniform SOSC cannot hold at every $x$ for the perturbed problem.

\begin{counterexample}\label{counterexample:sosc}
    Consider the lower-level problem
    \begin{align*}
        &\min_{y\in\mathbb{R}^2} \quad g(x,y)=(y_1+1)^2 + 0.1\,x\,y_2^4 + (1-6x)\,y_2^2\\
        &\operatorname{s.t.}\quad h_1(x,y)=-y_1\leq 0,\ h_2(x,y)=y_1-2\leq 0,\\
        &\quad \quad \ \ h_3(x,y)=y_2-6\leq 0,\ h_4(x,y)=-y_2-6\leq 0
    \end{align*}
    where $y = (y_1, y_2) \in \mathbb{R}^2$ and the upper-level variable $x \in [0, 1]$. There exists $\epsilon > 0$ such that for any $C^2$ perturbations $s(x,y)$ and $r_i(x,y)$ with
    \begin{align*}
        \|s\|, \, \|\nabla_y s\|, \, \|\nabla^2_{yy} s\|, \, \|r_i\|, \, \|\nabla_y r_i\|, \, \|\nabla^2_{yy} r_i\| \leq \epsilon,\qquad i=1,2,3,4
    \end{align*}
    the perturbed problem
    \begin{align*}
        \min_{y\in\mathbb{R}^2}\ \widetilde{g}(x,y)\coloneq g(x,y)+s(x,y)\quad
        \operatorname{s.t.}\quad \widetilde{h}_i(x,y)\coloneq h_i(x,y)+r_i(x,y)\leq 0,\qquad i=1,2,3,4
    \end{align*}
    fails to satisfy uniform SOSC over $x \in [0,1]$.
    
    \begin{figure}
        \centering
        \includegraphics[width=1\linewidth]{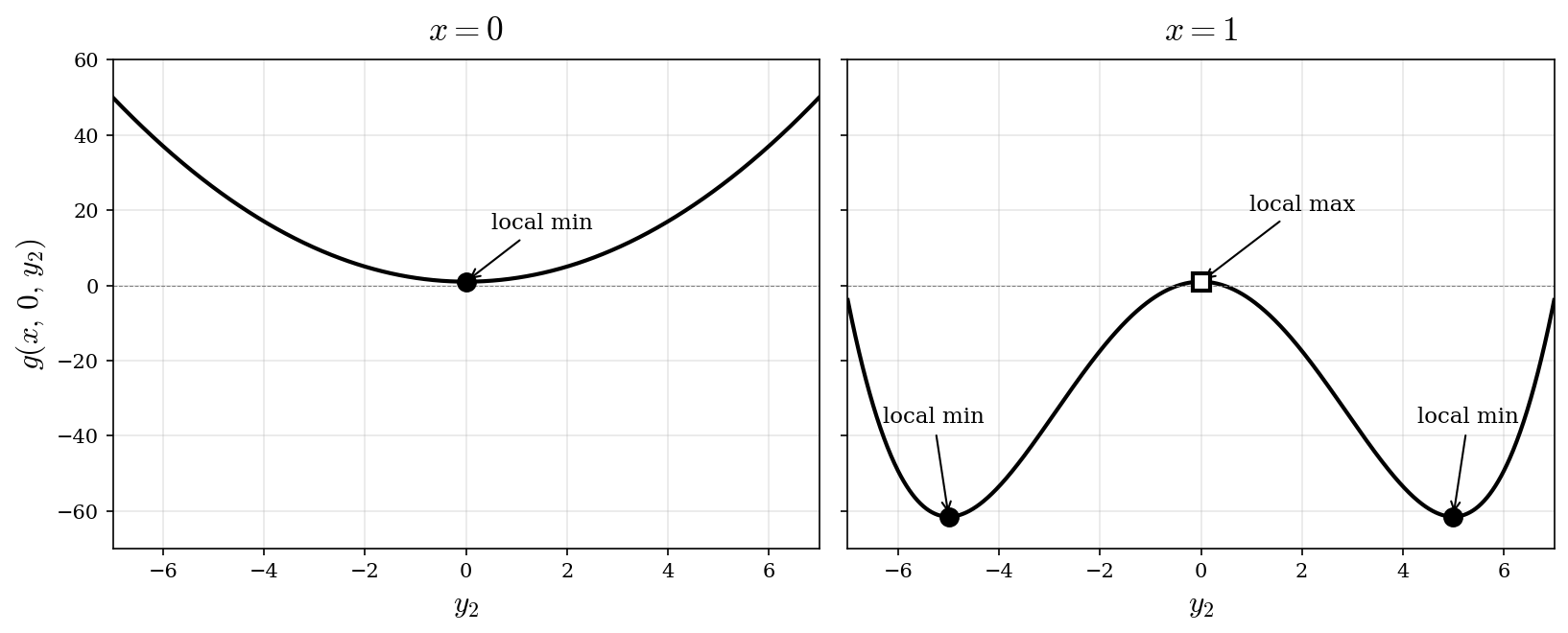}
        \caption{The objective function $g(x, y)$ restricted to the constraint boundary $\{y_1 = 0\}$. At $x = 0$ (left), there is a unique local minimizer at $y_2 = 0$ (filled circle). At $x = 1$ (right), the point $y_2 = 0$ has become a local maximizer (open square), and two new local minimizers have appeared at $y_2 = \pm 5$ (filled circles). The number of local minimizers on the constraint boundary stratum changes from one to two as $x$ varies.}
        \label{fig:sosc_counterexample}
    \end{figure}
    The constraints $h_1$, $h_2$, $h_3$, $h_4$ define the box $[0,2]\times[-6,6]$, 
    so LICQ holds at every feasible point and is preserved under any sufficiently small 
    $C^2$ perturbation. We next verify that SCSC is also preserved under any sufficiently small 
    $C^2$ perturbation. On the feasible set, 
    $\frac{\partial g}{\partial y_1} = 2(y_1+1) \ge 2$, so the KKT stationarity condition 
    $\frac{\partial g}{\partial y_1} - \lambda_1 + \lambda_2 = 0$ together with 
    $\lambda_1, \lambda_2 \ge 0$ forces $\lambda_1 \ge 2 > 0$. Hence every KKT point lies 
    on $y_1=0$, with $h_1$ active and $h_2$ inactive. Similarly, 
    $\frac{\partial g}{\partial y_2}\big|_{y_2=\pm 6} = \pm(14.4x+12) \ne 0$ for $x \in [0,1]$, 
    so $h_3$ and $h_4$ cannot be active at a KKT point without forcing a negative 
    multiplier. Therefore, at every KKT point, $h_1$ is the only active constraint and 
    $\lambda_1 \ge 2 > 0$, so SCSC holds. For any sufficiently small $C^2$ perturbation, 
    the same argument holds, so LICQ and SCSC are preserved.
    
    Figure~\ref{fig:sosc_counterexample} shows the lower-level objective restricted to the feasible set at $x = 0$ and $x = 1$. At $x = 0$, the problem has a unique local minimizer on the constraint boundary defined by $y_1=0$. At $x = 1$, the quartic term creates two additional local minimizers on the boundary, changing the stratum-wise count. By Theorem~\ref{thm:stab_strat_lmins_perturb_noM}, this difference in the number of local minimizers on the boundary stratum persists under any small perturbation. By Theorem~\ref{thm:main_sosc}, uniform SOSC cannot hold at every $x$ for any nearby problem.
\end{counterexample}

\begin{remark}[Uncoupled constraints]
Since LICQ depends only on the constraints, when the constraints are uncoupled, i.e., when $h(x,y) = h(y)$ does not depend on $x$, the ``every-$x$'' version of LICQ is prevalent by Section~\ref{subsection:pre_licq}. In contrast, the ``every-$x$'' version of SCSC and uniform SOSC remains non-prevalent even when the constraints are uncoupled: indeed, Counterexamples~\ref{counterexample:scsc1} and~\ref{counterexample:sosc} both use constraints that do not depend on $x$, yet no sufficiently small $C^2$ perturbation can make SCSC or uniform SOSC hold at every $x$.
\end{remark}

In summary, Statement~\ref{stmt:every_x} is non-prevalent: for each condition, there exist robust counterexamples that no small perturbation can fix. Having established the non-prevalence of Statement~\ref{stmt:every_x}, we now turn to its weaker counterpart Statement~\ref{stmt:almost_every_x} and show in Section~\ref{section:prevalence} that it is prevalent.

\section{Prevalence of LICQ, SCSC, and SOSC: ``Almost-Every-$x$'' Case}\label{section:prevalence}

In Section~\ref{sec:nonprevalence}, we established that Statement~\ref{stmt:every_x} is non-prevalent: the ``every $x$'' requirement cannot be enforced by any sufficiently small $C^2$ perturbation. We now turn to its weaker counterpart, Statement~\ref{stmt:almost_every_x}, and show that it is prevalent. To establish the prevalence of Statement~\ref{stmt:almost_every_x}, we fix a small positive constant $\nu > 0$ and draw $a = (a_1,\ldots,a_k)$ uniformly from $[0,\nu]^k$ and $b = (b_1,\ldots,b_m)$ uniformly from $[0,\nu]^m$. We then apply the following perturbation to the objective and constraint functions of the lower-level problem:
\begin{align}\label{eq:perturbation}
\widetilde{g}(x,y) &= g(x,y) + \langle y, b \rangle, \\
\widetilde{h}_i(x,y) &= h_i(x,y) + a_i. \nonumber
\end{align}
We impose the following assumption for the prevalence results in this section.

\begin{assumption}\label{assumption:general} The following conditions hold.
    \begin{enumerate} 
        \item The upper-level set \(\mathcal X\subset\mathbb R^n\) is Lebesgue measurable. Moreover, the functions \(g\) and \(h_i\) are smooth (i.e., \(C^\infty\)) on an open neighborhood of \(\mathcal X\times\mathbb R^m\);\label{assumption:general1}
        \item The lower-level feasible set $\mathcal{Y}(x)\coloneq \{y:h_i(x,y)\leq0,\ i=1,...,k\}$ satisfies a uniform Slater condition. Specifically, there exists a constant $\rho>0$ such that, for any $x\in\mathcal{X}$, there exists a point $y$ in the interior of $\mathcal{Y}(x)$ satisfying $h_i(x,y)\leq-\rho$ for any $i=1,...,k$.\label{assumption:general12}
    \end{enumerate}
\end{assumption}
The first assumption ensures that the measure-theoretic statements in this
section are well defined. The smoothness on an open neighborhood of 
$\mathcal{X}$ is needed because Sard-type arguments require an open domain 
to define critical points and critical values; the resulting measure-zero 
conclusions on this neighborhood then restrict to measure-zero subsets of 
$\mathcal{X}$. We assume \(C^\infty\) smoothness for simplicity, although
finite-order differentiability would suffice in principle. The second
assumption guarantees that the lower-level feasible set remains nonempty under
the constant perturbations in \eqref{eq:perturbation}: choosing
\(\nu<\rho\) ensures \(\widetilde{\mathcal Y}(x)\neq\emptyset\) for every
\(x\in\mathcal X\). Both assumptions are mild.  We now state the prevalence results for each of the three regularity conditions in turn. 

Unlike Section 2, where the non-prevalence of SCSC is analyzed under the assumption that LICQ holds, and the non-prevalence of SOSC is analyzed under the assumption that both LICQ and SCSC hold, here we establish the prevalence of LICQ, SCSC, and SOSC separately. This suffices because finite intersections of prevalent sets remain prevalent, so the three results together imply that, with probability one, a random perturbation makes LICQ, SCSC, and SOSC hold simultaneously at almost every $x \in \mathcal{X}$.

\begin{remark}
The perturbation scheme~\eqref{eq:perturbation} restricts the parameters
\((b,a)\) to the small box \([0,\nu]^m\times[0,\nu]^k\). This formulation is
slightly more restrictive than Definition~\ref{def:prevalence-formal}, which
requires full Lebesgue measure on the entire finite-dimensional probe space
\[
P=\{(b^\top y,a_1,\ldots,a_k):(b,a)\in\mathbb R^m\times\mathbb R^k\}.
\]
The proofs of Theorems~\ref{thm:LICQ}--\ref{thm:SOSC} rely on Sard's theorem
and the parametric transversality theorem, and the exceptional parameter set
has measure zero on every bounded box in \(P\), regardless of its size.
Taking a countable union over boxes yields measure
zero on the whole probe space \(P\), and hence gives prevalence in the sense
of Definition~\ref{def:prevalence-formal}. We state the results in the
bounded form only because it reflects the setting of interest: in
applications, one only wants arbitrarily small perturbations of the
lower-level objective and constraints, and the bound \(\nu\) also keeps the
perturbed feasible set \(\widetilde{\mathcal Y}(x)\) nonempty under
Assumption~\ref{assumption:general}(\ref{assumption:general12}).
\end{remark}

\subsection{Prevalence of LICQ}\label{subsection:pre_licq}

We begin with LICQ. Since LICQ depends only on the constraint functions and not on the objective, it suffices to only perturb the constraints.

\begin{theorem}[Prevalence of ``almost-every-$x$'' LICQ]\label{thm:LICQ}
Suppose Assumption~\ref{assumption:general} holds. For any fixed positive constant $\nu<\rho$, where $\rho$ is from Assumption~\ref{assumption:general}(\ref{assumption:general12}), draw $a=(a_1,...,a_k)$ uniformly from $[0,\nu]^k$, and perturb $h_i(x,y)$ as follows:
\begin{align}\label{eq:perturb}
    \widetilde{h}_i(x,y)=h_i(x,y)+a_i,\quad\forall i.
\end{align}
Then, with probability one (i.e., almost surely with respect to the random choice of $a$), the set of $x\in\mathcal{X}$, at which the linear independence constraint qualification (LICQ) fails for the perturbed constraints $\{\widetilde{h}_i(x,y)\}_{i=1}^k$ at some feasible $y$, has Lebesgue measure zero in $\mathcal{X}$.
\end{theorem}

The key idea is to reformulate LICQ failure as a critical value condition: the gradients $\{\nabla_y \widetilde{h}_i(x,y)\}_{i \in \mathcal{J}}$ are linearly dependent at a feasible point if and only if $(x, -a_\mathcal{J})$ is a critical value of the map $(x,y) \mapsto (x, h_{\mathcal J}(x,y))$. By Sard's Theorem, the set of such critical values has measure zero. Fubini's Theorem then transfers this to the $x$-space: for almost every $a$, the set of $x$ at which LICQ fails has measure zero. The full proof is given in Appendix~\ref{proof:thm:LICQ}. 

The following example illustrates Theorem~\ref{thm:LICQ}: a problem whose LICQ failure set has positive measure before perturbation is reduced to a single point after perturbation.

\begin{example}\label{ex:licq-prevalence}
Consider the lower-level feasible set in $\mathbb{R}$ defined by
\[
h_1(x,y) = y,\quad h_2(x,y) = y - \phi(x),\quad h_3(x,y) = -y - 1,
\]
where $y \in \mathbb{R}$, $x \in [-\tfrac{1}{2}, \tfrac{3}{2}]$, and
\[
\phi(x) = \begin{cases} x^3 & x \le 0,\\ 0 & 0 \le x \le 1,\\ (x-1)^3 & x \ge 1. \end{cases}
\]
The function $\phi$ is $C^2$ with $\phi \equiv 0$ on $[0,1]$, $\phi(x) < 0$ for $x \in [-\tfrac{1}{2}, 0)$, and $\phi(x) > 0$ for $x \in (1, \tfrac{3}{2}]$; as discussed under Assumption~\ref{assumption:general2}, $C^2$ suffices here. For $x \in [0,1]$, both $h_1$ and $h_2$ are active at $y = 0$ with $\nabla_y h_1 = \nabla_y h_2 = 1$, so LICQ fails on the entire interval $[0,1]$.

Now consider the constant perturbation $\widetilde{h}_i(x,y) = h_i(x,y) + a_i$ as in \eqref{eq:perturb} (see Figure~\ref{fig:licq_almost}). Since $m = 1$, LICQ fails at a feasible point if and only if at least two perturbed constraints are simultaneously active at that point. We analyze each pair in turn. The pair $(\widetilde{h}_1, \widetilde{h}_2)$: simultaneous activity at $y = -a_1$ requires $\phi(x) = a_2 - a_1$. Since $\phi \equiv 0$ on $[0,1]$, this cannot hold for any $x \in [0,1]$ when $a_1 \neq a_2$. If $a_1 > a_2$, the equation $x^3 = a_2 - a_1$ has exactly one solution $x = (a_2 - a_1)^{1/3} \in [-\tfrac{1}{2}, 0)$ and no solution on $[1, \tfrac{3}{2}]$ since $\phi > 0$ there. If $a_2 > a_1$, the equation $(x-1)^3 = a_2 - a_1$ has exactly one solution $x = 1 + (a_2 - a_1)^{1/3} \in (1, \tfrac{3}{2}]$ and no solution on $[-\tfrac{1}{2}, 0]$ since $\phi < 0$ there. The pair $(\widetilde{h}_1, \widetilde{h}_3)$: simultaneous activity at $y = -a_1 = -1 + a_3$ requires $a_1 + a_3 = 1$, which fails when $\|a\|$ is sufficiently small. The pair $(\widetilde{h}_2, \widetilde{h}_3)$: simultaneous activity at $y = \phi(x) - a_2 = -1 + a_3$ requires $\phi(x) = 1 - a_2 - a_3$, which has no solution when $\|a\|$ is sufficiently small.

Therefore, except when $a_1 = a_2$ (a measure-zero event in $[0,\nu]^3$), the LICQ failure set reduces from the interval $[0,1]$ to a single point, consistent with Theorem~\ref{thm:LICQ}.
\begin{figure}
    \centering
    \includegraphics[width=1\linewidth]{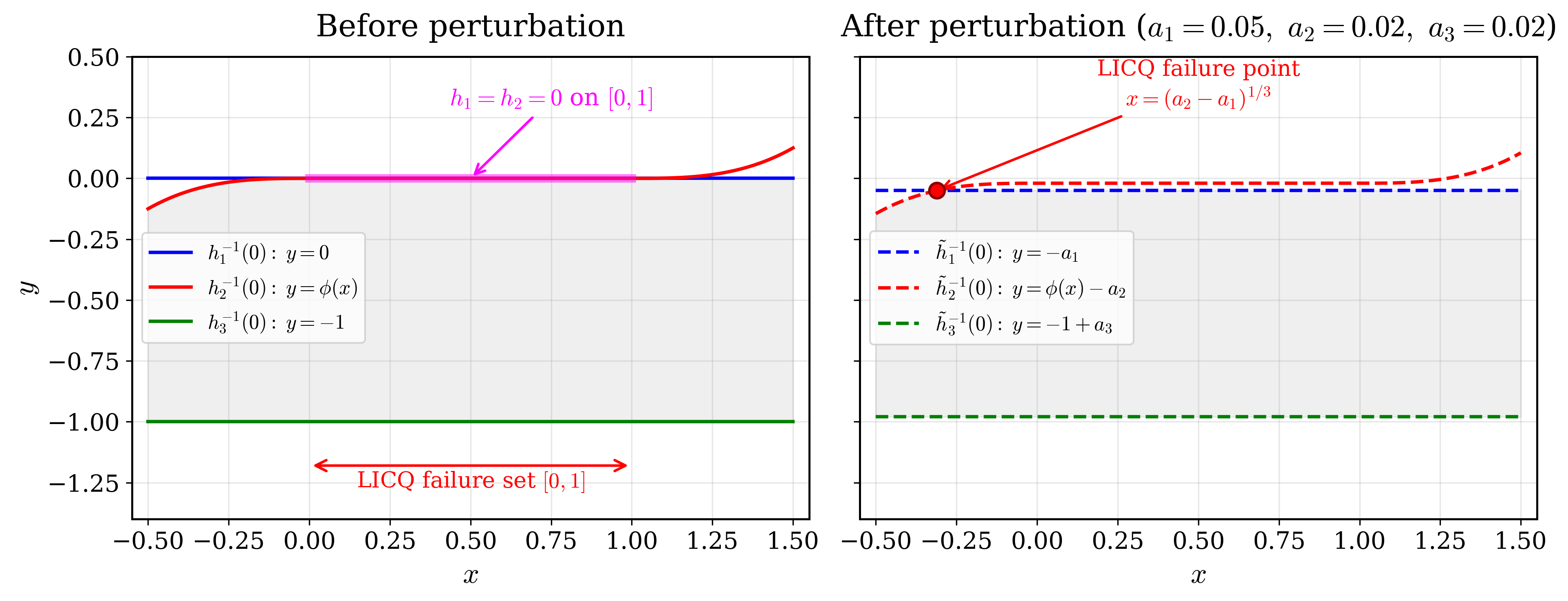}
    \caption{The constraint boundaries and feasible set $\mathcal{Y}(x)$ (shaded) in Example~\ref{ex:licq-prevalence}. Left: before perturbation, $h_1$ and $h_2$ coincide on $[0,1]$, so LICQ fails on the entire interval. Right: after perturbation with $a_1 \neq a_2$, the two constraint boundaries separate, and the LICQ failure set reduces to a single point.}
    \label{fig:licq_almost}
\end{figure}
\end{example}

\subsection{Prevalence of SCSC}

Next, we analyze SCSC for the lower-level problem. Unlike LICQ, SCSC depends on both the constraints and the behavior of the objective function. Therefore, we introduce simultaneous perturbations to the constraint functions and to the objective, as in \eqref{eq:perturb2}. The following theorem establishes that, for almost every perturbation $a$ and $b$, the strict complementarity holds for the lower-level at almost every upper-level variable $x$.

\begin{theorem}[Prevalence of ``almost-every-$x$'' SCSC]\label{thm:SCSC}
    Suppose Assumption~\ref{assumption:general} holds. For any fixed positive constant $\nu<\rho$, where $\rho$ is from Assumption~\ref{assumption:general}(\ref{assumption:general12}), draw $b=(b_1,...,b_m)$ uniformly from $[0,\nu]^m$, $a=(a_1,...,a_k)$ uniformly from $[0,\nu]^k$, and perturb $g(x,y)$ and $h_i(x,y)$ as follows:
    \begin{align}\label{eq:perturb2}
        \widetilde{g}(x,y)&=g(x,y)+\left<y,b\right>,\\ 
        \widetilde{h}_i(x,y)&=h_i(x,y)+a_i,\quad\forall i.\nonumber
    \end{align}
    Then, with probability one (i.e., almost surely with respect to the random choice of $b$ and $a$), the set of $x\in\mathcal{X}$, at which the strict complementary slackness condition (SCSC) does not hold at some KKT point $y$ of the perturbed problem, has Lebesgue measure zero in $\mathcal{X}$.
\end{theorem}

The proof uses the Parametric Transversality Theorem~\ref{thm:param_transversality}. For each active set $\mathcal{J}$ and each index $i \in \mathcal{J}$, we construct a system $H_{\mathcal{J},i}(z; x, b, a) = 0$ that encodes the existence of a KKT point violating SCSC (i.e., $\lambda_i = 0$ for some active constraint $i$). We show that the full map $H_{\mathcal{J},i}(\cdot\,; \cdot, \cdot, \cdot)$ is transverse to $\{0\}$. By the Parametric Transversality Theorem, for almost every $(x, b, a)$, the slice map $H_{\mathcal{J},i}(\cdot\,; x, b, a)$ is also transverse to $\{0\}$. Since this slice map has more equations than unknowns, transversality forces the solution set $(H_{\mathcal{J},i}(\cdot\,; x, b, a))^{-1}(0)$ to be empty. The full proof is given in Appendix~\ref{proof:thm:SCSC}.

The following example illustrates Theorem~\ref{thm:SCSC}: a problem whose SCSC failure set has positive measure before perturbation, while its perturbed version has a single point.

\begin{example}\label{ex:scsc-prevalence}
Consider the lower-level problem in $\mathbb{R}$ defined by
\[
\min_{y \in \mathbb{R}} \; g(x,y) = (y - \phi(x))^2 \quad \operatorname{s.t.} \quad h_1(x,y) = y \leq 0, \quad h_2(x,y) = -y - 1 \leq 0,
\]
where $x \in [-\tfrac{1}{2}, \tfrac{3}{2}]$ and $\phi$ is defined as in Example~\ref{ex:licq-prevalence}. The feasible set $\mathcal{Y}(x) = [-1, 0]$ is independent of $x$, so LICQ holds for every $x$. The KKT system reads
\[
2(y - \phi(x)) + \mu_1 - \mu_2 = 0, \quad \mu_1 y = 0, \quad \mu_2(-y-1) = 0, \quad \mu_1, \mu_2 \geq 0,
\]
with $y \in [-1, 0]$. The unique minimizer is $y^*(x)=\min\{0,\max(\phi(x),-1)\}$, with multipliers
\[
\mu_1(x) = \max(2\phi(x), 0), \quad \mu_2(x) = \max(-2\phi(x) - 2, 0).
\]
SCSC fails precisely when the unconstrained minimizer $\phi(x)$ lies on the boundary of $\mathcal{Y}(x)$, so that the corresponding multiplier vanishes while the constraint is active. Since $\phi \equiv 0$ on $[0,1]$, this occurs for every $x \in [0,1]$, so the SCSC failure set has positive measure. We choose the perturbation radius \(\nu<1/8\). Now consider the perturbation in \eqref{eq:perturb2} (see Figure~\ref{fig:scsc_almost}). The unconstrained minimizer of $\widetilde{g}(x,y) = (y - \phi(x))^2 + by$ is $y = \phi(x) - \tfrac{b}{2}$. SCSC fails when this value coincides with a perturbed boundary. The perturbed feasible set is $\widetilde{Y}(x) = [-1+a_2,\, -a_1]$, so there are two boundary values to consider. For the upper boundary $y = -a_1$: the unconstrained minimizer hits this boundary when $\phi(x) - \tfrac{b}{2} = -a_1$, i.e., $\phi(x) = \tfrac{b}{2} - a_1$. On the interval $[0,1]$, we have $\phi \equiv 0$, so this condition becomes $a_1 = \tfrac{b}{2}$. Since $a_1$ and $b$ are drawn independently and uniformly, the event $a_1 = \tfrac{b}{2}$ has probability zero. When $a_1 \neq \tfrac{b}{2}$, the original failure interval $[0,1]$ is entirely eliminated. On $[-\tfrac{1}{2}, 0)$, the condition becomes $x^3 = \tfrac{b}{2} - a_1$, which has exactly one solution $x = (\tfrac{b}{2} - a_1)^{1/3}$ when $a_1 > \tfrac{b}{2}$, and no solution when $a_1 < \tfrac{b}{2}$ since $x^3 < 0$ there. On $(1, \tfrac{3}{2}]$, the condition becomes $(x-1)^3 = \tfrac{b}{2} - a_1$, which has exactly one solution $x = 1 + (\tfrac{b}{2} - a_1)^{1/3}$ when $a_1 < \tfrac{b}{2}$, and no solution when $a_1 > \tfrac{b}{2}$ since $(x-1)^3 > 0$ there. For the lower boundary $y = -1 + a_2$: the unconstrained minimizer hits this boundary when $\phi(x) = -1 + a_2 + \tfrac{b}{2}$. Since $|\phi(x)| \leq \tfrac{1}{8}$ on $[-\tfrac{1}{2}, \tfrac{3}{2}]$ and $-1 + a_2 + \tfrac{b}{2} \leq -1 + \tfrac{3\nu}{2} < -\tfrac{1}{8}$ for $\nu < \tfrac{1}{8}$, this condition has no solution.

Therefore, for almost every $(a, b)$, the SCSC failure set reduces from the interval $[0,1]$ to a single point, consistent with Theorem~\ref{thm:SCSC}.

\begin{figure}
    \centering
    \includegraphics[width=1\linewidth]{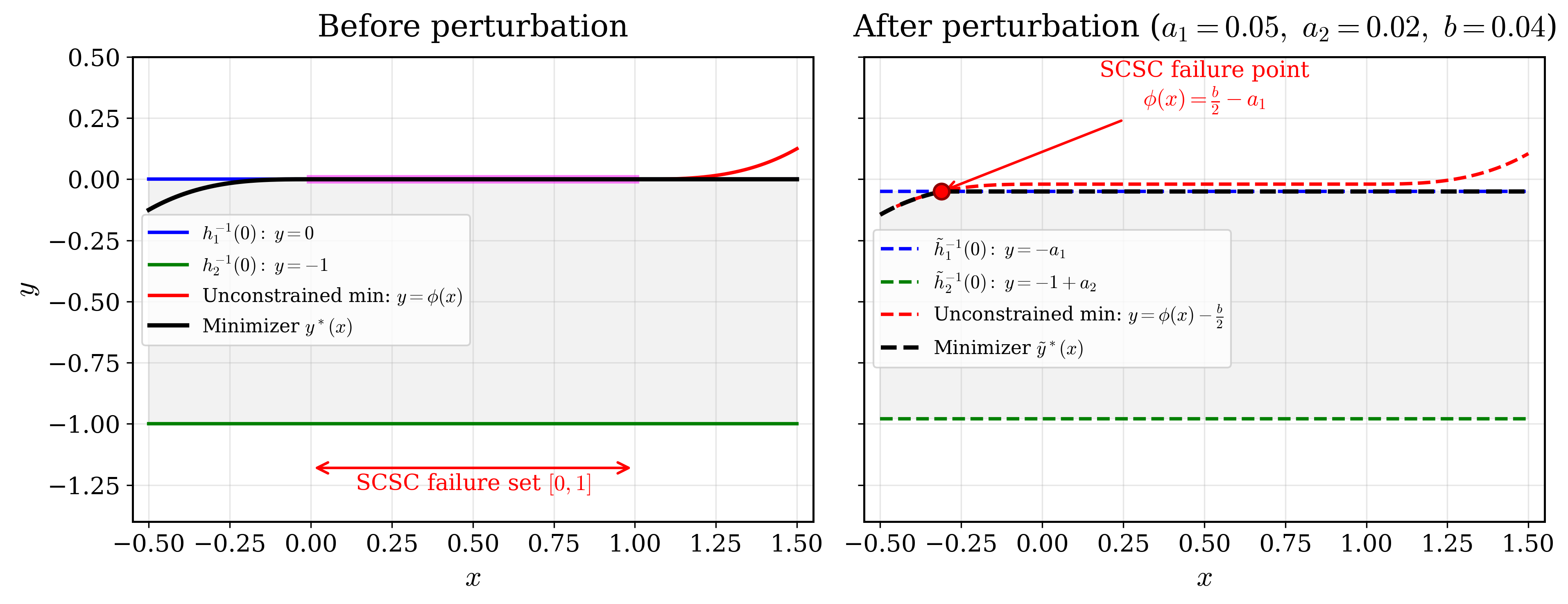}
    \caption{The feasible set (shaded), unconstrained minimizer, and constrained minimizer in Example~\ref{ex:scsc-prevalence}. Left: before perturbation, the unconstrained minimizer coincides with the boundary on $[0,1]$, so SCSC fails on the entire interval. Right: after perturbation, the SCSC failure set reduces to a single point (marked).}
    \label{fig:scsc_almost}
\end{figure}
\end{example}

\subsection{Prevalence of SOSC}

Finally, we consider SOSC for the lower-level problem. Similar to SCSC, SOSC depends on both the constraint structure and the curvature of the objective function, and therefore requires simultaneous perturbations of the constraint functions and to the objective, as in \eqref{eq:perturb3}. The following result establishes that, for almost every perturbation $a$ and $b$, the second-order sufficient optimality condition holds for the lower-level at almost every upper-level variable $x$.

\begin{theorem}[Prevalence of ``almost-every-$x$'' SOSC]\label{thm:SOSC}
    Suppose Assumption~\ref{assumption:general} holds. For any fixed positive constant $\nu<\rho$, where $\rho$ is from Assumption~\ref{assumption:general}(\ref{assumption:general12}), draw $b=(b_1,...,b_m)$ uniformly from $[0,\nu]^m$, $a=(a_1,...,a_k)$ uniformly from $[0,\nu]^k$, and perturb $g(x,y)$ and $h_i(x,y)$ as follows:
    \begin{align}\label{eq:perturb3}
        \widetilde{g}(x,y)&=g(x,y)+\left<y,b\right>,\\ 
        \widetilde{h}_i(x,y)&=h_i(x,y)+a_i,\quad\forall i.\nonumber
    \end{align}
    Then, with probability one (i.e., almost surely with respect to the random choice of $b$, and $a$), the set of $x\in\mathcal{X}$, at which the second-order sufficient optimality conditions do not hold at some local minimizer $y$ of the perturbed problem, has Lebesgue measure zero in $\mathcal{X}$.
\end{theorem}

The proof applies the same Sard-Fubini strategy as in Theorem~\ref{thm:LICQ}, but to the KKT map $G_{\mathcal J}(x, y, \lambda_{\mathcal J}) = (x, \nabla_y g + \sum_{i \in \mathcal{J}} \lambda_i \nabla_y h_i, h_{\mathcal J})$. If SOSC fails at a local minimizer, then there exists a nonzero critical direction $d$ with $d^\top \nabla^2_{yy} \mathcal{L}\, d = 0$, which implies the KKT matrix is singular. This singularity is equivalent to $(x, -b, -a_{\mathcal J})$ being a critical value of $G_{\mathcal J}$. By Sard's Theorem the set of such critical values has measure zero, and Fubini's Theorem transfers this to the $x$-space as before. The full proof is given in Appendix~\ref{proof:thm:SOSC}.

\begin{remark}
    The proof actually establishes a stronger property than stated in the theorem: for almost all perturbations, the KKT matrix in \eqref{eq:matrix1} is non-singular at every KKT point of the perturbed problem, not just at local minimizers. Since the second-order necessary condition dictates that the reduced Hessian is positive semi-definite at any local minimizer, this inherent non-degeneracy of the KKT matrix strictly precludes the existence of zero eigenvalues, thereby directly elevating second-order necessary condition to the strict positive definiteness required by second-order sufficient condition.
\end{remark}

The following example illustrates Theorem~\ref{thm:SOSC}: a problem whose SOSC failure set has positive measure before perturbation is eliminated after perturbation.

\begin{example}\label{ex:sosc-prevalence}
Consider the lower-level problem in $\mathbb{R}^2$ defined by
\[
\min_{y \in \mathbb{R}^2} \; g(x,y) = y_1^2 + 2y_1 + y_2^4 + \phi(x)^2 y_2^2 \quad \operatorname{s.t.} \quad 0 \leq y_1 \leq 1, \quad -1 \leq y_2 \leq 1,
\]
where $x \in [-\tfrac{1}{2}, \tfrac{3}{2}]$ and $\phi$ is defined as in Example~\ref{ex:licq-prevalence}. The feasible set $\mathcal{Y}(x) = [0,1] \times [-1,1]$ is independent of $x$, so LICQ holds everywhere. Since $\nabla_{y_1} g = 2y_1 + 2 > 0$ on $\mathcal{Y}(x)$, any minimizer must satisfy $y_1^* = 0$, with the constraint $-y_1 \leq 0$ active and multiplier $\lambda_1 = 2 > 0$. In the $y_2$-direction, $\nabla_{y_2} g = 2y_2(2y_2^2 + \phi(x)^2) = 0$ gives $y_2^* = 0$, since $2y_2^2 + \phi(x)^2 > 0$ for all $y_2$ when $\phi(x) \neq 0$, and the factor $2y_2^2 + \phi(x)^2$ vanishes only at $y_2 = 0$ when $\phi(x) = 0$. Thus the unique minimizer is $y^* = (0,0)$ for all $x$, and SCSC holds since $\lambda_1 = 2 > 0$.

The critical cone at $y^*$ is $\mathcal{C} = \{d : d_1 = 0\}$, and for $d = (0, d_2) \in \mathcal{C}$,
\begin{align}\label{eq:counterexample3.3}
    d^\top \nabla_{yy}^2 g(x, y^*) \, d = (12y_2^{*2} + 2\phi(x)^2) d_2^2 = 2\phi(x)^2 d_2^2.    
\end{align}
For $x \in [0,1]$, $\phi(x) = 0$, so \eqref{eq:counterexample3.3} vanishes for all $d \in \mathcal{C}$ and SOSC fails. For $x \notin [0,1]$, $\phi(x) \neq 0$, so \eqref{eq:counterexample3.3} is strictly positive for all nonzero $d \in \mathcal{C}$ and SOSC holds. Thus, the SOSC failure set is the entire interval $[0,1]$.

Now consider the perturbation in \eqref{eq:perturb3}. The linear perturbation $b_2 y_2$ does not affect the Hessian, so the second-order condition along $\mathcal{C}$ remains $12 y_2^2 + 2\phi(x)^2 > 0$. This condition fails only when $y_2 = 0$ and $\phi(x) = 0$ simultaneously. However, the perturbed stationarity condition in $y_2$ becomes $4y_2^3 + 2\phi(x)^2 y_2 + b_2 = 0$, which at $y_2 = 0$ and $\phi(x) = 0$ reduces to $b_2 = 0$. In other words, when $b_2 \neq 0$, the point $y_2 = 0$ is no longer stationary for any $x \in [0,1]$, and the perturbed stationary point $y_2^*$ satisfies $y_2^* \neq 0$, so $12 y_2^{*2} + 2\phi(x)^2 > 0$ and SOSC holds. Since the event $b_2 = 0$ has probability zero, for almost every $b$ the SOSC failure set is empty, consistent with Theorem~\ref{thm:SOSC}.
\end{example}

Now we have established that under the perturbation scheme in \eqref{eq:perturbation}, LICQ, SCSC, and SOSC hold with probability one (in the sense of random choice of $a$ and $b$) for the lower-level problem at almost every upper-level variable $x$. However, as shown in Section 2, the stronger ``every-$x$'' version of these conditions is non-prevalent: there exist problems for which no sufficiently small perturbation can make the conditions hold at every $x$. Although the ``almost-every-$x$'' and ``every-$x$'' versions differ only on a measure-zero set of upper-level variables, this difference can have substantial consequences for bilevel optimization. We analyze this gap in the following section.

\section{Gap Between ``Almost-Every-$x$'' and ``Every-$x$''}\label{section:obstructionforbilevel}
In Section~\ref{sec:nonprevalence}, we showed that Statement~\ref{stmt:every_x} is non-prevalent. In Section~\ref{section:prevalence}, we showed that Statement~\ref{stmt:almost_every_x} is prevalent. Although these two statements differ only on a measure-zero set of upper-level variables, the gap between them is far from negligible. In this section, we examine what is lost when the regularity conditions fail on such a measure-zero set. The structural consequences depend on which condition fails and on the problem setting:
\begin{itemize}
    \item \textbf{LICQ failure} (Section~\ref{subsec:role_licq}): the Lagrange multiplier loses uniqueness and continuity.
    \item \textbf{SCSC failure under strongly convex lower-level} (Section~\ref{subsec:role_scsc}): the unique minimizer $y^*(x)$ loses differentiability, turning the bilevel problem from a smooth into a nonsmooth optimization problem.
    \item \textbf{SCSC or uniform SOSC failure under non-strongly-convex lower-level} (Section~\ref{subsec:role_sosc}): a branch of local minimizers can terminate abruptly, making the solution mapping $x \mapsto y^*(x)$ discontinuous.
\end{itemize}
In each case, we also discuss the numerical consequences: the formulas~\eqref{eq:sensitivity_formula} and \eqref{eq:complementarity_sensitivity} used for hypergradient computation become ill-conditioned not only at the failure points themselves, but throughout an entire neighborhood surrounding them.

\subsection{The gap caused by LICQ failure}\label{subsec:role_licq}

When LICQ holds for the lower-level problem at every $x \in \mathcal{X}$, the Lagrange multiplier associated with any KKT point is well-behaved in the following sense.

\begin{proposition}[Uniqueness and continuity of Lagrange multipliers]\label{prop:multiplier_continuity}
    Suppose Assumption~\ref{assumption:general2} holds, at every $x \in \mathcal{X}$, LICQ holds for the lower-level problem. Then, for any KKT point $(x, y^\ast(x))$, the associated multiplier $\lambda(x)$ is unique. Moreover, for any continuous trajectory of KKT points $(x,y^\ast(x))$, the multiplier $\lambda(x)$ is also continuous.
\end{proposition}
The uniqueness of the multiplier is immediate from LICQ: the active constraint gradients are linearly independent, so the KKT system has a unique solution in $\lambda$. For the continuity, we refer the reader to \citet[Lemma~B.2]{jiang2024barrier}.

However, the prevalence results of Section 3 only guarantee that LICQ holds at almost every $x$, leaving open the possibility of failure on a measure-zero set. By definition, the failure of LICQ at an upper-level variable $x$ means there exists some lower-level feasible point $y$ where the condition is violated. As discussed in Remark~\ref{remark:licq_optimal}, this feasible point can in 
fact coincide with the lower-level minimizer $y^*(x)$: there exist problems in which no sufficiently small perturbation can separate the minimizer from an LICQ failure point. Therefore, in the following discussion regarding the measure-zero set of $x$ where LICQ fails, we specifically consider the case where this failure occurs at a local minimizer $y^*(x)$. At such a failure point $x_0$, the Lagrange multiplier associated with $y^*(x_0)$ may lose both uniqueness and continuity, as the following example demonstrates.

\begin{example}\label{example:multiplier_discontinuity}
Consider the lower-level problem
\[
\min_{y \in \mathbb{R}}\ -y \quad \operatorname{s.t.}\quad y \leq 1,\quad y \leq x,
\]
where $x \in [0, 2]$. For $x > 1$, the unique minimizer is $y^\ast(x) = 1$ with multipliers $\lambda_1 = 1,\, \lambda_2 = 0$. For $x < 1$, the unique minimizer is $y^\ast(x) = x$ with multipliers $\lambda_1 = 0,\, \lambda_2 = 1$. At $x = 1$, both constraints are active and their gradients coincide ($\nabla_y h_1 = \nabla_y h_2 = 1$), so LICQ fails. The KKT condition requires only $\lambda_1 + \lambda_2 = 1$, leaving the multiplier non-unique. Since $\lambda_1(x) = 1$ for $x > 1$ and $\lambda_1(x) = 0$ for $x < 1$, no selection of multipliers at $x = 1$ can make $\lambda(x)$ continuous.
\end{example}

Beyond the loss of multiplier continuity, the failure of LICQ at $y^\ast(x)$ also affects the sensitivity analysis that underpins hypergradient computation in bilevel optimization. To see this, suppose $(x, y^\ast(x))$ is a continuous KKT curve along which LICQ and SCSC hold, and the lower-level problem $g(x,\cdot)$ is strongly convex. Then the sensitivity of $y^\ast(x)$ with respect to $x$ takes the form (see, e.g., \citet{xu2023efficient})
\begin{align}\label{eq:sensitivity_formula}
    \begin{pmatrix}
        \nabla_x y^\ast(x) \\
        \nabla_x \lambda_{\mathcal{J}(x,y^\ast(x))}(x,y^\ast(x))
    \end{pmatrix}
    = -M_+(x)^{-1} N_+(x),
\end{align}
where $\mathcal{J}(x,y^\ast(x))$ denotes the active constraint index set at $y^\ast(x)$, and
\[
M_+(x) = \begin{pmatrix}
    \nabla_{yy}^2 \mathcal{L}(x, y^\ast(x), \lambda) & \nabla_y h_{\mathcal{J}(x,y^\ast(x))}(x, y^\ast(x))^\top \\
    \nabla_y h_{\mathcal{J}(x,y^\ast(x))}(x, y^\ast(x)) & 0
\end{pmatrix}, \quad
N_+(x) = \begin{pmatrix}
    \nabla_{xy}^2 \mathcal{L}(x, y^\ast(x), \lambda)^\top \\
    \nabla_x h_{\mathcal{J}(x,y^\ast(x))}(x, y^\ast(x))
\end{pmatrix}.
\]
A second sensitivity formula, due to~\citet{giovannelli2021inexact}, works directly with the full complementarity system without splitting constraints into active and inactive subsets. The sensitivity formula takes the form
\begin{equation}\label{eq:complementarity_sensitivity}
    \nabla v(x) = -\bigl(\nabla_v G(x, v(x))\bigr)^{-1} \nabla_x G(x, v(x)),
\end{equation}
where $v(x) \coloneq (y(x), \lambda(x))$, $G(x, v(x)) \coloneq \bigl(\nabla_y \mathcal{L}(x, y(x), \lambda(x)),\ \lambda(x) \circ h(x, y(x))\bigr)$ with $\circ$ denoting the elementwise (Hadamard) product, and
\[
\nabla_v G(x, v(x)) = \begin{pmatrix}
    \nabla_{yy}^2 \mathcal{L}(x, y(x), \lambda(x)) & \nabla_y h(x, y(x))^\top \\
    \operatorname{diag}(\lambda(x))\,\nabla_y h(x, y(x)) & C_h(x, y(x))
\end{pmatrix},
\]
\[
\nabla_x G(x, v(x)) = \begin{pmatrix}
    \nabla_{yx}^2 \mathcal{L}(x, y(x), \lambda(x)) \\
    \operatorname{diag}(\lambda(x))\,\nabla_x h(x, y(x))
\end{pmatrix},
\]
with $C_h(x, y) \coloneq \operatorname{diag}(h(x, y))$.

The nonsingularity of $M_+(x)$ in~\eqref{eq:sensitivity_formula} and that of $\nabla_v G(x, v(x))$ in~\eqref{eq:complementarity_sensitivity} both require LICQ. If there exists a point $x_0 \in \mathcal{X}$ at which LICQ fails, the numerical consequences for both~\eqref{eq:sensitivity_formula} and~\eqref{eq:complementarity_sensitivity} manifest at two levels:
\begin{itemize}
    \item At $x_0$ itself, LICQ failure means that there exists a nonzero vector $\beta$ indexed by $\mathcal{J}(x_0,y^*(x_0))$ such that $\sum_{i\in\mathcal{J}(x_0,y^*(x_0))}\beta_i\nabla_y h_i(x_0,y^*(x_0))=0$. This directly makes $M_+(x_0)$ in \eqref{eq:sensitivity_formula} singular. For $\nabla_v G(x, v(x))$ in~\eqref{eq:complementarity_sensitivity}, extending $\beta$ by zeros outside $\mathcal{J}(x_0,y^*(x_0))$, we obtain $\nabla_vG(x_0,v) (0,\beta)=0$ for any KKT pair $v=(y^*(x_0),\lambda)$. Hence both sensitivity formulas fail at $x_0$;
    \item In a neighborhood of $x_0$, the two formulas degrade through different mechanisms. For~\eqref{eq:sensitivity_formula}, applying the formula becomes numerically problematic because it requires identifying the exact active index set $\mathcal{J}(x,y^\ast(x))$. As shown in Remark~\ref{remark:licq_optimal} and Example~\ref{example:multiplier_discontinuity}, when LICQ fails at a minimizer $y^*(x_0)$, the active constraint pattern at $y^*(x)$ typically changes as $x$ crosses $x_0$: some constraint that is strictly inactive on one side becomes active exactly at $x_0$, or vice versa. Consequently, in a neighborhood of $x_0$, certain constraints have values $h_i(x, y^*(x))$ close to zero without being strictly active, and numerical solvers under finite precision cannot reliably distinguish these nearly-active constraints from strictly active ones. This leads the algorithm to identify an incorrect active set $\widehat{\mathcal{J}} \neq \mathcal{J}(x,y^\ast(x))$, assemble the wrong KKT system, and compute an incorrect hypergradient. For~\eqref{eq:complementarity_sensitivity}, the failure mode is different. The complementarity Jacobian becomes nearly singular because the active-gradient block loses rank. Since LICQ fails at $(x_0,y^*(x_0))$, there exists a nonzero vector $\beta$ such that $\sum_{i\in\mathcal{J}(x_0,y^*(x_0))}\beta_i\nabla_y h_i(x_0,y^*(x_0))=0$. Extend $\beta$ by zeros outside this active set and take $z=(0,\beta)$, where the first component corresponds to the $y$-variables and the second to the multiplier variables. Along a KKT trajectory with $y^*(x)\to y^*(x_0)$, we have $\nabla_vG(x,v(x))z=(\sum_i\beta_i\nabla_y h_i(x,y^*(x)),\, C_h(x,y^*(x))\beta)\to0$, because the first block tends to zero by the linear dependence above and the second block tends to zero since $h_i(x,y^*(x))\to0$ for every $i\in\mathcal{J}(x_0,y^*(x_0))$. Hence the smallest singular value of $\nabla_vG(x,v(x))$ tends to zero, and the inverse in~\eqref{eq:complementarity_sensitivity} becomes ill-conditioned near the LICQ failure point.
\end{itemize}

To make these failure modes concrete, we revisit Example~\ref{example:multiplier_discontinuity}, where $x_0 = 1$ is the LICQ failure point. Since $\nabla_y h_1 = \nabla_y h_2 = 1$ and $\nabla_{yy}^2 \mathcal{L} = 0$, at $x_0$ both constraints are active and
$M_+(x_0) = \begin{pmatrix} 0 & 1 & 1 \\ 1 & 0 & 0 \\ 1 & 0 & 0 \end{pmatrix}$ has two identical rows and is singular. Thus~\eqref{eq:sensitivity_formula} is not well defined at the LICQ failure point. The complementarity Jacobian also fails pointwise at $x_0$. Indeed,
\[
\nabla_v G(x, v(x)) = \begin{pmatrix} 0 & 1 & 1 \\ \lambda_1(x) & h_1(x, y^*(x)) & 0 \\ \lambda_2(x) & 0 & h_2(x, y^*(x)) \end{pmatrix}.
\]
At $x_0 = 1$, $h_1=h_2=0$, so rows $2,3$ reduce to $(\lambda_1(x_0),0,0)$ and $(\lambda_2(x_0),0,0)$, which are linearly dependent for any multiplier satisfying $\lambda_1+\lambda_2=1$; hence $\nabla_v G(x_0,v(x_0))$ is singular.

In a neighborhood of $x_0$, the difficulty for~\eqref{eq:sensitivity_formula} is active-set identification: with active set $\{1\}$ for $x>1$ and $\{2\}$ for $x<1$, $M_+(x)=\begin{pmatrix} 0 & 1 \\ 1 & 0 \end{pmatrix}$ is nonsingular on each side, but the inactive constraint satisfies $|h_i(x,y^*(x))|=|1-x|$, which is within floating-point tolerance of zero near $x_0=1$, so numerical solvers cannot reliably distinguish the strictly active constraint from the nearly-active one. For~\eqref{eq:complementarity_sensitivity}, the issue is instead ill-conditioning: 
\[
x > 1:\quad \det \nabla_v G(x, v(x)) = x - 1; \qquad x < 1:\quad \det \nabla_v G(x, v(x)) = 1 - x.
\]
In both cases $\det \nabla_v G(x, v(x))\to0$ as $x\to1$, so $\|(\nabla_v G(x,v(x)))^{-1}\|\to\infty$.

Therefore, although LICQ fails only on a measure-zero set of upper-level variables, its consequences extend to a neighborhood of positive measure around each failure point: both sensitivity formulas~\eqref{eq:sensitivity_formula} and~\eqref{eq:complementarity_sensitivity} are undefined at $x_0$ itself and numerically unreliable nearby, the former due to active set misidentification and the latter due to ill-conditioning of $\nabla_v G(x, v(x))$. Consequently, any bilevel algorithm that relies on~\eqref{eq:sensitivity_formula}~\citep{pmlr-v202-khanduri23a,xu2023efficient,kornowski2024firstordermethodslinearlyconstrained} or on~\eqref{eq:complementarity_sensitivity}~\citep{giovannelli2021inexact} to compute the hypergradient suffers not just at the measure-zero failure set, but throughout an entire neighborhood surrounding it.

\subsection{The gap caused by SCSC failure under strongly convex lower-level}\label{subsec:role_scsc}

When both LICQ and SCSC hold for the lower-level problem at every $x \in \mathcal{X}$, and the lower-level objective $g(x, \cdot)$ is strongly convex, the unique minimizer $y^\ast(x)$ and its associated multiplier $\lambda(x)$ vary smoothly with $x$.

\begin{proposition}[Smoothness of the minimizer and multiplier]\label{prop:smooth_kkt_curve}
    Suppose Assumption~\ref{assumption:general2} holds, the lower-level objective $g(x, \cdot)$ is strongly convex and each lower-level constraint function \(h_i(x,\cdot)\), \(i=1,\ldots,k\), is convex for every $x \in \mathcal{X}$, and that LICQ and SCSC hold for the lower-level problem at every $x \in \mathcal{X}$. Then the unique minimizer $y^\ast(x)$ and its associated Lagrange multiplier $\lambda(x)$ are smooth functions of $x$.
\end{proposition}

The proof outline is as follows. By Theorem~\ref{thm:main_scsc}, the active set is constant across all $x \in \mathcal{X}$. The result then follows from the Implicit Function Theorem applied to the reduced KKT system. The full proof is given in Appendix~\ref{proof:prop:smooth_kkt_curve}.

Proposition~\ref{prop:smooth_kkt_curve} implies that under the setting of strongly convex lower-level with LICQ and SCSC, the sensitivity formula~\eqref{eq:sensitivity_formula} is well-defined at every $x \in \mathcal{X}$, and the hyperfunction $x \mapsto f(x, y^\ast(x))$ is smooth. In particular, the bilevel problem is a smooth optimization problem.

However, the prevalence results of Section~\ref{section:prevalence} only guarantee that SCSC holds at almost every $x$, leaving open the possibility of failure on a measure-zero set. At such failure points, the active constraint index set may jump, and the minimizer $y^\ast(x)$ may lose differentiability, as the following example demonstrates.

\begin{example}\label{example:scsc_kink}
    Consider the lower-level problem
    \[
    \min_{y \in \mathbb{R}}\ y^2 \quad \operatorname{s.t.}\quad y - x \leq 0,
    \]
    where $x \in [-1, 1]$. For $x > 0$, the unconstrained minimizer $y = 0$ is feasible, so $y^\ast(x) = 0$ with multiplier $\lambda(x) = 0$ and the constraint inactive. For $x < 0$, the feasible set is $\{y : y \leq x < 0\}$, which does not contain $y = 0$, so the minimizer is $y^\ast(x) = x$ with multiplier $\lambda(x) = -2x > 0$ and the constraint active. At $x = 0$, the constraint $y \leq 0$ is active at $y^\ast(0) = 0$ with multiplier $\lambda(0) = 0$, so SCSC fails. The minimizer
    \[
    y^\ast(x) = \min(0, x) = \begin{cases} x & \text{if } x < 0, \\ 0 & \text{if } x \geq 0, \end{cases}
    \]
    is continuous but not differentiable at $x = 0$.
\end{example}

Example~\ref{example:scsc_kink} illustrates that even a single point of SCSC failure can introduce a nondifferentiable point in the mapping $x \mapsto y^\ast(x)$, and consequently in the hyperfunction $x \mapsto f(x, y^\ast(x))$. In other words, the bilevel problem becomes a nonsmooth optimization problem. Since smooth and nonsmooth optimization problems belong to fundamentally different problem classes requiring qualitatively different algorithmic tools, the loss of SCSC on even a measure-zero set can have a significant impact on the tractability of the bilevel problem.

We next explain how SCSC failure affects the two sensitivity formulas~\eqref{eq:sensitivity_formula} and~\eqref{eq:complementarity_sensitivity} introduced in Section~\ref{subsec:role_licq}. The reduced formula~\eqref{eq:sensitivity_formula} relies on a fixed active pattern. When SCSC fails at $x_0$, there may exist an active constraint satisfying $h_i(x_0,y^*(x_0))=0$ and $\lambda_i(x_0)=0$. Such a constraint can be active on one side of $x_0$ and inactive on the other side, so the minimizer mapping may switch between two reduced KKT systems at $x_0$. Consequently, \eqref{eq:sensitivity_formula} may fail to provide a well-defined derivative of $y^*(x)$ at $x_0$. In a neighborhood of $x_0$, as in the LICQ case (see Section~\ref{subsec:role_licq}), the same active-pattern change also makes active-set identification numerically unreliable, since a solver may misidentify a nearly active constraint as active, or an active constraint as inactive.

For the complementarity formula~\eqref{eq:complementarity_sensitivity}, SCSC failure directly degenerates the complementarity Jacobian. If SCSC fails at $x_0$, then for some index $i$ we have $\lambda_i(x_0)=0$ and $h_i(x_0,y^*(x_0))=0$. The corresponding row of \(\nabla_vG(x_0,v(x_0))\) contains the terms
\(\lambda_i(x_0)\nabla_y h_i(x_0,y^*(x_0))\) and
\(h_i(x_0,y^*(x_0))\); hence this row vanishes at $x_0$, and $\nabla_vG(x_0,v(x_0))$ is singular. In a neighborhood of $x_0$, as $\lambda_i(x)\to0$ and $h_i(x,y^*(x))\to0$, the same row becomes small, so the inverse in~\eqref{eq:complementarity_sensitivity} can become ill-conditioned. Thus, even if SCSC fails only on a measure-zero set of upper-level variables, the complementarity-based hypergradient computation can be unreliable around the failure point.

To make this concrete, we revisit Example~\ref{example:scsc_kink}. For the problem $\min_y y^2$ s.t. $y - x \leq 0$, the unique minimizer is $y^*(x) = \min(0, x)$ with multiplier $\lambda(x) = \max(0, -2x)$, and SCSC fails at $x_0 = 0$. The complementarity Jacobian takes the form
\[
\nabla_v G (x,v(x)) = \begin{pmatrix} 2 & 1 \\ \lambda(x) & h(x, y^*(x)) \end{pmatrix}.
\]
We evaluate this matrix on each side of $x_0 = 0$:
\[
x < 0: \ 
\nabla_v G(x,v(x)) = \begin{pmatrix} 2 & 1 \\ -2x & 0 \end{pmatrix},
\quad \det = 2x;
\quad
x > 0: \ 
\nabla_v G(x,v(x)) = \begin{pmatrix} 2 & 1 \\ 0 & -x \end{pmatrix},
\quad \det = -2x.
\]
In both cases, $\det \to 0$ as $x \to 0$, so $\|(\nabla_v G(x,v(x)))^{-1}\| \to \infty$. Hence the sensitivity computation~\eqref{eq:complementarity_sensitivity} is ill-conditioned throughout a neighborhood of $x_0$, not only at $x_0$ itself.

\subsection{The gap caused by SCSC or uniform SOSC failure under non-strongly-convex lower-level}\label{subsec:role_sosc}

In the non-strongly-convex setting, the lower-level problem may have multiple local minimizers. Thus, a bilevel objective defined through a selected local minimizer branch $y^*(x)$ is well behaved only if this branch can be continued as $x$ varies. We first state a positive result showing that LICQ, SCSC, and uniform SOSC guarantee the global existence of such a smooth branch. This justifies the global smooth branch assumed in existing works such as \citet{giovannelli2021inexact,nolasco2025sensitivity}. This conclusion is not an immediate consequence of the Implicit Function Theorem or Robinson-type stability results~\citep{robinson1980strongly}, which only provide local continuation near a given point.

\begin{theorem}[Global smooth continuation of local minimizer branches]\label{thm:global_continuation}
    Suppose Assumption~\ref{assumption:general2} holds, and \(\mathcal X\) is a connected, simply connected 
    smooth manifold with generalized boundary. Suppose also that
    \(\mathcal Y(x)\) is uniformly bounded for all \(x\in\mathcal X\). Furthermore, assume that,
    at every \(x\in\mathcal X\), LICQ, SCSC, and uniform SOSC hold for the lower-level problem.
    Then, for any \(x_0\in\mathcal X\) and any local minimizer \(y_0\) of the lower-level problem
    at \(x_0\), there exists a unique global \(C^\infty\) local-minimizer branch
    \(y^\ast:\mathcal X\to\mathbb R^m\) such that \(y^\ast(x_0)=y_0\).
\end{theorem}

Here ``Simply connected'' means every loop in $\mathcal{X}$ can be continuously shrunk to a point. Formal definitions are given in Appendix~\ref{sec:appendix_def}. Any convex subset of $\mathbb{R}^n$ satisfies this property. The proof proceeds in two steps. First, we establish a global continuous continuation of the local minimizer \(y_0\). Although the Implicit Function Theorem gives local continuations, the key difficulty is to assemble them into a globally well-defined map \(y^*:\mathcal X\to\mathbb R^m\) with \(y^*(x_0)=y_0\), where \(y^*(x)\) is a local minimizer of the lower-level problem at every \(x\in\mathcal X\). We achieve this by showing that the projection from the graph of local minimizers to \(\mathcal X\) is a covering map, and then invoking the homotopy lifting property: since \(\mathcal X\) is simply connected, any two paths connecting the same endpoints lift to the same endpoint, ensuring the global continuation is well-defined. Second, we upgrade continuity to smoothness using the Implicit Function Theorem at each point. The full proof is given in Appendix~\ref{proof:thm:global_continuation}.

Theorem~\ref{thm:global_continuation} justifies the global smooth branch assumption implicit in \citet{giovannelli2021inexact,nolasco2025sensitivity}: LICQ, SCSC, and uniform SOSC at every $x$ are sufficient. However, the prevalence results of Section~\ref{section:prevalence} only guarantee that these conditions hold at almost every $x$. When either SCSC or uniform SOSC fails, the global continuation may break down: a branch of local minimizers can abruptly terminate at a failure point. We illustrate these two distinct mechanisms with the following examples.

\begin{example}\label{example:scsc_discontinuity}
Consider the following parametric optimization problem:
\[
\min_{y \in \mathbb{R}^2}\ \frac{1}{2}y_1^2 - 2y_2^2 + x y_2 \quad \operatorname{s.t.}\quad -1 \leq y_1 \leq 1,\quad 0 \leq y_2 \leq 1,
\]
where $x \in [-1, 1]$ is the upper-level variable and $y \in \mathbb{R}^2$ is the decision variable. Since all constraints are linear, the Hessian of the Lagrangian with respect to $y$ is a constant matrix everywhere:
\[
\nabla_{yy}^2 L = \nabla_{yy}^2 g(x,y) = \begin{pmatrix} 1 & 0 \\ 0 & -4 \end{pmatrix}.
\]
Recall that uniform SOSC only needs to be verified at valid local minimizers. For $x \in [-1, 1]$, the set of local minimizers $M(x)$ and their corresponding KKT multipliers $\lambda$ for the active constraints are:
\begin{itemize}
    \item \textbf{Case 1}: $y^\ast = (0,0) \in M(x)$ (requires $x \in (0, 1]$). The active constraint is $-y_2 \leq 0$, with multiplier $\lambda = x > 0$.
    \item \textbf{Case 2}: $y^\ast = (0,1) \in M(x)$ (requires $x \in [-1, 1]$). The active constraint is $y_2 - 1 \leq 0$, with multiplier $\lambda = 4 - x \geq 3 > 0$.
\end{itemize}
For any valid local minimizer $y^\ast \in M(x)$, since SCSC holds, the critical cone $\mathcal{C}(x, y^\ast)$ is restricted to the orthogonal complement of the active constraint gradients. The gradients of the active constraints are $\pm e_2 = (0, \pm 1)^\top$, so the critical cone is
\[
\mathcal{C}(x, y^\ast) = \{d \in \mathbb{R}^2 : d_2 = 0\}.
\]
For any $d = (d_1, 0)^\top \in \mathcal{C}(x, y^\ast)$, we have $d^\top \nabla_{yy}^2 L\, d = d_1^2 = \|d\|^2$, so uniform SOSC holds with modulus $\sigma = 1$.
 
However, the local minimizer set mapping $M(x)$ is not continuous at $x = 0$. For $x \in (0, 1]$, the local minimizer set is $M(x) = \{(0,0), (0,1)\}$. At $x = 0$, strict complementarity fails at $(0,0)$: the constraint $-y_2 \leq 0$ is active with multiplier $\lambda = 0$. This causes the critical cone to expand from $\{d : d_2 = 0\}$ to $\{d : d_2 \geq 0\}$, exposing the negative curvature direction $d = (0, 1)^\top$ of the Hessian. Consequently, $(0,0)$ degenerates into a saddle point and is no longer in $M(0)$, leaving $M(0) = \{(0,1)\}$. The branch of local minimizers at $(0,0)$ terminates abruptly at $x = 0$ and cannot be continued to $x < 0$ (see Figure~\ref{fig:scsc_discontinuity_3d}).

\begin{figure}
    \centering
    \includegraphics[width=1\linewidth]{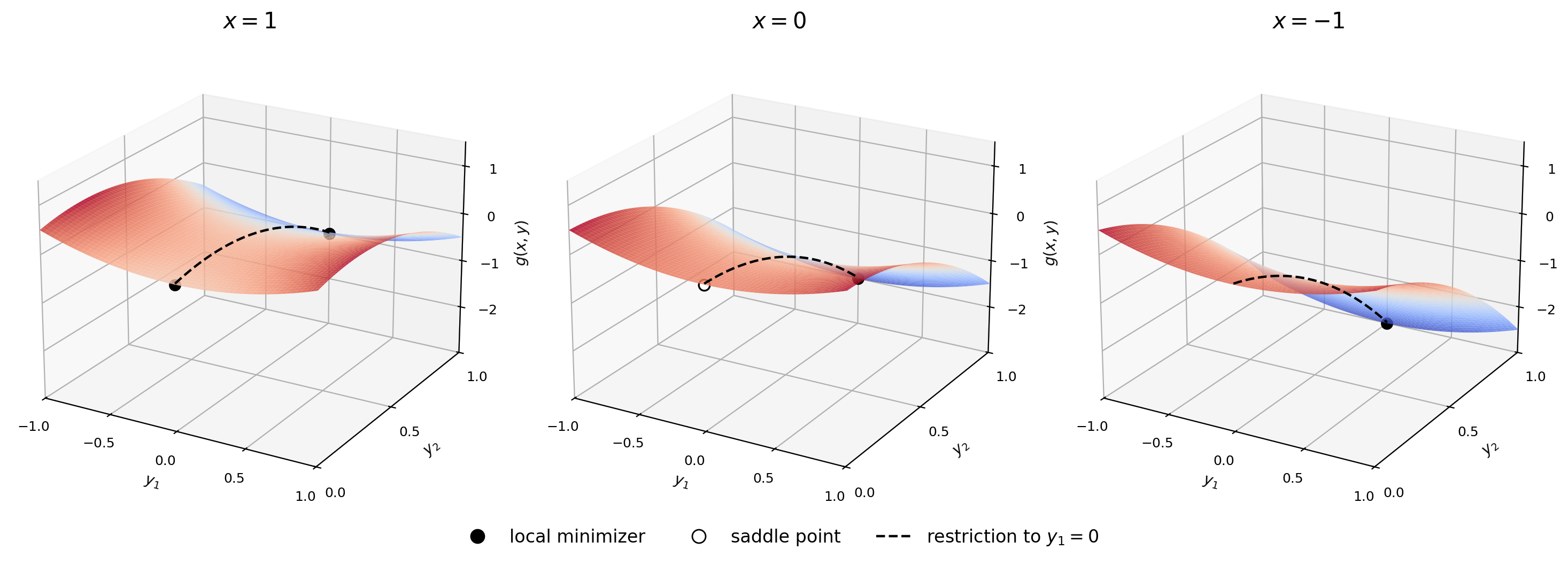}
    \caption{The objective function $g(x,y)$ over the feasible set at $x=1$, $x=0$, and $x=-1$. For $x>0$ (left), $(0,0)$ is a local minimizer. At $x=0$ (center), it degenerates into a saddle point. For $x<0$ (right), $(0,0)$ is no longer a KKT point. The dashed curve on each surface shows the restriction of $g$ to $y_1=0$.}
    \label{fig:scsc_discontinuity_3d}
\end{figure}
\end{example}

Example~\ref{example:scsc_discontinuity} shows that the failure of SCSC can cause a local minimizer branch to terminate by degenerating into a saddle point. The next example illustrates a different mechanism: even when SCSC holds everywhere, the failure of uniform SOSC can produce a bifurcation in which a local minimizer and a local maximizer on a boundary stratum collide and annihilate each other.

\begin{example}\label{example:sosc_bifurcation}
Consider the following parametric optimization problem:
\[
\min_{y \in \mathbb{R}^2}\ y_1 + \frac{1}{2}y_1^2 + \frac{1}{3}y_2^3 - x y_2 \quad \operatorname{s.t.}\quad -y_1 \leq 0,\quad y_1 - 2 \leq 0,\quad -\frac{3}{2} - y_2 \leq 0,\quad y_2 - \frac{3}{2} \leq 0,
\]
where $x \in [-1, 1]$ and $y = (y_1, y_2) \in \mathbb{R}^2$. Since $\frac{\partial g}{\partial y_1} = 1 + y_1 > 0$ on the feasible set, any local minimizer must satisfy $y_1 = 0$. At such points, the constraint $-y_1 \leq 0$ is active with multiplier $\lambda = 1 > 0$, so LICQ and SCSC hold. The critical cone is $\mathcal{C} = \{d : d_1 = 0\}$, and the reduced Lagrangian on the boundary stratum $\{y_1 = 0\}$ is $\ell(y_2) = \frac{1}{3}y_2^3 - xy_2$, with $\ell'(y_2) = y_2^2 - x$ and $\ell''(y_2) = 2y_2$.

We focus on the two critical points of $\ell$ that undergo bifurcation. For $x > 0$, the equation $\ell'(y_2) = 0$ yields $y_2 = \sqrt{x}$ (a local minimizer with $\ell''(\sqrt{x}) = 2\sqrt{x} > 0$) and $y_2 = -\sqrt{x}$ (a local maximizer with $\ell''(-\sqrt{x}) = -2\sqrt{x} < 0$). As $x \to 0^+$, the second-order modulus $2\sqrt{x} \to 0$, and at $x = 0$ the two critical points merge into a single degenerate point at $y_2 = 0$ with $\ell''(0) = 0$. For $x < 0$, neither critical point exists.

Pointwise SOSC holds at every local minimizer for every $x > 0$, but uniform SOSC fails because the modulus $2\sqrt{x} \to 0$ as $x \to 0^+$. This permits the saddle-node bifurcation: the local minimizer and the local maximizer collide and annihilate at $x = 0$, causing the local minimizer branch to terminate (see Figure~\ref{fig:sosc_bifurcation_2d}).

\begin{figure}
    \centering
    \includegraphics[width=1\linewidth]{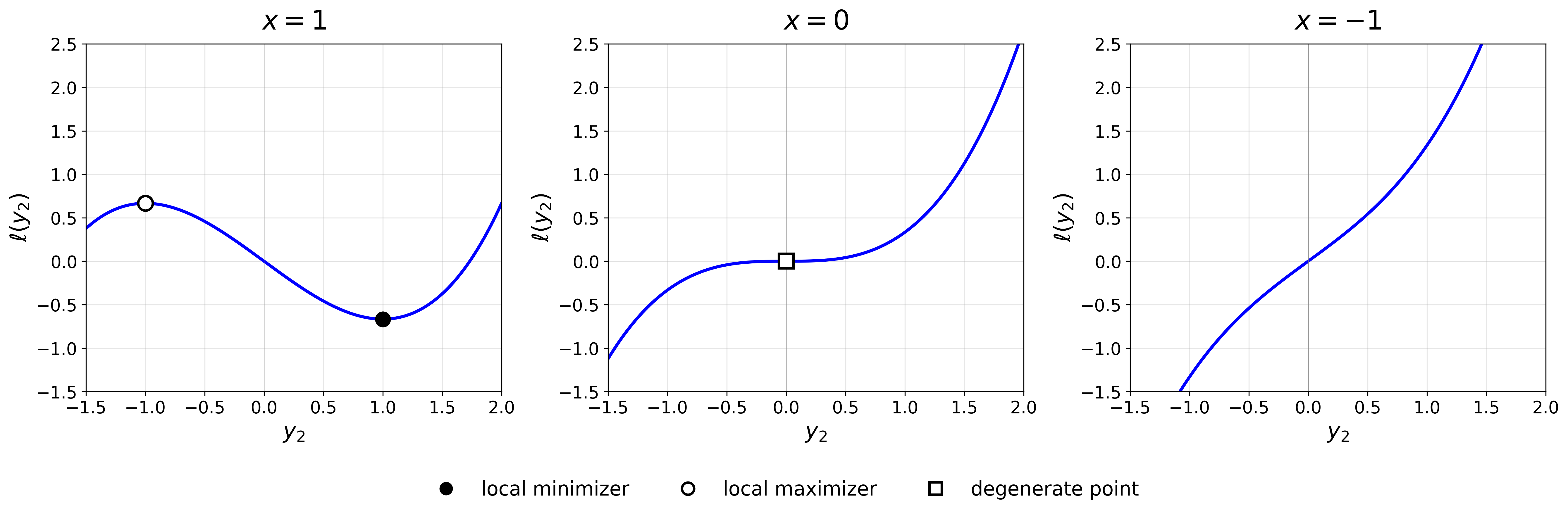}
    \caption{The reduced Lagrangian $\ell(y_2) = \frac{1}{3}y_2^3 - xy_2$ on the boundary stratum $\{y_1 = 0\}$ in Example~\ref{example:sosc_bifurcation}. At $x = 1$ (left), there is a local minimizer at $y_2 = 1$ (filled circle) and a local maximizer at $y_2 = -1$ (open circle). At $x = 0$ (center), the two critical points merge into a single degenerate point at $y_2 = 0$ (open square) with $\ell''(0) = 0$. At $x = -1$ (right), neither critical point exists. The local minimizer and local maximizer collide and annihilate via a saddle-node bifurcation at $x = 0$.}
    \label{fig:sosc_bifurcation_2d}
\end{figure}

\end{example}

Examples~\ref{example:scsc_discontinuity} and~\ref{example:sosc_bifurcation} show that the failure of SCSC or uniform SOSC can cause branches of local minimizers to terminate abruptly. In the bilevel framework, where the hyperfunction is defined by selecting a local minimizer branch $y^\ast(x)$ as in~\citet{giovannelli2021inexact,nolasco2025sensitivity}, such termination makes the solution mapping $x \mapsto y^\ast(x)$ discontinuous. Consequently, the hyperfunction $x \mapsto f(x, y^\ast(x))$ is also discontinuous, placing the bilevel problem outside the class of even nonsmooth continuous optimization. This is a more severe degradation than the loss of differentiability caused by SCSC failure alone (Section~\ref{subsec:role_scsc}).
 
Beyond this structural discontinuity, the failure of LICQ, SCSC, or uniform SOSC also degrades the numerical computation of the hypergradient in a neighborhood of each failure point. The reduced sensitivity formula~\eqref{eq:sensitivity_formula} relies on the nonsingularity of the reduced KKT matrix $M_+(x)$, while the complementarity-based sensitivity formula~\eqref{eq:complementarity_sensitivity} relies on the nonsingularity of the Jacobian $\nabla_v G(x,v(x))$. When uniform SOSC fails along a local minimizer branch, the curvature of the Lagrangian on the critical cone degenerates: the SOSC modulus
$\sigma(x) \coloneq  \min_{\|d\|=1, d \in \mathcal{C}(x,y^*)} d^\top \nabla^2_{yy}\mathcal{L}\, d$
approaches zero as $x$ approaches the failure point. Since this curvature term keeps both $M_+(x)$ and $\nabla_v G(x,v(x))$ uniformly nonsingular, the two matrices may remain nonsingular away from the failure point but become increasingly ill-conditioned in its neighborhood. Consequently, $\|M_+(x)^{-1}\|$ and $\|(\nabla_v G(x,v(x)))^{-1}\|$ become large as $x$ approaches the failure point, so that the computed hypergradient can be dominated by numerical error throughout an entire neighborhood surrounding it.

\subsection{Discussion}\label{subsec:discussion}

We conclude this section with two remarks that provide further perspective on the gap between the ``almost-every-$x$'' and ``every-$x$'' regularity requirements.

\paragraph{The role of uniform quadratic growth.}
Theorem~\ref{thm:global_continuation} shows that LICQ, SCSC, and uniform SOSC together guarantee the global smooth continuation of each local minimizer branch. On the other hand, Examples~\ref{example:scsc_discontinuity} and~\ref{example:sosc_bifurcation} demonstrate that removing either SCSC or uniform SOSC can cause a local minimizer branch to terminate abruptly, destroying even continuity. At the local level, the mechanism behind continuation is clear: the Implicit Function Theorem, applied to the reduced KKT system, provides a locally unique smooth extension of each local minimizer whenever the KKT Jacobian $M_+(x)$ is nonsingular (which requires LICQ and SOSC) and the active set is constant (which requires SCSC). But why can these local extensions be assembled into a global one? And why does global continuation break down when any of the three conditions is removed?

The answer lies in a consequence that LICQ, SCSC, and uniform SOSC jointly imply: uniform local quadratic growth.

\begin{theorem}[Uniform local quadratic growth]\label{prop:uniform_quadratic_growth}
    Suppose Assumption~\ref{assumption:general2} holds, the upper-level variable space $\mathcal{X}$ is a compact set, and the lower-level feasible set $\mathcal{Y}(x)$ is uniformly bounded for all $x \in \mathcal{X}$. Furthermore, assume that the objective $g$ and constraint functions $h_i$ are twice continuously differentiable, and at every $x \in \mathcal{X}$, LICQ, SCSC, and uniform SOSC hold for the lower-level problem. Then, the local minimizers satisfy a uniform local quadratic growth condition over $\mathcal{X}$: there exist uniform constants $c > 0$ and $\delta > 0$ such that for any $x \in \mathcal{X}$ and any local minimizer $y^\ast(x)$ at $x$, it holds that
    \[
    g(x, y) \geq g(x, y^\ast(x)) + c\|y - y^\ast(x)\|^2, \quad \forall\, y \in \mathcal{Y}(x) \cap B_\delta(y^\ast(x)).
    \]
\end{theorem}

Uniform quadratic growth means that every local minimizer sits at the bottom of a quadratic bowl whose depth and width are bounded below uniformly over $x$. This prevents local minimizers from degenerating as $x$ varies: no matter how far the upper-level variable moves, every local minimizer retains a definite gap in objective value from all nearby feasible points. It is this uniform gap that allows the local extensions provided by the Implicit Function Theorem to be assembled into a global continuation.

It is worth noting a distinction between the single-level and the parametric (bilevel) settings. In single-level constrained optimization, SOSC at a KKT point directly implies local quadratic growth at that point; this is a standard result in nonlinear programming~\citep{shapiro1992perturbation}. One might therefore expect that in the parametric setting, imposing LICQ and uniform SOSC at every $x$ would suffice to guarantee uniform local quadratic growth. However, this is not the case: SCSC is also indispensable. The reason is that SOSC controls curvature only within the critical cone, while the local quadratic growth in directions outside the critical cone is maintained by the strict separation between active and inactive constraints. More precisely, uniform quadratic growth requires a uniform constant $c > 0$ and radius $\delta > 0$ independent of $x$. Each condition controls a different component:
\begin{itemize}
    \item Uniform SOSC provides a uniform lower bound on the curvature of the Lagrangian Hessian within the critical cone. This lower bound contributes to both the quadratic growth constant $c$ and the uniform radius $\delta$. When it fails, this curvature approaches zero, and the constants in the quadratic growth estimate degenerate. This permits a bifurcation (Example~\ref{example:sosc_bifurcation}).
    \item SCSC, together with the compactness of $\mathcal{X}$, provides a uniform lower bound on the multipliers at active constraints. This uniform gap ensures growth in directions outside the critical cone and also contributes to both the quadratic growth constant $c$ and the uniform radius $\delta$. When SCSC fails at some $x_0$, the multiplier $\lambda_i(x_0) = 0$ for some active constraint, and the constants $c$ and $\delta$ in the quadratic growth estimate may degenerate as $x \to x_0$. At $x_0$ itself, the critical cone expands to include directions with negative curvature, destroying quadratic growth entirely and causing the local minimizer to degenerate into a saddle point (Example~\ref{example:scsc_discontinuity}).
\end{itemize}
In summary, LICQ, SCSC, and uniform SOSC are not merely independent technical assumptions: they work in concert to establish uniform quadratic growth, which is the structural foundation underlying the continuity of local minimizer branches. This stands in contrast to the single-level setting, where SOSC alone suffices for local quadratic growth, and highlights an intrinsic difficulty of bilevel optimization with constrained lower-level problems.

\paragraph{Pointwise SOSC versus uniform SOSC.}
In Section~\ref{section:prevalence}, we showed that pointwise SOSC at almost every $x$ is prevalent under finite-dimensional perturbations. In Section~\ref{sec:nonprevalence}, we showed that uniform SOSC at every $x$ is non-prevalent. A natural intermediate question is whether pointwise SOSC at every $x$ is prevalent. We leave this as an open question, but argue that its resolution is not essential for the bilevel optimization setting.

The reason is that, for bilevel optimization, pointwise SOSC at every $x$ is insufficient to guarantee the well-behavedness of the solution mapping or the hyperfunction. As demonstrated by Example~\ref{example:sosc_bifurcation}, even when LICQ and SCSC hold everywhere and pointwise SOSC holds at every local minimizer for every $x > 0$, the lack of a uniform lower bound on the second-order modulus permits a saddle-node bifurcation at $x = 0$, causing a local minimizer branch to disappear. Consequently, the solution mapping $x \mapsto y^\ast(x)$ is discontinuous and the hyperfunction $x \mapsto f(x, y^\ast(x))$ is ill-defined at the bifurcation point.

This observation highlights that pointwise SOSC at every $x$ is still insufficient for the structural properties established in Sections~\ref{subsec:role_sosc} and~\ref{subsec:discussion}. In particular, it does not ensure a uniform local quadratic growth estimate, and consequently it cannot prevent a local minimizer branch from terminating via a saddle-node bifurcation (Example~\ref{example:sosc_bifurcation}). Once such a termination occurs, the solution mapping $x \mapsto y^\ast(x)$ and the hyperfunction $x \mapsto f(x, y^\ast(x))$ along this branch lose continuity, even though SOSC holds pointwise away from the bifurcation point. Note that SOSC is by definition only imposed at local minimizers, and the bifurcation point itself has degenerated into a saddle point, so SOSC is not required to hold there. Therefore, while the prevalence of pointwise SOSC at every $x$ is of independent mathematical interest, it does not address the specific role of uniform SOSC in the bilevel setting, which is to ensure the global smooth continuation of local minimizer branches and the continuity of the associated hyperfunction.

\section{Conclusion}

In this work, we studied the strength and limitations of the regularity assumptions (LICQ, SCSC, and SOSC) commonly adopted in bilevel optimization with constrained lower-level problems. We showed that the ``every $x$'' versions are strong assumptions in the sense of non-prevalence. We established rigidity theorems identifying structural invariants that must be preserved when these conditions hold everywhere, including the stratification of the feasible set, the active constraint pattern along KKT trajectories, and the stratum-wise count of local minimizers. We then constructed robust counterexamples certifying the failure of each condition. In contrast, the ``almost every $x$'' versions of these conditions are weak assumptions in the sense of prevalence: they hold with probability one after arbitrarily small perturbations of the lower-level objective and constraints. We further analyzed the gap between the two requirements: even though the ``every $x$'' and ``almost every $x$'' versions differ only on a measure-zero set, this difference can cause the Lagrange multiplier to lose uniqueness, the hyperfunction to lose differentiability, or a branch of local minimizers to terminate abruptly. In all cases, the sensitivity formulas used for hypergradient computation become ill-conditioned in an entire neighborhood of each failure point. Finally, we showed that LICQ, SCSC, and uniform SOSC jointly imply uniform local quadratic growth, which serves as the structural foundation for the global smooth continuation of local minimizer branches.

\bibliographystyle{abbrvnat}
\bibliography{reference}

\appendix
\section*{Appendix}

\section{Definitions and Auxiliary Results}\label{sec:appendix_def}

This appendix collects the definitions and results used in the main text and proofs. We begin with the formal definition of prevalence (Definition~\ref{def:prevalence-formal}), the notion adopted throughout this work to quantify how strong or weak an assumption is. We then recall Sard's Theorem and transversality, which are the core techniques for the prevalence results in Section~\ref{section:prevalence}. Next, we collect the topological notions of connectedness, path-connectedness, local path-connectedness, and simple connectivity that appear in the assumptions of the rigidity and continuation theorems in Sections~\ref{sec:nonprevalence} and~\ref{subsec:role_sosc}. We then introduce manifolds with generalized boundary, which underlie the rigidity and stability results in Section~\ref{sec:nonprevalence}. Finally, we recall covering maps, which are used in the global continuation result (Lemma~\ref{lem:global_extension}).

\begin{definition}[Prevalence~\citep{hunt1992prevalence}]\label{def:prevalence-formal}
Let $\mathcal{V}$ be a topological vector space, and let $\mathcal{S} \subseteq \mathcal{V}$ be a Borel-measurable subset. We say that $\mathcal{S}$ is prevalent if there exists a finite-dimensional subspace $P \subseteq \mathcal{V}$, called a probe space, such that for every $v \in \mathcal{V}$, the set $\{p \in P : v + p \in \mathcal{S}\}$ has full Lebesgue measure in $P$. We say that a property is prevalent in $\mathcal{V}$ if the set of elements in $\mathcal{V}$ satisfying this property is prevalent.
\end{definition}

\begin{remark}
For Definition~\ref{def:prevalence-formal}, we take the ambient space to be
\[
\mathcal{V}=C^\infty(\mathcal X\times\mathbb R^m;\mathbb R^{k+1}),
\]
whose elements are the problem data \((g,h_1,\ldots,h_k)\), equipped with the
compact-open \(C^\infty\) topology. Equivalently, this topology is generated by
seminorms given by uniform bounds on derivatives of finite order over compact
subsets. The probe space used in Section~\ref{section:prevalence} is
\[
P=\{(b^\top y,a_1,\ldots,a_k):(b,a)\in\mathbb R^m\times\mathbb R^k\},
\]
which is a finite-dimensional linear subspace of \(\mathcal V\), naturally identified
with \(\mathbb R^{m+k}\) and equipped with the standard Lebesgue measure.
\end{remark}

\subsection*{Sard's Theorem and Transversality}

\begin{theorem}[Sard's Theorem]\label{thm:sard}
    Let $f : \mathbb{R}^p \to \mathbb{R}^q$ be $C^l$, where $l \geq \max\{p - q + 1, 1\}$. Let $X \subset \mathbb{R}^p$ denote the critical set of $f$, which is the set of points $x \in \mathbb{R}^p$ at which the Jacobian matrix of $f$ has rank $< q$. Then the critical value set, that is, the image $f(X)$, has Lebesgue measure $0$ in $\mathbb{R}^q$.
\end{theorem}

\begin{definition}[Transversality]\label{def:transversality}
    Let $f:X\to Y$ be a smooth map between smooth manifolds, and let $Z$ be a submanifold of $Y$. We say that $f$ is transverse to $Z$, denoted $f\pitchfork Z$, if for every $x\in f^{-1}(Z)$,
    $$\mathrm{im}(df_x)+T_{f(x)}Z=T_{f(x)}Y,$$
    where $T_{f(x)}Z$ and $T_{f(x)}Y$ are the tangent spaces of $Z$ and $Y$ at $f(x)$, and $\mathrm{im}(df_x)$ is the image of $df_x: T_x X\to T_{f(x)} Y$.
\end{definition}

Transversality formalizes the notion that $f$ intersects $Z$ in the most generic way possible. In particular, if $\dim X < \mathrm{codim}\, Z$, then $f \pitchfork Z$ implies $f^{-1}(Z) = \emptyset$.

\begin{theorem}[Parametric Transversality Theorem]\label{thm:param_transversality}
    Let $F:X\times S\to Y$ be a smooth map of manifolds, and let $Z$ be a submanifold of $Y$. If $F \pitchfork Z$, then for almost every $s\in S$, the slice map $f_s(\cdot)\coloneq F(\cdot,s)$ satisfies $f_s \pitchfork Z$.
\end{theorem}

\subsection*{Connectedness, path-connectedness, and simple connectivity}

We recall the topological notions used in the requirements of Theorems~\ref{thm:main_licq},~\ref{thm:main_scsc},~\ref{thm:main_sosc}, and~\ref{thm:global_continuation}. Throughout, $X$ denotes a topological space.

\begin{definition}[Connected space]\label{def:connected}
    $X$ is connected if it cannot be written as a disjoint union $X = U \sqcup V$ of two nonempty open subsets. Equivalently, the only subsets of $X$ that are both open and closed are $\emptyset$ and $X$ itself.
\end{definition}

\begin{definition}[Path-connected space]\label{def:path_connected}
    $X$ is path-connected if for any two points $x_0, x_1 \in X$, there exists a continuous map $\gamma : [0, 1] \to X$ with $\gamma(0) = x_0$ and $\gamma(1) = x_1$. Such a map is called a path from $x_0$ to $x_1$.
\end{definition}

\begin{definition}[Locally path-connected space]\label{def:locally_path_connected}
    $X$ is locally path-connected if every point $x \in X$ has a neighborhood basis of path-connected open sets; that is, for any $x \in X$ and any open neighborhood $U$ of $x$, there exists a path-connected open set $V$ with $x \in V \subseteq U$.
\end{definition}

Path-connectedness implies connectedness, but the converse is false in general. For example, the closure of the topologist's sine curve is connected but not path-connected. However, when $X$ is locally path-connected, the two notions coincide: a connected, locally path-connected space is automatically path-connected. This equivalence is used in the proof of Theorem~\ref{thm:main_licq}: since any MGB is locally path-connected, the assumption that $\mathcal{X}$ is a connected MGB yields path-connectedness for free.

\begin{definition}[Homotopy of paths relative to endpoints]\label{def:homotopy}
    Two paths $\gamma_0, \gamma_1 : [0, 1] \to X$ sharing the same endpoints $\gamma_0(0) = \gamma_1(0) = x_0$ and $\gamma_0(1) = \gamma_1(1) = x_1$ are homotopic relative to endpoints if there exists a continuous map $H : [0, 1] \times [0, 1] \to X$ satisfying $H(s, 0) = \gamma_0(s)$, $H(s, 1) = \gamma_1(s)$, and $H(0, t) = x_0$, $H(1, t) = x_1$ for all $s, t \in [0, 1]$.
\end{definition}

\begin{definition}[Simply connected space]\label{def:simply_connected}
    A path-connected space $X$ is simply connected if every loop in $X$ can be continuously contracted to a point: for any continuous map $\gamma : [0, 1] \to X$ with $\gamma(0) = \gamma(1) = x_0$, the map $\gamma$ is homotopic relative to endpoints to the constant map $s \mapsto x_0$.
\end{definition}

Equivalently, a path-connected space is simply connected if and only if any two paths with the same endpoints are homotopic relative to endpoints. This equivalent formulation is the form used in the proof of Lemma~\ref{lem:global_extension}: when continuing a branch of local minimizers along two different paths joining the same pair of upper-level variables, simple connectivity guarantees that the two resulting continuations agree, so the global branch is single-valued.

\subsection*{Manifolds with generalized boundary}

\begin{definition}[Manifold with generalized boundary~{\citep[Definition~3.1.9]{jongen2000nonlinear}}]\label{def:MGB}
A subset $X$ of $\mathbb{R}^n$ is called a manifold with generalized boundary (MGB) of class $C^r$ if, for each $x\in X$, there exists a local $C^r$-coordinate system of $\mathbb{R}^n$, say $\psi:U\to V$, with $\psi(x)=0$, such that
\[
\psi(U\cap X)
=
\{(v_1,\ldots,v_n)\in V
\mid
v_1=\cdots=v_q=0,\;
v_{q+1}\ge 0,\ldots,v_{q+p}\ge 0
\}.
\]
Here $q=q(x)$ and $p=p(x)$ may depend on $x$. The coordinates
$v_{q+1},\ldots,v_{q+p}$ are called semifree, and
$v_{q+p+1},\ldots,v_n$ are called free.
\end{definition}

\begin{definition}[Dimension and strata~{\citep[Definition~3.1.16]{jongen2000nonlinear}}]\label{def:stratum}
Let $X$ be an MGB in $\mathbb{R}^n$ of dimension $s$, that is,
$q(x)=n-s$ for all $x\in X$. For $0\le \ell\le s$, define
\[
\Sigma^\ell
=
\{x\in X\mid p(x)=s-\ell\}.
\]
Each path-component of $\Sigma^\ell$ is called an $\ell$-dimensional stratum of $X$.
\end{definition}

\begin{definition}[$C^r$-diffeomorphism~{\citep[Definition~3.1.26]{jongen2000nonlinear}}]\label{def:diffeo}
Let $X \subset \mathbb{R}^n$, $Y \subset \mathbb{R}^k$, and $f : X \to Y$ be a map.
\begin{enumerate}
    \item[\textbf{a.}] We say $f$ is of class $C^r$ if for each $x \in X$ there exists an open set $U \subset \mathbb{R}^n$ containing $x$ and a $C^r$-map $F : U \to \mathbb{R}^k$ that agrees with $f$ on $U \cap X$.
    \item[\textbf{b.}] We say $f$ is a $C^r$-diffeomorphism if $f$ is a bijection and both $f$ and $f^{-1}$ are of class $C^r$.
\end{enumerate}
\end{definition}

\subsection*{Covering maps}

The following notions from algebraic topology are used in the proof of Theorem~\ref{thm:global_continuation} to establish the global uniqueness of the continuation map.

\begin{definition}[Covering map]\label{def:covering}
    Let $p: E \to B$ be a continuous surjection between topological spaces. The map $p$ is called a covering map if every point $b \in B$ has an open neighborhood $U$ such that $p^{-1}(U)$ is a disjoint union of open sets in $E$, each of which is mapped homeomorphically onto $U$ by $p$.
\end{definition}

We will use the following standard fact from point-set topology.
\begin{lemma}[Tube lemma]\label{lem:tube}
    Let $X$ and $Y$ be topological spaces with $Y$ compact. If $N$ is an open subset of $X \times Y$ containing the slice $\{x_0\} \times Y$ for some $x_0 \in X$, then there exists an open neighborhood $W$ of $x_0$ in $X$ such that $W \times Y \subset N$. As a direct consequence, the projection $\pi_X: X \times Y \to X$ is a closed map.
\end{lemma}

A continuous map $p : E \to B$ is called proper if the preimage $p^{-1}(C)$ of every compact set $C \subset B$ is compact. The following result is well known in point-set topology.

\begin{proposition}[Proper local homeomorphism implies covering map~{\cite[Lemma~2]{ho1975note}}]\label{prop:proper_covering}
    Let $p: E \to B$ be a surjective proper local homeomorphism between locally compact Hausdorff spaces. Then $p$ is a covering map.
\end{proposition}

The key consequence we use is the following lifting property.

\begin{theorem}[Homotopy lifting property~{\citep[Proposition~1.30]{hatcher2002algebraic}}]\label{thm:homotopy_lifting}
    Let $p: E \to B$ be a covering map. Given a homotopy $H: [0,1] \times [0,1] \to B$ and a lift $\widetilde{H}(\cdot, 0): [0,1] \to E$ of the initial path (i.e., $p \circ \widetilde{H}(\cdot, 0) = H(\cdot, 0)$), there exists a unique continuous lift $\widetilde{H}: [0,1] \times [0,1] \to E$ such that $p \circ \widetilde{H} = H$.
\end{theorem}

\section{Localized Form of the Counterexample Arguments}\label{sec:local_version}

In Section~\ref{sec:nonprevalence}, Counterexamples~\ref{counterexample:licq2}, \ref{counterexample:licq3}, \ref{counterexample:scsc1}, and \ref{counterexample:sosc} are presented without repeatedly displaying
the compact sets \(K_{\rm in}\) and \(K_{\rm out}\). In this appendix,
we provide the localized version of the same argument.

\begin{enumerate}
    \item \textit{Choice of compact sets.}
    For each counterexample, we can choose compact sets
    \(K_{\rm in}\subset \operatorname{int}(K_{\rm out})\) such that the lower-level
    feasible sets are uniformly contained in
    \(\operatorname{int}(K_{\rm in})\), namely,
    \[
        \mathcal Y(x)\subset \operatorname{int}(K_{\rm in}),
        \qquad \forall x\in \mathcal X .
    \]
    We then compare the relevant invariant inside \(K_{\rm in}\). The goal is to show
    that no sufficiently small perturbation can make the \(K_{\rm in}\)-localized
    analogue of Statement~1 hold: the corresponding regularity condition cannot
    hold for the lower-level problem at every \(x\) when restricted to \(K_{\rm in}\).
    Here the restriction to \(K_{\rm in}\) is only a localization of the analysis,
    not an additional constraint in the lower-level problem.

    \item \textit{Endpoint stability.}
    The compact-local stability result is applied only at the two endpoint
    parameters \(x_1\) and \(x_2\) used in the counterexample. It gives two perturbation thresholds, and we choose the smaller one. Then, under any
    perturbation below this size, the relevant invariant inside \(K_{\rm in}\)
    is preserved at both endpoints. Therefore the endpoint mismatch of this
    invariant remains after perturbation.

    \item \textit{Localization by cutoff.}
    To apply the rigidity theorem, we use a smooth cutoff function which equals
    one on a neighborhood of \(K_{\rm in}\) and vanishes outside
    \(K_{\rm out}\). The
    cutoff problem agrees with the original perturbed problem on
    \(K_{\rm in}\). Moreover, by choosing the perturbation sufficiently small,
    the cutoff perturbed feasible set remains inside
    \(\operatorname{int}(K_{\rm in})\) for every \(x\). Thus the regularity
    condition inside \(K_{\rm in}\) for the original perturbed problem is the
    same as the global regularity condition for the cutoff perturbed problem.

    \item \textit{Contradiction with rigidity.}
    If the \(K_{\rm in}\)-localized analogue of Statement~1 held after
    perturbation, then the cutoff perturbed problem would satisfy the
    corresponding every-\(x\) regularity condition globally. The rigidity
    theorem would then force the endpoint invariants at \(x_1\) and \(x_2\) to
    be identical. This contradicts the endpoint mismatch preserved by the
    stability result. Hence the regularity condition must fail for some \(x\)
    inside \(K_{\rm in}\), and this failure is also a failure for the original
    perturbed problem.
\end{enumerate}

\section{Detailed proof}

\subsection{Proof of Theorem~\ref{thm:main_licq}}\label{proof:thm:main_licq}

In \citet[page~87]{jongen2000nonlinear}, the authors show that if LICQ holds
for the lower-level problem, then the lower-level feasible set
\(\mathcal{Y}(x)\) is an MGB. Therefore, to
complete the proof, it suffices to show that \(\mathcal{Y}(x_1)\) and
\(\mathcal{Y}(x_2)\) are diffeomorphic for any \(x_1,x_2\in\mathcal X\). Once
this is established, the desired result follows directly from
\citet[Corollary~3.1.29]{jongen2000nonlinear}, which guarantees that
diffeomorphisms between MGBs preserve stratification.

We next recall the one-parameter result of
\citet[Theorem~10.4.5]{jongen2000nonlinear}. Consider a one-parameter family of
constraint sets
\[
    \mathcal Z(t)\coloneq \{y\in\mathbb R^m: H_i(t,y)\le 0,\ i=1,\ldots,k\},
    \qquad t\in[a,b].
\]
If each \(\mathcal Z(t)\) satisfies LICQ and
\(
    \{(t,y)\in[a,b]\times\mathbb R^m:y\in\mathcal Z(t)\}
\)
is compact, then \(\mathcal Z(t_1)\) and \(\mathcal Z(t_2)\) are diffeomorphic
for any \(t_1,t_2\in[a,b]\).

We now apply this result to the present higher-dimensional set
\(\mathcal X\subset\mathbb R^n\). Since \(\mathcal X\) is a connected MGB, it is
locally path-connected and hence path-connected. Thus, for any
\(x_1,x_2\in\mathcal X\), there exists a continuous path connecting them.
Because the image of this path is compact, it can be covered by finitely many
local coordinate charts of \(\mathcal X\). Within these charts, the path can be
replaced by a piecewise smooth path \(\gamma:[0,1]\to\mathcal X\) with
\(\gamma(0)=x_1\) and \(\gamma(1)=x_2\).

On each smooth segment of \(\gamma\), define the one-parameter family
\[
    \mathcal Z(t)\coloneq \mathcal Y(\gamma(t))
    =\{y\in\mathbb R^m:h_i(\gamma(t),y)\le 0,\ i=1,\ldots,k\}.
\]
Since LICQ holds for every \(x\in\mathcal X\), LICQ holds for
\(\mathcal Z(t)\) for all \(t\). Moreover, the uniform boundedness of
\(\mathcal Y(x)\), together with the continuity of the constraints, ensures
the compactness hypothesis above on each smooth segment. Therefore, we can
apply the theorem on each segment of \(\gamma\). Composing the resulting
diffeomorphisms over the finitely many segments, we conclude that
\(\mathcal Y(x_1)\) and \(\mathcal Y(x_2)\) are diffeomorphic.

\subsection{Proof of Theorem~\ref{thm:stab_strat_perturb}}\label{proof:thm:stab_strat_perturb}

We first recall the structural stability theorem from \citet{jongen2000nonlinear}, which is the main tool used in the proof.

\begin{theorem}[\citet{jongen2000nonlinear}, Theorem 6.3.1]\label{thm:structural_stability}
Let \(M[h,g]\) be a compact, regular constraint set and let \(f|_{M}\) be nondegenerate and separating. Then there exists a \(C^2\)-neighborhood \(\mathcal O\) of
\((f,h_1,\ldots,h_m,g_1,\ldots,g_s)\) such that, for every
\((\widetilde f,\widetilde h_1,\ldots,\widetilde h_m,\widetilde g_1,\ldots,\widetilde g_s)\in\mathcal O\), there exist \(C^\infty\)-diffeomorphisms
\(\widetilde\phi:\mathbb R^n\to\mathbb R^n\) and
\(\widetilde\psi:\mathbb R\to\mathbb R\) satisfying:
\begin{enumerate}
    \item \(\widetilde\phi\), respectively \(\widetilde\psi\), equals the identity outside a compact set in \(\mathbb R^n\), respectively \(\mathbb R\);
    \item \(\widetilde\phi\) maps \(M[h,g]\) onto \(\widetilde M=M[\widetilde h,\widetilde g]\);
    \item \(\widetilde f|_{\widetilde M}=\widetilde\psi\circ f\circ \widetilde\phi^{-1}_{|\widetilde M}\).
\end{enumerate}
\end{theorem}

We first recall the structural stability theorem from \citet{jongen2000nonlinear}, 
which is the main tool used in the proof.

\paragraph{Step 1 (Setup):}
We apply Theorem~\ref{thm:structural_stability} to the lower-level feasible set 
$\mathcal{Y}(x)$ at a fixed $x \in \mathcal X$. In the notation of 
\citet{jongen2000nonlinear}, equality constraints are denoted by $h_i = 0$ and 
inequality constraints by $g_j \le 0$. In our setting there are no equality 
constraints, and the inequality constraints are $h_i(x, \cdot) \le 0$. Since 
$\mathcal{Y}(x)$ is closed and bounded, it is compact. Combined with LICQ at 
every feasible point, this makes $\mathcal{Y}(x)$ a compact regular constraint 
set in the sense of Theorem~\ref{thm:structural_stability}.

It remains to choose an auxiliary objective \(f\) such that
\(f|_{\mathcal Y(x)}\) is nondegenerate and separating. This auxiliary objective
\(f\) is not the lower-level objective \(g\) in our bilevel problem. It is
introduced only to satisfy the requirements of
Theorem~\ref{thm:structural_stability}, and its particular choice is irrelevant
for our purpose, since we only use the feasible-set diffeomorphism conclusion (2)
of Theorem~\ref{thm:structural_stability}.

Here nondegenerate means that every stationary point is nondegenerate in the
sense of \citet[Definition~3.2.14]{jongen2000nonlinear}, and separating means
that distinct stationary points of \(f|_{\mathcal Y(x)}\) have distinct function
values. Such a choice is possible by a genericity result on manifolds with
generalized boundary. Indeed, since \(\mathcal Y(x)\) is a compact regular
constraint set under LICQ, it is a compact MGB. By
\citet[Theorem~7.1.17]{jongen2000nonlinear}, the set of smooth functions whose
restriction to a compact MGB is nondegenerate and separating is \(C^k\)-open
and \(C^k\)-dense for every \(k\ge 2\). Hence we may choose a smooth auxiliary
objective \(f\) such that \(f|_{\mathcal Y(x)}\) is nondegenerate and separating.

\paragraph{Step 2 (Cutoff to a strong $C^2$-perturbation):}
The neighborhood $\mathcal{O}$ in Theorem~\ref{thm:structural_stability} is 
taken in the strong $C^2$-topology of \citet[Definition~6.1.8]{jongen2000nonlinear}. 
For $F \in C^2(\mathbb R^m, \mathbb R)$, a basic neighborhood has the form
\begin{align}\label{eq:strongtopo}
    \mathcal V_{\rho, F}^2
\coloneq  \left\{
\widetilde F \in C^2(\mathbb R^m, \mathbb R) :
|D^\alpha \widetilde F(y) - D^\alpha F(y)| < \rho(y),\
y \in \mathbb R^m,\ |\alpha| \le 2
\right\},
\end{align}
where $\rho : \mathbb R^m \to (0, \infty)$ is continuous, and the topology on 
finite collections is the product topology.

Our requirement controls the perturbation only on the compact set \(K_{\rm out}\) in the
\(C^2(K_{\rm out})\)-norm. To pass to the strong \(C^2\)-topology, we use a cutoff. Choose
\(\chi \in C_c^\infty(\operatorname{int}(K_{\rm out}))\) such that \(\chi = 1\) on a
neighborhood of \(K_{\rm in}\). Here \(C_c^\infty(\operatorname{int}(K_{\rm out}))\) denotes the space of smooth functions with compact support contained in \(\operatorname{int}(K_{\rm out})\). Define
\[
\widehat h_i(x,y) \coloneq  h_i(x,y) + \chi(y) r_i(x,y), \qquad i = 1, \ldots, k.
\]
Then \(\widehat h_i = \widetilde h_i\) on \(K_{\rm in}\) and \(\widehat h_i = h_i\) outside
\(K_{\rm out}\). The cutoff perturbation \(\chi r_i\) is supported in \(K_{\rm out}\), where the
tolerance function \(\rho\) defining \(\mathcal O\) in \eqref{eq:strongtopo} has a positive lower bound.
Hence sufficiently small \(\|r_i(x, \cdot)\|_{C^2(K_{\rm out})}\) ensures that the cutoff
data \((f, \widehat h_1(x, \cdot), \ldots, \widehat h_k(x, \cdot))\) lies in
\(\mathcal O\).

\paragraph{Step 3 (The cutoff feasible set lies in \(K_{\rm in}\)):}
Since \(\mathcal{Y}(x) \subset \operatorname{int}(K_{\rm in})\), the compact set
\(K_{\rm out} \setminus \operatorname{int}(K_{\rm in})\) is disjoint from \(\mathcal{Y}(x)\), so
there exists \(\eta > 0\) with
\[
\max_i h_i(x, y) \ge \eta, \qquad y \in K_{\rm out} \setminus \operatorname{int}(K_{\rm in}).
\]
For sufficiently small perturbations, \(\max_i \widehat h_i(x, y) > 0\) on this
set, so the cutoff feasible set contains no point in
\(K_{\rm out} \setminus \operatorname{int}(K_{\rm in})\). Outside \(K_{\rm out}\) the cutoff perturbation
vanishes, and the original feasible set is already contained in \(K_{\rm in}\).
Therefore \(\widehat{\mathcal{Y}}(x) \subseteq K_{\rm in}\). Since \(\chi = 1\) on a
neighborhood of \(K_{\rm in}\), the cutoff and original perturbed constraints coincide
there, so
\[
\widehat{\mathcal{Y}}(x) = \widetilde{\mathcal{Y}}(x) \cap K_{\rm in}.
\]

\paragraph{Step 4 (Conclusion):}
By Theorem~\ref{thm:structural_stability} applied to the auxiliary objective \(f\)
and to the constraint functions \(h_i(x,\cdot)\), \(i=1,\ldots,k\), with the cutoff
constraints \(\widehat h_i(x,\cdot)\), \(\mathcal Y(x)\) is diffeomorphic to
\(\widehat{\mathcal Y}(x)=\widetilde{\mathcal Y}(x)\cap K_{\rm in}\). By
\citet[Corollary~3.1.29]{jongen2000nonlinear}, this diffeomorphism preserves
the natural stratification.

\subsection{Proof of Theorem~\ref{thm:main_scsc}}\label{proof:thm:main_scsc}
Let $x \mapsto y^\ast(x)$ be a continuous trajectory of KKT points. By Proposition~\ref{prop:multiplier_continuity}, the associated multiplier vector $\lambda^\ast(x)$ is unique and continuous along this trajectory. We first show that the active constraint index set is locally constant. Fix any $\overline{x} \in \mathcal{X}$ and let $\overline{\mathcal{J}} \coloneq  \{i : h_i(\overline{x}, y^\ast(\overline{x})) = 0\}$ be the active set at $\overline{x}$. By the SCSC assumption, we have $\lambda_i^\ast(\overline{x}) > 0$ for all $i \in \overline{\mathcal{J}}$ and $h_j(\overline{x}, y^\ast(\overline{x})) < 0$ for all $j \notin \overline{\mathcal{J}}$. Since both $\lambda^\ast(x)$ and the map $x \mapsto h_i(x, y^\ast(x))$ are continuous along the trajectory, there exists a neighborhood $U$ of $\overline{x}$ such that:$$\lambda_i^\ast(x) > 0 \quad (i \in \overline{\mathcal{J}}) \quad \text{and} \quad h_j(x, y^\ast(x)) < 0 \quad (j \notin \overline{\mathcal{J}})$$for all $x \in U$. This ensures that the active set $\mathcal{J}(x,y^\ast(x))$ remains identical to $\overline{\mathcal{J}}$ for all $x \in U$, proving that the mapping $x \mapsto \mathcal{J}(x,y^\ast(x))$ is locally constant. Since $\mathcal{X}$ is connected, any locally constant mapping on $\mathcal{X}$ must be globally constant. Consequently, the active constraint index set $\mathcal{J}(x,y^\ast(x))$ is invariant along the entire continuous KKT trajectory.

\subsection{Proof of Theorem~\ref{thm:stab_scsc}}\label{proof:thm:stab_scsc}
This result follows as a special case of
Theorem~\ref{thm:stab_strat_lmins_perturb_noM}. Indeed, since
\(g(x,\cdot)\) is strongly convex and \(\mathcal Y(x)\) is convex, the original
lower-level problem has a unique local minimizer, which is also its unique
global minimizer. Moreover, strong convexity together with LICQ implies the
nonsingularity of the KKT matrix at this minimizer. Hence the assumptions of
Theorem~\ref{thm:stab_strat_lmins_perturb_noM} are satisfied for the unique
local minimizer of the original problem.

By the proof of Theorem~\ref{thm:stab_strat_lmins_perturb_noM}, for sufficiently small
perturbations measured in \(C^2(K_{\rm out})\), the number of local minimizers
in each stratum inside \(K_{\rm in}\) is preserved. Since the original problem
has exactly one local minimizer, the perturbed problem has exactly one local
minimizer in \(K_{\rm in}\), and this minimizer lies on the same stratum as
\(y^*(x)\). Therefore, we have $\mathcal J(x,\widetilde y^*(x))=\mathcal J(x,y^*(x))$. The proof is completed.

\subsection{Proof of Theorem~\ref{thm:main_sosc}}\label{proof:thm:main_sosc}

We prove the two parts separately.

\paragraph{Step 1 (Local continuation):}
Fix any $\overline x\in\mathcal X$ and let $\overline y$ be any local minimizer of the lower-level problem at $\overline x$, with corresponding
multiplier $\overline\lambda\ge 0$ satisfying the KKT conditions. By the standing assumptions, LICQ, SCSC, and uniform SOSC
hold at $(\overline x,\overline y)$. Denote the active set by
\[
\overline A\coloneq \{i:\ h_i(\overline x,\overline y)=0\}.
\]
By SCSC, we have $\overline\lambda_i>0$ for $i\in\overline A$ and $\overline\lambda_i=0$ for $i\notin\overline A$. Consider the reduced KKT
system in variables $(y,\lambda_{\overline A})$:
\begin{equation}\label{eq:reducedKKT_noeq_rewrite}
\Psi(x,y,\lambda_{\overline A})
\coloneq 
\begin{pmatrix}
\nabla_y g(x,y)+\nabla_y h_{\overline A}(x,y)^\top \lambda_{\overline A}\\
h_{\overline A}(x,y)
\end{pmatrix}
=0 .
\end{equation}
Under LICQ and uniform SOSC at $(\overline x,\overline y)$, the Jacobian of $\Psi$ with respect to $(y,\lambda_{\overline A})$ is nonsingular at
$(\overline x,\overline y,\overline\lambda_{\overline A})$. Hence the Implicit Function Theorem yields a neighborhood $U_{\overline x}\subset\mathcal X$
and a unique $C^1$ mapping
\[
x\ \mapsto\ (y_{\overline y}(x),\lambda_{\overline y,\overline A}(x)) \qquad (x\in U_{\overline x}),
\]
satisfying $(y_{\overline y}(\overline x),\lambda_{\overline y,\overline A}(\overline x))=(\overline y,\overline\lambda_{\overline A})$ and
\begin{equation*}
\Psi(x,y_{\overline y}(x),\lambda_{\overline y,\overline A}(x)))
=
\begin{pmatrix}
\nabla_y g(x,y_{\overline y}(x))+\nabla_y h_{\overline A}(x,y_{\overline y}(x))^\top \lambda_{\overline y,\overline A}(x)\\
h_{\overline A}(x,y_{\overline y}(x))
\end{pmatrix}
=0 
\end{equation*}
holds for all $x\in U_{\overline x}$. By continuity and SCSC at $\overline x$, after possibly shrinking
$U_{\overline x}$ we further have $\lambda_{\overline y,\overline A}(x)\in\mathbb{R}^{|\overline A|}_{>0}$ for all $x\in U_{\overline x}$ and
$h_i(x,y_{\overline y}(x))<0$ for all $i\notin\overline A$. Therefore, along this branch the active set remains exactly $\overline A$,
and $x\mapsto y_{\overline y}(x)$ is a $C^1$ trajectory of KKT points that stays on a fixed stratum of $\mathcal{Y}(x)$.

Moreover, the continued point $y_{\overline{y}}(x)$ remains a local minimizer
for all $x \in U_{\overline{x}}$. Indeed, since the active set $\overline{A}$
is constant along the continuation, the critical cone satisfies
\[
\mathcal{C}(x, y_{\overline{y}}(x))
= \ker \nabla_y h_{\overline{A}}(x, y_{\overline{y}}(x))
\]
for all $x \in U_{\overline{x}}$. Under LICQ,
$\nabla_y h_{\overline{A}}(x, y_{\overline{y}}(x))$ has full row rank,
so $\ker \nabla_y h_{\overline{A}}(x, y_{\overline{y}}(x))$ has constant
dimension and admits a smooth basis $\{v_1(x), \ldots, v_r(x)\}$.
At $(\overline{x}, \overline{y})$, uniform SOSC gives
\[
d^\top \nabla_{yy}^2 \mathcal{L}(\overline{x}, \overline{y},
\overline{\lambda}_{\overline{A}})\, d
\geq \sigma \|d\|^2
\quad \text{for all } d \in \mathcal{C}(\overline{x}, \overline{y}).
\]
Expressing $d = V(x)\alpha$ where $V(x) = [v_1(x), \ldots, v_r(x)]$, the
minimum of $d^\top \nabla_{yy}^2 \mathcal{L}\, d$ over $\|d\| = 1$,
$d \in \mathcal{C}(x, y_{\overline{y}}(x))$ equals the smallest generalized
eigenvalue of $(V(x)^\top \nabla_{yy}^2 \mathcal{L}(x)\, V(x),\;
V(x)^\top V(x))$, which is a continuous function of $x$. Since this
quantity equals $\sigma > 0$ at $\overline{x}$, it remains positive in a
sufficiently small neighborhood $U_{\overline{x}}$. This implies that
$y_{\overline{y}}(x)$ satisfies SOSC, and hence is a local minimizer,
for all $x \in U_{\overline{x}}$.

Distinct local minimizers yield locally disjoint branches. Indeed, the Implicit Function Theorem provides local uniqueness for solutions of \eqref{eq:reducedKKT_noeq_rewrite} near
$(\overline y,\overline\lambda_{\overline A})$. Consequently, if $\overline y_1\neq \overline y_2$ are two distinct local minimizers at the same
$\overline x$, then for $x$ sufficiently close to $\overline x$ the two branches
$y_{\overline y_1}(\cdot)$ and $y_{\overline y_2}(\cdot)$ cannot intersect.

\paragraph{Step 2 (The active-index-wise number of local minimizers is constant):}
For each active index set \(\mathcal J \subseteq \{1,\ldots,k\}\), define
\[
S_{\mathcal J}(x)
\coloneq 
\{y\in \mathcal Y(x): \mathcal J(x,y)=\mathcal J\},
\qquad
\mathcal J(x,y)\coloneq \{i:h_i(x,y)=0\}.
\]
Let
\[
N_{\mathcal J}(x)
\coloneq 
\#\{y\in S_{\mathcal J}(x): y \text{ is a local minimizer of the lower-level problem at } x\}.
\]
By assumption, the total number of local minimizers is finite for every \(x\in \mathcal X\), so
\(N_{\mathcal J}(x)\) is finite for every \(\mathcal J\) and every \(x\). 

We first show that \(N_{\mathcal J}\) is lower semicontinuous. Fix \(\overline x\in\mathcal X\), and let $\overline y_1,\ldots,\overline y_q$ be all local minimizers of the lower-level problem at \(\overline x\) that lie in
\(S_{\mathcal J}(\overline x)\), where \(q=N_{\mathcal J}(\overline x)\). By \textbf{Step}~1, each \(\overline y_\ell\)
admits a locally unique \(C^1\) continuation \(y_\ell(x)\) defined for \(x\) in a neighborhood of
\(\overline x\), and this continuation remains a local minimizer with the same active index set
\(\mathcal J\). Shrinking the neighborhood if necessary, the continuations
\(y_1(x),\ldots,y_q(x)\) remain distinct. Hence, for all \(x\) sufficiently close to \(\overline x\), we have $
N_{\mathcal J}(x)\ge q=N_{\mathcal J}(\overline x)$.
Thus \(N_{\mathcal J}\) is lower semicontinuous. 

We next prove upper semicontinuity. Suppose, to the contrary, that upper semicontinuity fails at
\(\overline x\). Then there exist \(x_\ell\to \overline x\) and local minimizers
\(y_\ell\in S_{\mathcal J}(x_\ell)\) such that \(y_\ell\) does not belong to any of the local
continuation neighborhoods generated by the local minimizers in \(S_{\mathcal J}(\overline x)\).
Since the feasible sets \(\mathcal Y(x)\) are uniformly bounded, after passing to a subsequence, we
may assume that $y_\ell\to \overline y.$ By Lemma~\ref{lem:closedness_minimizers}, the graph of local minimizers is closed. Hence \(\overline y\) is a local minimizer of the
lower-level problem at \(\overline x\). We further claim that \(\overline y\in S_{\mathcal J}(\overline x)\). Indeed, by the multiplier-continuity argument used in the proof of
Lemma~\ref{lem:closedness_minimizers} around \eqref{eq:stabilityofactiveset},
the KKT multipliers associated with \(y_\ell\) converge, after passing to a
subsequence, to the unique KKT multiplier associated with \(\overline y\). Since
SCSC holds at \((\overline x,\overline y)\), this convergence, together with the
continuity of the constraints, implies that the active index set is unchanged
in the limit. Therefore, for all sufficiently large \(\ell\),
\[
\mathcal J(\overline x,\overline y)=\mathcal J(x_\ell,y_\ell)=\mathcal J .
\]
Thus \(\overline y\in S_{\mathcal J}(\overline x)\). Now, since \(\overline y\) is a local minimizer in \(S_{\mathcal J}(\overline x)\), \textbf{Step}~1 gives a locally unique
\(C^1\) continuation of \(\overline y\) near \(\overline x\), and this continuation remains on the same active-index
stratum \(S_{\mathcal J}\). Since \(x_\ell\to\overline x\) and \(y_\ell\to\overline y\), the local uniqueness of
this continuation implies that, for all sufficiently large \(\ell\), the point \(y_\ell\) must lie in the
continuation neighborhood of \(\overline y\). This contradicts the choice of \(y_\ell\). Hence
\(N_{\mathcal J}\) is upper semicontinuous.

Combining lower and upper semicontinuity, \(N_{\mathcal J}\) is locally constant on \(\mathcal X\).
Since \(\mathcal X\) is connected and \(N_{\mathcal J}\) is integer-valued, \(N_{\mathcal J}\) is constant
on \(\mathcal X\). Since \(\mathcal J\subseteq\{1,\ldots,k\}\) was arbitrary, the number of local
minimizers on each active-index stratum is constant across all \(x\in\mathcal X\).

\subsection{Proof of Theorem~\ref{thm:stab_strat_lmins_perturb_noM}}\label{proof:thm:stab_strat_lmins_perturb_noM}

To prove Theorem~\ref{thm:stab_strat_lmins_perturb_noM}, we need the following auxiliary lemmas.

\begin{lemma}\label{lemma:stab-zero}
Let $F:\mathbb{R}^m\to\mathbb{R}^m$ be a $C^1$ mapping and let $z^\ast\in\mathbb{R}^m$
satisfy
\[
    F(z^\ast)=0,\qquad \det D F(z^\ast)\neq 0.
\]
Then there exist constants $r>0$ and $\eta>0$ such that the following holds: for every
$C^1$ mapping $R:\mathbb{R}^m\to\mathbb{R}^m$ with
\[
    \sup_{z\in B_r(z^\ast)}\|R(z)\|\le \eta,
    \qquad
    \sup_{z\in B_r(z^\ast)}\|D R(z)\|\le \eta,
\]
the perturbed equation
\[
    F(z)+R(z)=0
\]
admits a unique solution $z_R\in B_r(z^\ast)$, where $B_r(z^\ast)$ is a ball centered at $z^\ast$ with radius $r$.
\end{lemma}
    
\begin{proof}
Let $\sigma\coloneq \sigma_{\min}\bigl(D F(z^\ast)\bigr)>0$ and set
$A\coloneq D F(z^\ast)^{-1}$, so that $\|A\|=1/\sigma$. By continuity of $D F$ there exists
$r>0$ such that
\begin{equation}\label{eq:DF-close}
    \|D F(z)-D F(z^\ast)\|\le \frac{\sigma}{4}
    \qquad\text{for all }z\in B_r(z^\ast).
\end{equation}
By the Mean Value Theorem, for any $z\in B_r(z^\ast)$ we have
\[
    F(z)-F(z^\ast)-D F(z^\ast)(z-z^\ast)
    = \int_0^1\bigl(D F(z^\ast+t(z-z^\ast))-D F(z^\ast)\bigr)(z-z^\ast)\,dt,
\]
hence, using \eqref{eq:DF-close}, we have
\begin{equation}\label{eq:Taylor-F}
    \|F(z)-D F(z^\ast)(z-z^\ast)\|
    \le \frac{\sigma}{4}\,\|z-z^\ast\|
    \qquad\forall z\in B_r(z^\ast).
\end{equation}

Fix $0<\eta<\sigma^2 r/8$, and let $R$ satisfy the smallness conditions in the lemma.
Consider the closed ball $B_r\coloneq B_r(z^\ast)$ and define
\[
    \mathcal T:B_r\to\mathbb{R}^m,\qquad
    \mathcal T(z)\coloneq z-A\bigl(F(z)+R(z)\bigr).
\]
\paragraph{Step 1 ($\mathcal T$ maps $B_r$ into itself):}
For any $z\in B_r$, using $F(z^\ast)=0$ and \eqref{eq:Taylor-F},
\begin{align*}
    \mathcal T(z)-z^\ast
    &= z-z^\ast-A\bigl(F(z)-F(z^\ast)\bigr)-A R(z) \\
    &= -A\bigl(F(z)-D F(z^\ast)(z-z^\ast)\bigr)-A R(z),
\end{align*}
It follows that
\[
    \|\mathcal T(z)-z^\ast\|
    \le \|A\|\cdot \frac{\sigma}{4}\,\|z-z^\ast\|
         + \|A\|\cdot \|R(z)\|
    \le \frac{1}{\sigma}\cdot \frac{\sigma}{4}\,r + \frac{1}{\sigma}\,\eta
    = \frac{r}{4} + \frac{\eta}{\sigma}.
\]
By our choice of $\eta<\sigma^2 r/8$, we have $\eta/\sigma<r/8$, hence
\[
    \|\mathcal T(z)-z^\ast\|\le \frac{r}{4}+\frac{r}{8} = \frac{3}{8}r<r,
\]
so $\mathcal T(B_r)\subseteq B_r$.
\paragraph{Step 2 ($\mathcal T$ is a contraction on $B_r$):}
For any $z,z'\in B_r$, by the Mean Value Theorem,
\[
    \mathcal T(z)-\mathcal T(z')
    =\Bigl(I-A\int_0^1 D(F+R)\bigl(z'+t(z-z')\bigr)\,dt\Bigr)(z-z').
\]
Therefore,
\begin{align*}
    \|\mathcal T(z)-\mathcal T(z')\|
    &\le \sup_{\xi\in B_r}\|I-A D(F+R)(\xi)\|\,\|z-z'\| \\
    &= \sup_{\xi\in B_r}\|A\bigl(D F(z^\ast)-D F(\xi)-D R(\xi)\bigr)\|\,\|z-z'\|.
\end{align*}
Using \eqref{eq:DF-close} and $\sup_{B_r}\|D R\|\le\eta$, we get
\[
    \|\mathcal T(z)-\mathcal T(z')\|
    \le \frac{1}{\sigma}\Bigl(\frac{\sigma}{4}+\eta\Bigr)\|z-z'\|.
\]
Choosing $\eta\le \sigma/4$, we obtain
\[
    \|\mathcal T(z)-\mathcal T(z')\|\le \frac{1}{2}\,\|z-z'\|
    \qquad\forall z,z'\in B_r,
\]
so $\mathcal T$ is a contraction on $B_r$.

By the Banach fixed-point theorem, there exists a unique $z_R\in B_r$ such that
$\mathcal T(z_R)=z_R$. By definition of $\mathcal T$, this is equivalent to
$F(z_R)+R(z_R)=0$. Uniqueness of the fixed point yields uniqueness of the zero
of $F+R$ in $B_r(z^\ast)$, which completes the proof.
\end{proof}

\begin{lemma}\label{lemma:near_licq}
Fix $x$ and write
\[
\mathcal{Y}\coloneq \mathcal{Y}(x)=\{y\in\mathbb R^m:h_i(x,y)\le 0,\ i=1,\ldots,k\}.
\]
Let $K\subset\mathbb R^m$ be compact with $\mathcal{Y}\subset \operatorname{int}(K)$.
Assume that LICQ holds at every point of $\mathcal{Y}$. Then there exist constants
$\rho>0$ and $\sigma_\ast>0$ such that for every $y\in K$ satisfying
\[
\max_{1\le i\le k} h_i(x,y)\le \rho,
\]
and every nonempty index set $\mathcal{J}\subseteq\{1,\ldots,k\}$ satisfying
\[
|h_j(x,y)|\le \rho,\qquad j\in \mathcal{J},
\]
we have
\[
\sigma_{\min}\big(\nabla_y h_\mathcal{J}(x,y)\big)\ge \sigma_\ast .
\]
Here $\nabla_y h_\mathcal{J}(x,y)$ denotes the matrix whose rows are
$\nabla_y h_j(x,y)^\top$, $j\in \mathcal{J}$.
\end{lemma}

\begin{proof}
Suppose the claim is false. Then there exist $y_n\in K$ and nonempty
sets $\mathcal{J}_n\subseteq\{1,\ldots,k\}$ such that
\[
\max_i h_i(x,y_n)\le \frac1n,\qquad
|h_j(x,y_n)|\le \frac1n\quad j\in \mathcal{J}_n,
\]
but
\[
\sigma_{\min}\big(\nabla_y h_{\mathcal{J}_n}(x,y_n)\big)<\frac1n.
\]
Since there are finitely many index sets, passing to a subsequence gives
$\mathcal{J}_n=\mathcal{J}$. Since $K$ is compact, passing to another subsequence gives
$y_n\to \overline y\in K$. By continuity,
\[
h_i(x,\overline y)\le 0,\quad i=1,\ldots,k,
\]
so $\overline y\in \mathcal{Y}$. Moreover, for every $j\in \mathcal{J}$,
\[
h_j(x,\overline y)=0.
\]
Hence $\mathcal{J}\subseteq \mathcal{J}(x,\overline y)$. LICQ at $\overline y$ implies that
$\nabla_y h_{\mathcal J}(x,\overline y)$ has full row rank, so
\[
\sigma_{\min}\big(\nabla_y h_{\mathcal J}(x,\overline y)\big)>0.
\]
This contradicts the continuity of singular values and
\[
\sigma_{\min}\big(\nabla_y h_{\mathcal J}(x,y_n)\big)<\frac1n.
\]
Therefore, we finish the proof.
\end{proof}

\noindent\textbf{Detailed proof of Theorem~\ref{thm:stab_strat_lmins_perturb_noM}:} Fix $x\in\mathcal X$ throughout. For each $y\in\mathcal Y(x)$ denote the active index set
$\mathcal J(x,y)\coloneq \{j\in I: h_j(x,y)=0\}$, and let $L(x,y,\lambda)\coloneq g(x,y)+\lambda^\top h_{\mathcal J(x,y)}(x,y)$
be the Lagrangian associated with the active constraints at $(x,y)$. Throughout the proof, all perturbations are measured in the local norm
\[
\|s(x,\cdot)\|_{C^2(K_{\rm out})}
+
\max_i\|r_i(x,\cdot)\|_{C^2(K_{\rm out})}.
\]
Since \(K_{\rm in}\subset \operatorname{int}(K_{\rm out})\), smallness in this norm
implies the \(C^2\)-smallness on \(K_{\rm in}\) required in the arguments below.
We will repeatedly reduce the perturbation size; the final constant
\(\epsilon(x,K_{\rm in},K_{\rm out})\) is chosen as the minimum of the finitely
many thresholds obtained in the proof.

\paragraph{Step 1 (Finiteness of KKT pairs):}
Let \(\mathcal K(x)\) denote the set of KKT pairs \((y,\lambda)\) of the lower-level
problem. We first show that \(\mathcal K(x)\) is finite.

Fix a KKT pair \((\overline y,\overline\lambda)\in \mathcal K(x)\), and let
\(A\coloneq \mathcal J(x,\overline y)\) be the active index set at \(\overline y\). We write \(\overline\lambda_A \in \mathbb{R}^{|A|}\) for the subvector of \(\overline\lambda\) indexed by \(A\). Consider the reduced KKT system
\[
\Psi_A(y,\lambda_A)\coloneq 
\begin{pmatrix}
\nabla_y g(x,y)+\nabla_y h_A(x,y)^\top \lambda_A\\
h_A(x,y)
\end{pmatrix}=0.
\]
Its Jacobian with respect to \((y,\lambda_A)\) at
\((\overline y,\overline\lambda_A)\) is
\[
\begin{pmatrix}
\nabla^2_{yy}L(x,\overline y,\overline\lambda) &
\big(\nabla_y h_A(x,\overline y)\big)^\top\\
\nabla_y h_A(x,\overline y) & 0
\end{pmatrix},
\]
which is nonsingular by assumption. Hence, by the Implicit Function Theorem,
\((\overline y,\overline\lambda_A)\) is the unique local solution of this reduced KKT system.

Moreover, by SCSC at \((\overline y,\overline\lambda)\), we have
\(\overline\lambda_i>0\) for all \(i\in A\), and \(h_i(x,\overline y)<0\) for all
\(i\notin A\). Therefore, for any KKT pair \((y,\lambda)\) sufficiently close
to \((\overline y,\overline\lambda)\), we have \(\lambda_i>0\) for all \(i\in A\), and
\(h_i(x,y)<0\) for all \(i\notin A\). By complementarity, this forces
\(h_i(x,y)=0\) for all \(i\in A\), and hence \(\mathcal J(x,y)=A\). Thus
\((y,\lambda_A)\) solves the same reduced KKT system. The local uniqueness above
then implies \((y,\lambda)=(\overline y,\overline\lambda)\). Hence every KKT pair is isolated.

We next show that \(\mathcal K(x)\) is bounded. Since \(\mathcal Y(x)\) is
bounded and closed, it is compact, so the primal components of all KKT pairs
are bounded. Suppose, for contradiction, that the multiplier components are
unbounded. Then there exists a sequence
\((y_s,\lambda_s)\in\mathcal K(x)\) such that \(\|\lambda_s\|\to\infty\).
Passing to a subsequence and relabeling the constraints if necessary, we may
assume that $(\lambda_s)_1=\max_i(\lambda_s)_i\to\infty.$ Since \(y_s\in\mathcal Y(x)\) and \(\mathcal Y(x)\) is compact, passing to a
further subsequence gives \(y_s\to y^\infty\in\mathcal Y(x)\). The stationarity
condition gives
\[
\nabla_y g(x,y_s)+\sum_{i=1}^k(\lambda_s)_i\nabla_y h_i(x,y_s)=0.
\]
Dividing by \((\lambda_s)_1\) and passing to a further subsequence, we obtain
\[
\sum_{i=1}^k \mu_i \nabla_y h_i(x,y^\infty)=0,
\]
where
\[
\mu_i\coloneq \lim_{s\to\infty}\frac{(\lambda_s)_i}{(\lambda_s)_1},
\qquad \mu_1=1.
\]
If \(\mu_i>0\), then \((\lambda_s)_i>0\) for all sufficiently large \(s\), and
hence complementarity gives \(h_i(x,y_s)=0\). Taking limits yields
\(h_i(x,y^\infty)=0\). Therefore, all indices with \(\mu_i>0\) belong to
\(\mathcal J(x,y^\infty)\). Since \(\mu_1=1\), the above identity is a nontrivial linear
dependence among the active constraint gradients at \(y^\infty\), contradicting
LICQ. Hence the multipliers are bounded, and \(\mathcal K(x)\) is bounded.

The set $\mathcal{K}(x)$ consists of pairs $(y,\lambda)$ satisfying the KKT system: the stationarity equation $\nabla_y g(x,y) + \sum_i \lambda_i \nabla_y h_i(x,y) = 0$, the primal feasibility inequalities $h_i(x,y) \leq 0$, the dual feasibility inequalities $\lambda_i \geq 0$, and the complementarity equations $\lambda_i h_i(x,y) = 0$ for $i = 1,\dots,k$. Since all these equations and inequalities are continuous in $(y,\lambda)$, the set $\mathcal{K}(x)$ is closed. Combined with the boundedness established above, $\mathcal{K}(x)$ is compact. A compact set in which every point is isolated must be finite. Hence $\mathcal{K}(x)$ is finite. Consequently, its projection onto the
primal space is also finite, and hence the set of local minimizers is finite.

\paragraph{Step 2 (Local persistence of each KKT pair):}
Fix any KKT pair \((\overline y,\overline\lambda)\in\mathcal K(x)\), and set
\(A\coloneq \mathcal J(x,\overline y)\). Consider the reduced KKT system written as a square
equation in \(z\coloneq (y,\lambda_A)\in\mathbb R^{m+|A|}\):
\[
F_A(z)\coloneq 
\begin{pmatrix}
\nabla_y g(x,y)+(\nabla_y h_A(x,y))^\top \lambda_A\\
h_A(x,y)
\end{pmatrix}=0 .
\]
The nonsingularity of the KKT matrix at \((\overline y,\overline\lambda)\) is precisely
\(\det DF_A(\overline z)\neq 0\), where
\(\overline z\coloneq (\overline y,\overline\lambda_A)\).

Now perturb the objective and constraints by \(C^2\) functions with sufficiently
small \(C^2\)-norm. This induces a perturbed reduced KKT mapping
\(\widetilde F_A\) on the same active set \(A\), and we can write
\(\widetilde F_A=F_A+R_A\), where \(R_A\) is \(C^1\) and satisfies
\[
\sup_{z\in B_r(\overline z)}\|R_A(z)\|\leq \eta,\qquad
\sup_{z\in B_r(\overline z)}\|DR_A(z)\|\leq \eta
\]
for some \(r>0\) and all sufficiently small perturbations. This follows because the \(C^2\)-smallness of the perturbations implies
\(C^1\)-smallness of the induced perturbation \(R_A=\widetilde F_A-F_A\)
on a neighborhood of \(\overline z\). By Lemma~\ref{lemma:stab-zero}, for each sufficiently small perturbation there
exists a unique solution $\widetilde z
= (\widetilde y,\widetilde\lambda_A)
\in B_r(\overline z)$ to \(\widetilde F_A(z)=0\).

It remains to verify that this reduced solution is a KKT pair of the
perturbed problem. By SCSC at \((\overline y,\overline\lambda)\), we have
\[
\overline\lambda_i>0 \quad \text{for all } i\in A,
\qquad
h_i(x,\overline y)<0 \quad \text{for all } i\notin A.
\]
Choosing \(r\) and the perturbation size sufficiently small yields
\[
\widetilde\lambda_i>0 \quad \text{for all } i\in A,
\qquad
\widetilde h_i(x,\widetilde y)<0
\quad \text{for all } i\notin A.
\]
Extending \(\widetilde\lambda_A\) by setting
\(\widetilde\lambda_i=0\) for \(i\notin A\), we obtain a KKT pair
\((\widetilde y,\widetilde\lambda)\) of the perturbed problem.
Moreover, its active constraint set is exactly \(A\).

By \textbf{Step}~1, the set \(\mathcal K(x)\) contains only finitely many KKT pairs.
Applying the above argument to each KKT pair
\((\overline y,\overline\lambda)\in\mathcal K(x)\), and then taking the minimum
over this finite set, we obtain a positive threshold such that, whenever the
perturbation is below this threshold in the \(C^2(K_{\rm out})\)-norm, each
original KKT pair in \(\mathcal K(x)\) admits a unique nearby perturbed KKT pair
with the same active index set.

\paragraph{Step 3 (No new KKT pairs appear):}
Let \(\phi(y)\coloneq \max_{i\in\{1,\ldots,k\}}h_i(x,y)\). By
Lemma~\ref{lemma:near_licq}, applied with \(K=K_{\rm in}\), there exist
constants \(\rho>0\) and \(\sigma_*>0\) such that, for every \(y\in K_{\rm in}\)
satisfying \(\phi(y)\le \rho\), and every nonempty index set
\(\mathcal J\subseteq I\) satisfying \(|h_j(x,y)|\le \rho\) for all
\(j\in\mathcal J\), we have $\sigma_{\min}\big(\nabla_y h_{\mathcal J}(x,y)\big)\ge \sigma_*$. Since \(\mathcal Y(x)\subset \operatorname{int}(K_{\rm in})\), we may choose
\(\eta\in(0,\rho)\) sufficiently small such that $T\coloneq \{y\in K_{\rm in}:\phi(y)\le \eta\}
\subset \operatorname{int}(K_{\rm in})$.

We first exclude perturbed KKT pairs outside \(T\). For any
\(y\in K_{\rm in}\setminus \operatorname{int}(T)\), we have
\(\phi(y)\ge \eta\). If the perturbation is sufficiently small, then
\[
\widetilde\phi(y)\coloneq \max_i\widetilde h_i(x,y)\ge \eta/2>0,
\qquad
y\in K_{\rm in}\setminus \operatorname{int}(T).
\]
Thus no point in \(K_{\rm in}\setminus \operatorname{int}(T)\) is feasible for
the perturbed problem. Hence every perturbed KKT pair with
\(\widetilde y\in K_{\rm in}\) must satisfy \(\widetilde y\in T\).

We next obtain a uniform bound on the multipliers of perturbed KKT pairs in
\(T\). Let \((\widetilde y,\widetilde\lambda)\) be a perturbed KKT pair with
\(\widetilde y\in T\), and let $\widetilde{\mathcal J}\coloneq \{i:\widetilde h_i(x,\widetilde y)=0\}.$
If \(\widetilde{\mathcal J}\neq\emptyset\), then for every
\(j\in\widetilde{\mathcal J}\),
\[
|h_j(x,\widetilde y)|
=
|h_j(x,\widetilde y)-\widetilde h_j(x,\widetilde y)|
\le
\|\widetilde h_j-h_j\|_{C^0(K_{\rm in})}.
\]
Since \(\phi(\widetilde y)\le\eta<\rho\), by taking the perturbation sufficiently small, we can apply Lemma~\ref{lemma:near_licq}. Hence
\[
\sigma_{\min}
\big(\nabla_y h_{\widetilde{\mathcal J}}(x,\widetilde y)\big)
\ge \sigma_* .
\]
Moreover, since \(\nabla_y\widetilde h_{\widetilde{\mathcal J}}\) is uniformly
close to \(\nabla_y h_{\widetilde{\mathcal J}}\) on \(K_{\rm in}\), we have
\[
\sigma_{\min}
\big(\nabla_y\widetilde h_{\widetilde{\mathcal J}}(x,\widetilde y)\big)
\ge \frac{\sigma_*}{2}.
\]
By the KKT stationarity condition, the multiplier on the active set is uniquely
determined by
\[
\widetilde{\lambda}_{\widetilde{\mathcal J}}
=
-\Big(
\nabla_y \widetilde h_{\widetilde{\mathcal J}}(x,\widetilde y)
\nabla_y \widetilde h_{\widetilde{\mathcal J}}(x,\widetilde y)^\top
\Big)^{-1}
\nabla_y \widetilde h_{\widetilde{\mathcal J}}(x,\widetilde y)
\nabla_y \widetilde g(x,\widetilde y).
\] The uniform singular value bound above and the uniform boundedness
of \(\nabla_y\widetilde g\) and \(\nabla_y\widetilde h_i\) on \(K_{\rm in}\)
imply that $\|\widetilde\lambda\|\le M$ for some constant \(M>0\), independent of the specific perturbation. If
\(\widetilde{\mathcal J}=\emptyset\), then \(\widetilde\lambda=0\), so the same
bound also holds. Choose \(M'>M\) and define
\[
\Lambda\coloneq \{\lambda\in\mathbb R^k:\|\lambda\|\le M'\}.
\]

We now exclude perturbed KKT pairs in the tubular region \(T\) away from the
original KKT pairs. For each original KKT pair
\((\bar y,\bar\lambda)\in\mathcal K(x)\), choose an open ball
\(B_{r_{(\bar y,\bar\lambda)}}(\bar y,\bar\lambda)\) contained in the local
uniqueness neighborhood obtained in \textbf{Step}~2. Since \(\mathcal K(x)\) is finite,
we may choose these balls pairwise disjoint and contained in
\(\operatorname{int}(K_{\rm in})\times\operatorname{int}(\Lambda)\). Reducing the
perturbation size if necessary, the perturbed KKT pair produced by \textbf{Step}~2 for
each original KKT pair lies in the corresponding ball. Let \(\mathcal U\) be
the union of these balls, and set
\[
\mathcal C\coloneq (T\times\Lambda)\setminus\mathcal U .
\]
The set \(\mathcal C\) is compact and contains no original KKT pair. Define the
KKT residual
\begin{align*}
r(y,\lambda)
\coloneq &
\big\|\nabla_y g(x,y)+\sum_{i\in I}\lambda_i\nabla_y h_i(x,y)\big\|
+\sum_{i\in \{1,...,k\}}|\lambda_i h_i(x,y)|\\
&+\sum_{i\in \{1,...,k\}}|\min(0,\lambda_i)|
+\sum_{i\in \{1,...,k\}}|\min(0,-h_i(x,y))|.
\end{align*}
This residual vanishes exactly at the KKT pairs. Therefore, since
\(\mathcal C\) is compact and contains no original KKT pair, there exists
\(\delta>0\) such that
\[
r(y,\lambda)\ge\delta,\qquad (y,\lambda)\in\mathcal C.
\]
For sufficiently small perturbations, the perturbed residual
\(\widetilde r\) is uniformly close to \(r\) on \(T\times\Lambda\). Hence
\[
\widetilde r(y,\lambda)\ge \delta/2>0,
\qquad (y,\lambda)\in\mathcal C.
\]

Now let \((\widetilde y,\widetilde\lambda)\) be any perturbed KKT pair with
\(\widetilde y\in K_{\rm in}\). The first part of the argument gives
\(\widetilde y\in T\), and the multiplier bound gives
\(\widetilde\lambda\in\Lambda\). Thus $
(\widetilde y,\widetilde\lambda)\in T\times\Lambda .$
Since \((\widetilde y,\widetilde\lambda)\) is a perturbed KKT pair, we have
\(\widetilde r(\widetilde y,\widetilde\lambda)=0\). The lower bound on
\(\widetilde r\) over \(\mathcal C\) implies that
\((\widetilde y,\widetilde\lambda)\notin\mathcal C\). Hence $
(\widetilde y,\widetilde\lambda)\in\mathcal U.$
 
It remains to identify the point inside each ball. Suppose
\((\widetilde y,\widetilde\lambda)\) lies in the ball centered at
\((\bar y,\bar\lambda)\in\mathcal K(x)\), and let
\(A\coloneq \mathcal J(x,\bar y)\). By choosing the balls sufficiently small and using
SCSC at \((\bar y,\bar\lambda)\), we have
\[
\widetilde\lambda_i>0 \quad \text{for all } i\in A,
\qquad
h_i(x,\widetilde y)<0 \quad \text{for all } i\notin A .
\]
For sufficiently small perturbations,
\[
\widetilde h_i(x,\widetilde y)<0 \quad \text{for all } i\notin A .
\]
Since \((\widetilde y,\widetilde\lambda)\) is a perturbed KKT pair,
complementarity and \(\widetilde\lambda_i>0\) for \(i\in A\) imply
\[
\widetilde h_i(x,\widetilde y)=0 \quad \text{for all } i\in A.
\]
Therefore, the active constraint set of \(\widetilde y\) is exactly \(A\).
Thus any perturbed KKT pair in this ball solves the reduced perturbed KKT
system associated with the same active set \(A\). By the local uniqueness result
in \textbf{Step}~2, the ball contains at most one perturbed KKT pair. Since \textbf{Step}~2 also
gives the existence of one perturbed KKT pair in this ball, each ball contains
exactly one perturbed KKT pair, with the same active constraint set as the
original KKT pair at its center.

Therefore, the original KKT points and the perturbed KKT points within
\(K_{\rm in}\) are in one-to-one correspondence, and this correspondence
preserves the active constraint pattern.

\paragraph{Step 4 (Non-minimizing KKT points remain non-minimizing):}
We show that a KKT point of the original problem that is not a local minimizer
cannot become a local minimizer after perturbation. Let
\((\overline y,\overline\lambda)\in\mathcal K(x)\) be a KKT pair of the original problem that $\overline y$
is not a local minimizer, and let
\[
A\coloneq \mathcal J(x,\overline y).
\]
By \textbf{Step}~3, the corresponding perturbed KKT pair
\((\widetilde y,\widetilde\lambda)\) has the same active constraint set \(A\).
Moreover, by SCSC, the critical cone of the original problem at
\((\overline y,\overline\lambda)\) is
\[
\mathcal{C}(x,\overline y)
=
\{d:\nabla_y h_A(x,\overline y)d=0\}.
\]
Since \(\overline y\) is not a local minimizer, SOSC cannot hold at
\((\overline y,\overline\lambda)\). Together with the nonsingularity of the KKT
matrix, this yields a direction \(d\in \mathcal{C}(x,\overline y)\) such that
\[
d^\top \nabla^2_{yy}L(x,\overline y,\overline\lambda)d<0.
\]
Now consider the perturbed critical cone. Since the active set is preserved and
\(\widetilde\lambda_i>0\) for all \(i\in A\), we have
\[
\widetilde{\mathcal{C}}(x,\widetilde y)
=
\{d:\nabla_y\widetilde h_A(x,\widetilde y)d=0\}.
\]
The matrices \(\nabla_y\widetilde h_A(x,\widetilde y)\) are uniformly close to
\(\nabla_y h_A(x,\overline y)\) and have the same rank. Therefore the corresponding
null spaces vary continuously. Hence there exists
\(\widetilde d\in \widetilde{\mathcal{C}}(x,\widetilde y)\) with
\(\widetilde d\to d\) as the perturbation size tends to zero.

By the \(C^2\)-smallness of the perturbation, we then have
\[
\widetilde d^\top
\nabla^2_{yy}\widetilde L(x,\widetilde y,\widetilde\lambda)
\widetilde d
<0
\]
for all sufficiently small perturbations. Thus the second-order necessary
condition for local minimality fails at
\((\widetilde y,\widetilde\lambda)\). Therefore
\(\widetilde y\) is not a local minimizer of the perturbed problem.

\paragraph{Step 5 (Local minimizers are preserved and counted):}
Let \(\overline y\) be a local minimizer of the original problem with multiplier
\(\overline\lambda\), and let \(A\coloneq \mathcal J(x,\overline y)\). By \textbf{Step}~3, the
corresponding perturbed KKT pair \((\widetilde y,\widetilde\lambda)\) is unique
in a neighborhood of \((\overline y,\overline\lambda)\), and its active set is
preserved:
\[
\widetilde{\mathcal J}(x,\widetilde y)=\mathcal J(x,\overline y)=A.
\]
In particular, we have $
\widetilde y\in \widetilde{\mathcal S}_{A}(x).$

It remains to show that \(\widetilde y\) is still a local minimizer. Since
\(\overline y\) is a local minimizer and the KKT matrix at
\((\overline y,\overline\lambda)\) is nonsingular, the second-order sufficient
condition holds at \((\overline y,\overline\lambda)\). By the preservation of
the active set and the \(C^2\)-smallness of the perturbation, SOSC
also holds at \((\widetilde y,\widetilde\lambda)\). Hence \(\widetilde y\) is a
strict local minimizer of the perturbed problem.

Consequently, each local minimizer on a stratum \(\mathcal S_{\mathcal J}(x)\)
gives rise to a unique perturbed local minimizer on
\(\widetilde{\mathcal S}_{\mathcal J}(x)\cap K_{\rm in}\). Conversely, by
\textbf{Step}~3 every perturbed KKT pair in \(K_{\rm in}\) corresponds to an original KKT
pair, and by \textbf{Step}~4 an original KKT pair whose primal component is not a local
minimizer cannot give rise to a perturbed local minimizer. Therefore no
additional perturbed local minimizers in \(K_{\rm in}\) can appear. After
choosing the perturbation threshold sufficiently small so that all smallness
requirements in Steps~2--5 hold, this gives the constant
\(\epsilon(x,K_{\rm in},K_{\rm out})\) and proves the stated equality of the
stratum-wise counts.

\subsection{Proof of Theorem~\ref{thm:LICQ}}\label{proof:thm:LICQ}

Throughout the proof, Sard's theorem is applied on an open neighborhood of 
the original set \(\mathcal X\), where \(g\) and \(h_i\), \(i=1,\ldots,k\), are smooth. 
For notational simplicity, this neighborhood is still denoted by \(\mathcal X\);
the final measure-zero conclusion is then restricted back to the original set. Let $\mathcal{J}$ be a subset of $\{1,...,k\}$, and $H_\mathcal{J}(x,y)=(h_i(x,y))_{i\in \mathcal{J}}\in\mathbb{R}^{|\mathcal{J}|}$. Consider the mapping:
    $$G_\mathcal{J}(x,y)=(x,H_\mathcal{J}(x,y))\in\mathcal{X}\times\mathbb{R}^{|\mathcal{J}|}\subset\mathbb{R}^{n}\times\mathbb{R}^{|\mathcal{J}|}.$$
    The key idea is as follows. First, by Sard’s Theorem, the critical values of \(G_{\mathcal J}\), namely those values \((x,-a_{\mathcal J})\) for which there exists some \(y\) satisfying \(G_{\mathcal J}(x,y)=(x,-a_{\mathcal J})\) and \(dG_{\mathcal J}(x,y)\) is not full row rank, has Lebesgue measure zero in $\mathcal{X} \times \mathbb{R}^{|\mathcal{J}|}$. Second, by applying Fubini’s Theorem, we transfer the measure-zero property from the joint space of $(x,a)$ to the $x$-space for almost every $a \in [0,\nu]^k$ and obtain that for almost every $a$, the set of $x$ such that $(x,-a)$ is a critical value is of measure zero. Finally, the critical values of $G_\mathcal{J}$ correspond to the points where $\{\widetilde{h}_i(x,y)\}_{i\in \mathcal{J}}$ become linearly dependent at some feasible $y$, thereby establishing the link between the critical values of $G_\mathcal{J}$ and the failure of LICQ.\\

    We define the following sets 
    \begin{itemize}
        \item $C_\mathcal{J}$ be the critical values set of $G_\mathcal{J}(x,y)$;
        \item $E_\mathcal{J}\coloneq \{(x,a)\in\mathcal{X}\times[0,\nu]^k:(x,-a_\mathcal{J})\in C_\mathcal{J}\}$;
        \item $E\coloneq \cup_{\mathcal{J}\subset\{1,...,k\}}E_\mathcal{J}$.
    \end{itemize}
    
    By Sard's Theorem, the Lebesgue measure of $C_\mathcal{J}$, i.e., the critical values set of $G_\mathcal{J}(x,y)$, is zero. Thus, the measure of $E_\mathcal{J}$ in $\mathcal{X}\times[0,\nu]^k$ is also zero. Since the collection of such index subsets $\mathcal{J}\subset\{1,...,k\}$ is finite, the union $E$ also has measure zero.
    
    Define the indicator function \begin{align*}
        \mathcal{I}_{E}(x,a)=\begin{cases}
        1\quad \text{when}\quad (x,a)\in  E  \\
        0\quad \text{when}\quad (x,a)\notin  E.
    \end{cases}
    \end{align*}
    Note that 
    $$\int_{\mathcal{X}\times[0,\nu]^k}\mathcal{I}_{E}(x,a) dx\ da=0.$$
    By Fubini's Theorem, we have
    \begin{align}\label{eq:fubini1}
        \int_{[0,\nu]^k}\int_{\mathcal{X}}\mathcal{I}_{E}(x,a)dx\ da=\int_{\mathcal{X}\times[0,\nu]^k}\mathcal{I}_{E}(x,a) dx\ da=0.
    \end{align}
    We claim that, for almost every $a \in [0,\nu]^k$, the set of $x$ such that $(x,a) \in E$ has Lebesgue measure zero in $\mathcal{X}$. Suppose, to the contrary, that this is not the case. Then there exists a subset $U \subset [0,\nu]^k$ with positive measure such that, for every $a \in U$, the following integral is positive:
    $$\int_{\mathcal{X}}\mathcal{I}_{E}(x,a)dx>0.$$
    This implies
    $$\int_{[0,\nu]^k}\int_{\mathcal{X}}\mathcal{I}_{E}(x,a)dx\ da\geq\int_{U}\int_{\mathcal{X}}\mathcal{I}_{E}(x,a)dx\ da>0$$
    which contradicts \eqref{eq:fubini1}.

    The Jacobian of $G_\mathcal{J}(x,y)$ can be computed as follows:
    \begin{align}\label{eq:sard1}
        d G_\mathcal{J}(x,y)=\begin{pmatrix}
        I_n & 0\\
        \nabla_x H_\mathcal{J}(x,y) & \nabla_y H_\mathcal{J}(x,y)
        \end{pmatrix}\in\mathbb{R}^{(n+|\mathcal{J}|)\times(n+m)}.
    \end{align}
    The condition that $d G_\mathcal{J}(x,y)$ is not full row rank is equivalent to $\mathrm{rank}(\nabla_y H_\mathcal{J}(x,y))<|\mathcal{J}|$. For any subvector $a_\mathcal{J}\coloneq (a_i)_{i\in \mathcal{J}}\in\mathbb{R}^{|\mathcal{J}|}$, the condition that $(x,y)$ satisfies $h_i(x,y)+a_i=0$ ($\forall i\in \mathcal{J}$) is equivalent to $G_\mathcal{J}(x,y)=(x,-a_\mathcal{J})$. Thus, $(x,-a_\mathcal{J})$ is a critical value of $G_\mathcal{J}$ if and only if $G_\mathcal{J}(x,y)=(x,-a_\mathcal{J})$ and $\mathrm{rank}(\nabla_y H_\mathcal{J}(x,y))<|\mathcal{J}|$ for some $y$.
    
    Note that LICQ fails at $(x,y)$ if and only if the active constraint gradients $\{\nabla_y \widetilde{h}_i(x,y):i\in \mathcal{J}(x,y)\}$ are linearly dependent, where $\mathcal{J}(x,y)$ denotes the active index set at $(x,y)$. This condition is precisely equivalent to $G_{\mathcal{J}(x,y)}(x,y)=(x,-a_{\mathcal J(x,y)})$ and $\mathrm{rank}(\nabla_y H_{\mathcal{J}(x,y)}(x,y))<|\mathcal{J}(x,y)|$, which implies $(x,-a_{\mathcal J(x,y)})$ is a critical value of $G_{\mathcal{J}(x,y)}$ by the argument above. It follows that $(x,a) \in E_{\mathcal{J}(x,y)} \subset E$. Therefore, the failure of LICQ at some feasible point $(x,y)$ for the perturbed constraints implies $(x,a)$ belonging to the set $E$. By the claim that for almost every $a \in [0,\nu]^k$, the set of $x$ such that $(x,a) \in E$ has Lebesgue measure zero in $\mathcal{X}$, we complete the proof.

\subsection{Proof of Theorem~\ref{thm:SCSC}}\label{proof:thm:SCSC}
Throughout the proof, Sard's theorem is applied on an open neighborhood of 
the original set \(\mathcal X\), where \(g\) and \(h_i\), \(i=1,\ldots,k\), are smooth. 
For notational simplicity, this neighborhood is still denoted by \(\mathcal X\);
the final measure-zero conclusion is then restricted back to the original set. By Theorem~\ref{thm:LICQ}, with probability one (i.e., almost surely with respect to the random choice of $a$), the perturbed constraints satisfy LICQ at every feasible point $y$ for almost every $x$. Hence, under this regularity, the strict complementary slackness condition can be well-defined.
    
    The main idea is to construct, for each index subset $\mathcal{J} \subset \{1,\dots,k\}$ and each $i \in \mathcal{J}$, a parameterized system $H_{\mathcal{J},i}(z;x,b,a)=0$ whose variables are the primal–dual pair $z=(y,\lambda_\mathcal{J})$ and whose parameters are $(x,b,a)$. If, for some parameters $(x,b,a)$, there exists a KKT point $(y,\lambda_\mathcal{J})$ of the perturbed problem that violates the strict complementary slackness condition, then the corresponding system $H_{\mathcal{J},i}(z;x,b,a)=0$ admits a solution. We then analyze the transversality properties of this system with respect to the zero set. By showing that the family map $H_{\mathcal{J},i}(\cdot;\cdot,\cdot,\cdot)=0$ is transverse to $\{0\}$, the Parametric Transversality Theorem guarantees that, for almost every choice of parameters $(x,b,a)$, the slice map $H_{\mathcal{J},i}(\cdot;x,b,a)=0$ with those parameters is also transverse to $\{0\}$. Since the Jacobian $D_{z}H_{\mathcal{J},i}$ with respect to $z=(y,\lambda_\mathcal{J})$ has more rows than columns, transversality implies that the slice map cannot have any zeros, i.e., the system has no solution. Consequently, for almost every $(x,b,a)$, there exists no KKT point violating SCSC, which establishes the desired result.

    Define $\mathcal{J}\subset\{1,...,k\}$, $\mathcal{J}^c\coloneq \{1,...,k\}\setminus \mathcal{J}$, and $h_\mathcal{J}(x,y)\coloneq (h_i(x,y))_{i\in \mathcal{J}}\in\mathbb{R}^{|\mathcal{J}|}$, $a_\mathcal{J}\coloneq (a_i)_{i\in \mathcal{J}}\in\mathbb{R}^{|\mathcal{J}|}$, $a_{\mathcal{J}^c}\coloneq (a_i)_{i\in \mathcal{J}^c}\in\mathbb{R}^{k-|\mathcal{J}|}$. Set $z\coloneq (y,\lambda_\mathcal{J})\in\mathbb{R}^{m+|\mathcal{J}|}$ and choose index $i\in \mathcal{J}$. We consider the following system:
    \begin{align*}
        H_{\mathcal{J},i}(z;x,b,a)\coloneq \begin{pmatrix}
            \nabla_y g(x,y)+b+\sum_{j\in \mathcal{J}}\lambda_j\nabla_y h_j(x,y)\\
            h_\mathcal{J}(x,y)+a_\mathcal{J}\\
            \lambda_i
        \end{pmatrix}=0\in\mathbb{R}^{m+|\mathcal{J}|+1}
    \end{align*}
    We will prove that $H_{\mathcal{J},i}(\cdot;\cdot,\cdot,\cdot)$ is transverse to $\{0\}$, then by Parametric Transversality Theorem, the slice map $H_{\mathcal{J},i}(\cdot;x,b,a)$ is transverse to $\{0\}$ for almost every $(x,b,a)$.

    Define the zero set
    \begin{align*}
        \mathcal{M}_{\mathcal{J},i}\coloneq \{(z,x,b,a):H_{\mathcal{J},i}(z;x,b,a)=0\}\subset\mathbb{R}^{m+|\mathcal{J}|}\times\mathbb{R}^{n}\times\mathbb{R}^{m}\times\mathbb{R}^{k}.
    \end{align*}
    
    For any $(z,x,b,a)\in\mathcal{M}_{\mathcal{J},i}$, we compute the Jacobian of $H_{\mathcal{J},i}(z;x,b,a)$ with respect to $(z,x,b,a)$:
    \begin{align*}
        D_z H_{\mathcal{J},i}(z;x,b,a)=\begin{pmatrix}
            \nabla^2_{yy}g(x,y)+\sum_{j\in \mathcal{J}}\lambda_j\nabla_{yy}^2 h_j(x,y) & \nabla_y h_\mathcal{J}(x,y)\\
            \nabla_y h_\mathcal{J}(x,y)^\top & 0\\
            0 & e_i^\top
        \end{pmatrix}\in\mathbb{R}^{(m+|\mathcal{J}|+1)\times(m+|\mathcal{J}|)}.
    \end{align*}
    \begin{align*}
        D_{x} H_{\mathcal{J},i}(z;x,b,a)=\begin{pmatrix}
            \nabla_{xy}^2 g(x,y)+\sum_{j\in \mathcal{J}}\lambda_j\nabla_{xy}^2 h_j(x,y)\\
            \nabla_x h_\mathcal{J}(x,y)^\top\\
            0
        \end{pmatrix}\in\mathbb{R}^{(m+|\mathcal{J}|+1)\times n}.
    \end{align*}
    \begin{align*}
        D_{(b,a_\mathcal{J},a_{\mathcal{J}^c})} H_{\mathcal{J},i}(z;x,b,a)=\begin{pmatrix}
            I_m & 0 & 0\\
            0 & I_{|\mathcal J|} & 0\\
            0 & 0 & 0
        \end{pmatrix}\in\mathbb{R}^{(m+|\mathcal{J}|+1)\times(m+k)}.
    \end{align*}
    It is clear that $D_{(z,x,b,a)}H_{\mathcal{J},i}(z;x,b,a)$, i.e., the following matrix, has full row rank
    \begin{align*}
        \begin{pmatrix}
            \nabla^2_{yy}g(x,y)+\sum_{j\in \mathcal{J}}\lambda_j\nabla_{yy}^2 h_j(x,y) & \nabla_y h_\mathcal{J}(x,y) & \nabla_{xy}^2 g(x,y)+\sum_{j\in \mathcal{J}}\lambda_j\nabla_{xy}^2 h_j(x,y) & I_m & 0 & 0\\
            \nabla_y h_\mathcal{J}(x,y)^\top & 0 & \nabla_x h_\mathcal{J}(x,y)^\top & 0 & I_{|\mathcal{J}|} & 0\\
            0 & e_i^\top & 0 & 0 & 0 & 0
        \end{pmatrix}.
    \end{align*}
    Therefore, we have $H_{\mathcal{J},i}(\cdot;\cdot,\cdot,\cdot)\pitchfork \{0\}$. By the Parametric Transversality Theorem, since the family map $H_{\mathcal{J},i}(\cdot;\cdot,\cdot,\cdot)$ is transverse to $\{0\}$, there exists a measure-zero subset $$C_{\mathcal{J},i}\subset\mathbb{R}^n\times[0,\nu]^m\times[0,\nu]^{|\mathcal{J}|},$$ such that for every $(x,b,a_{\mathcal J})\in C_{\mathcal{J},i}^c=\mathbb{R}^n\times[0,\nu]^m\times[0,\nu]^{|\mathcal{J}|}\setminus C_{\mathcal{J},i}$ , the slice map
    \begin{align*}
        H_{\mathcal{J},i}(\cdot;x,b,a):z\mapsto H_{\mathcal{J},i}(z;x,b,a)
    \end{align*}
    is also transverse to $\{0\}$. By the definition of transversality in Definition~\ref{def:transversality}, and noting that $\{0\}$ is a manifold with dimension $0$, we obtain that for any $z\in(H_{\mathcal{J},i}(\cdot;x,b,a))^{-1}(0)$, the Jacobian mapping 
    $$D_z H_{\mathcal{J},i}(\cdot;x,b,a):\mathbb{R}^{m+|\mathcal{J}|}\to\mathbb{R}^{m+|\mathcal{J}|+1}$$
    is surjective. However, noting that $m+|\mathcal{J}|+1>m+|\mathcal{J}|$, the matrix $D_z H_{\mathcal{J},i}(\cdot;x,b,a)$ is impossible to have full row rank. Therefore, we obtain that $(H_{\mathcal{J},i}(\cdot;x,b,a))^{-1}(0)=\emptyset$ for every $(x,b,a_{\mathcal J})\in C_{\mathcal{J},i}^c$. 

    The above argument holds for any $\mathcal{J}$ and $i$. Consider the set
    $$E_{\mathcal{J},i}\coloneq \{(x,b,a)\in\mathcal{X}\times[0,\nu]^m\times[0,\nu]^k:(x,b,a_\mathcal{J})\in C_{\mathcal{J},i}\}$$
    and $E\coloneq \cup_{\mathcal{J}\subset\{1,...,k\}}\cup_{i\in \mathcal{J}}E_{\mathcal{J},i}$. Since for any $\mathcal{J}$ and $i$, $E_{\mathcal{J},i}$ has measure zero, and there are only finitely many combinations of $\mathcal{J}$ and $i$, it follows that $E$ also has measure zero. Define the indicator function
    \begin{align*}
        \mathcal{I}_E(x,b,a)=\begin{cases}
            1\quad \text{when}\quad (x,b,a)\in  E  \\
            0\quad \text{when}\quad (x,b,a)\notin  E.
        \end{cases}
    \end{align*}
    Note that 
    $$\int_{\mathcal{X}\times[0,\nu]^m\times[0,\nu]^k}\mathcal{I}_{E}(x,b,a) dx\ db\ da=0.$$
    By Fubini's Theorem, we have
    \begin{align}\label{eq:fubini2}
        \int_{[0,\nu]^m\times[0,\nu]^k}\int_{\mathcal{X}}\mathcal{I}_{E}(x,b,a)dx\ db\ da=\int_{\mathcal{X}\times[0,\nu]^m\times[0,\nu]^k}\mathcal{I}_{E}(x,b,a) dx\ db\ da=0
    \end{align}
    We claim that, for almost every $b\in[0,\nu]^m$, $a \in [0,\nu]^k$, the set of $x$ such that $(x,b,a) \in E$ has Lebesgue measure zero in $\mathcal{X}$. Suppose, to the contrary, that this is not the case. Then there exists a subset $U \subset [0,\nu]^m\times[0,\nu]^k$ with positive measure such that, for every $(b,a) \in U$, the following integral is positive:
    $$\int_{\mathcal{X}}\mathcal{I}_{E}(x,b,a)dx>0.$$
    This implies
    $$\int_{[0,\nu]^m\times[0,\nu]^k}\int_{\mathcal{X}}\mathcal{I}_{E}(x,b,a)dx\ db\ da\geq\int_{U}\int_{\mathcal{X}}\mathcal{I}_{E}(x,b,a)dx\ db da>0$$
    which contradicts \eqref{eq:fubini2}.

    Note that, by the construction of $H_{\mathcal{J},i}(\cdot;x,b,a)$, if there exists a KKT point of the perturbed problem that does not satisfy the strict complementary slackness condition, then for some index subset $\mathcal{J}$ and some $i \in \mathcal{J}$, the system $H_{\mathcal{J},i}(\cdot;x,b,a)=0$ must admit a solution. Hence, if $(x,b,a) \notin E$, it follows that all KKT points associated with these parameters satisfy the strict complementary slackness condition. We obtain that, for almost every $(b,a)\in[0,\nu]^m\times[0,\nu]^k$, the set of $x$ such that there exists KKT points of the perturbed problem, which does not satisfy strict complementary slackness condition, has measure zero.

\subsection{Proof of Theorem~\ref{thm:SOSC}}\label{proof:thm:SOSC}

Throughout the proof, Sard's theorem is applied on an open neighborhood of 
the original set \(\mathcal X\), where \(g\) and \(h_i\), \(i=1,\ldots,k\), are smooth. 
For notational simplicity, this neighborhood is still denoted by \(\mathcal X\);
the final measure-zero conclusion is then restricted back to the original set. By Theorem~\ref{thm:LICQ},~\ref{thm:SCSC}, with probability one (i.e., almost surely with respect to the random choice of $b$ and $a$), the perturbed constraints satisfy LICQ and SCSC for almost every $x$. At those $x$, for any associated KKT point $(y,\lambda)$, LICQ and SCSC imply that the critical cone reduces to the linear space
    \begin{align}\label{eq:linear_space}
        \mathcal{C}(x,y)\coloneq \{d\in\mathbb{R}^m:\nabla_y\widetilde{h}_{\mathcal{J}(x,y)}(x,y)d=0\},
    \end{align} 
    where $\mathcal{J}(x,y)=\{i:\widetilde{h}_i(x,y)=0\}$. Hence, the SOSC simplifies to:
    $$\text{for all }d\ne 0\text{ with }d\in\mathcal{C}(x,y),\quad d^\top\nabla^2_{yy}\widetilde{\mathcal{L}}(x,y,\lambda) d>0,$$
    with the perturbed Lagrangian
    $$\widetilde{\mathcal{L}}(x,y,\lambda)=\widetilde{g}(x,y)+\sum_{i\in \mathcal{J}(x,y)}\lambda_i \widetilde{h}_i(x,y).$$ To prove Theorem~\ref{thm:SOSC}, it suffices to show that, for almost every perturbation, the following holds: for almost every $x$, if $y$ is a local minimizer and $\lambda$ is a corresponding Lagrange multiplier, then there is no direction $d$ in $\mathcal{C}(x,y)$ such that $d^\top\nabla^2_{yy}\widetilde{\mathcal{L}}(x,y,\lambda) d=0$.

    Define $\mathcal{J}\subset\{1,...,k\}$, and $h_\mathcal{J}(x,y)\coloneq (h_i(x,y))_{i\in \mathcal{J}}\in\mathbb{R}^{|\mathcal{J}|}$, $a_\mathcal{J}\coloneq (a_i)_{i\in \mathcal{J}}\in\mathbb{R}^{|\mathcal{J}|}$. We consider the following map:
    \begin{align*}
        G_{\mathcal{J}}(x,y,\lambda_\mathcal{J})\coloneq (
            x,\nabla_y g(x,y)+\sum_{i\in \mathcal{J}}\lambda_i\nabla_y h_i(x,y),h_\mathcal{J}(x,y))
    \end{align*}
    We define the following sets 
    \begin{itemize}
        \item $C_\mathcal{J}$ be the critical values set of $G_\mathcal{J}(x,y,\lambda_\mathcal{J})$;
        \item $E_\mathcal{J}\coloneq \{(x,b,a)\in\mathcal{X}\times[0,\nu]^m\times[0,\nu]^k:(x,-b,-a_\mathcal{J})\in C_\mathcal{J}\}$;
        \item $E\coloneq \cup_{\mathcal{J}\subset\{1,...,k\}}E_\mathcal{J}$.
    \end{itemize}
    Similarly to the argument in the proof of Theorem~\ref{thm:LICQ}, we claim that, for almost every $(b,a) \in [0,\nu]^m\times[0,\nu]^k$, the set of $x$ such that $(x,b,a) \in E$ has Lebesgue measure zero in $\mathcal{X}$. The Jacobian of $G_\mathcal{J}(x,y,\lambda_\mathcal{J})$ can be computed as follows:
    \begin{align}
        d G_\mathcal{J}(x,y,\lambda_\mathcal{J})=\begin{pmatrix}
        I_n & 0 & 0\\
        \ast & \nabla^2_{yy} g(x,y)+\sum_{i\in \mathcal{J}}\lambda_i\nabla^2_{yy}h_i(x,y) & \nabla_y h_\mathcal{J}(x,y)^\top\\
        \ast & \nabla_y h_\mathcal{J}(x,y) & 0 
    \end{pmatrix}\in\mathbb R^{(n+m+|\mathcal J|)\times(n+m+|\mathcal J|)}.
    \end{align}
    The condition that $d G_\mathcal{J}(x,y,\lambda_\mathcal{J})$ is not full row rank is equivalent to the degeneracy of 
    \begin{align}\label{eq:matrix1}
        \begin{pmatrix}
        \nabla^2_{yy} g(x,y)+\sum_{i\in \mathcal{J}}\lambda_i\nabla^2_{yy}h_i(x,y) & \nabla_y h_\mathcal{J}(x,y)^\top\\
        \nabla_y h_\mathcal{J}(x,y) & 0 
    \end{pmatrix}\in\mathbb R^{(m+|\mathcal J|)\times(m+|\mathcal J|)}.
    \end{align}
    Suppose \(y\) is a local minimizer with corresponding multiplier \(\lambda\) of the perturbed problem~\eqref{eq:perturb3}. Then the KKT condition gives
    \[
    G_{\mathcal J(x,y)}(x,y,\lambda_{\mathcal J(x,y)})
    =
    (x,-b,-a_{\mathcal J(x,y)}).
    \]
    If there exists $d\ne 0$ in $\mathcal{C}(x,y)$ defined in \eqref{eq:linear_space} such that $\nabla^2_{yy}\widetilde{\mathcal{L}}(x,y,\lambda)d+\nabla_y h_{\mathcal J(x,y)}(x,y)^\top\mu=0$, consider the following problem
    \begin{align}\label{eq:opt1}
        \min_{v\in\mathcal{C}(x,y)}\frac{1}{2}v^\top\nabla^2_{yy}\widetilde{\mathcal{L}}(x,y,\lambda) v.
    \end{align}
    By the second-order necessary conditions~\cite[Theorem~12.5]{nocedal2006numerical}, we obtain that the problem~\eqref{eq:opt1} is convex in the subspace $\mathcal{C}(x,y)$. Note that $d^\top\nabla^2_{yy}\widetilde{\mathcal{L}}(x,y,\lambda) d=0$. Thus, $d$ is a minimizer of \eqref{eq:opt1}, which implies that there exists \(\mu\) such that
    \[
    \nabla^2_{yy}\widetilde{\mathcal{L}}(x,y,\lambda)d
    +
    \nabla_y h_{\mathcal J(x,y)}(x,y)\mu
    =0.
    \]
    It follows that
    \begin{align*}
        &\begin{pmatrix}
        \nabla^2_{yy} g(x,y)+\sum_{i\in {\mathcal{J}(x,y)}}\lambda_i\nabla^2_{yy}h_i(x,y) & \nabla_y h_{\mathcal{J}(x,y)}(x,y)^\top\\
        \nabla_y h_{\mathcal{J}(x,y)}(x,y) & 0 
    \end{pmatrix}
    \begin{pmatrix}
    d \\ 
    \mu
    \end{pmatrix}\\
    =&\begin{pmatrix}
        \nabla^2_{yy}\widetilde{\mathcal{L}}(x,y,\lambda) d+\nabla_y h_{\mathcal{J}(x,y)}(x,y)^\top\mu\\
        \nabla_y h_{\mathcal{J}(x,y)}(x,y)d
    \end{pmatrix}\\
    =&0,
    \end{align*}
    which means \eqref{eq:matrix1} degenerates for the index set 
    \(\mathcal J(x,y)\). Hence
    \((x,-b,-a_{\mathcal J(x,y)})\) is a critical value of 
    \(G_{\mathcal J(x,y)}\). Consequently, if a triple $(x,b,a)$ admits a local minimizer $y$ of the perturbed lower-level problem, together with an associated multiplier $\lambda$, for which the second-order sufficient condition fails, then it follows that $(x,b,a) \in E$. Noting that, for almost every $(b,a) \in [0,\nu]^m \times [0,\nu]^k$, the set of $x$ satisfying $(x,b,a) \in E$ has Lebesgue measure zero in $\mathcal{X}$, we thereby complete the proof.

\subsection{Proof of Proposition~\ref{prop:smooth_kkt_curve}}\label{proof:prop:smooth_kkt_curve}
Since $g(x, \cdot)$ is strongly convex and each $h_i(x, \cdot)$ is convex, the feasible set $\mathcal{Y}(x)$ is convex and the lower-level problem admits a unique global minimizer $y^\ast(x)$ for each $x \in \mathcal{X}$. By the Berge maximum theorem, the map $x \mapsto y^\ast(x)$ is continuous on $\mathcal{X}$. Under LICQ, $y^\ast(x)$ is a KKT point with a unique associated Lagrange multiplier $\lambda(x)$.

Let $\mathcal{J}(x,y^\ast(x)) \coloneq  \{i : h_i(x, y^\ast(x)) = 0\}$ denote the active constraint index set at $y^\ast(x)$. By Theorem~\ref{thm:main_scsc}, together with SCSC holding at every $x \in \mathcal{X}$, the active set $\mathcal{J}(x,y^\ast(x))$ is constant in $x$. Denote this common active set by $\mathcal{J}$. The pair $(y^\ast(x), \lambda_{\mathcal{J}}(x))$ satisfies the reduced KKT system
\[
\Psi(x, y, \lambda_{\mathcal{J}}) \coloneq  \begin{pmatrix} \nabla_y g(x, y) + \nabla_y h_{\mathcal{J}}(x, y)^\top \lambda_{\mathcal{J}} \\ h_{\mathcal{J}}(x, y) \end{pmatrix} = 0.
\]
The Jacobian of $\Psi$ with respect to $(y, \lambda_{\mathcal{J}})$ is exactly the matrix $M_+(x)$ defined in~\eqref{eq:sensitivity_formula}. Under LICQ, the rows of $\nabla_y h_{\mathcal{J}}(x, y^\ast(x))$ are linearly independent. Combined with the strong convexity of $g(x, \cdot)$ and the convexity of each $h_i(x, \cdot)$, this ensures that $M_+(x)$ is nonsingular at every $x \in \mathcal{X}$. By the Implicit Function Theorem applied at each $x \in \mathcal{X}$, $(y^\ast(x), \lambda_{\mathcal{J}}(x))$ is a smooth function of $x$.

\subsection{Proof of Theorem~\ref{thm:global_continuation}}\label{proof:thm:global_continuation}

To prove Theorem~\ref{thm:global_continuation}, we need the following auxiliary lemma.

\begin{lemma}\label{lem:global_extension}
    Let $X$ be a connected, locally path-connected, simply connected, and locally compact Hausdorff topological space, and let $Y$ be a locally compact Hausdorff topological space. Let the set be $S \subset X \times Y$, and denote the natural projection map by $p: S \to X$, where $p(x, y) = x$. Suppose the following conditions hold: \begin{enumerate}
        \item \textbf{Local Homeomorphism:} The natural projection map $p: S \to X$, defined by $p(x, y) = x$, is a local homeomorphism. Equivalently, for any $(x_0, y_0) \in S$, there exists an open neighborhood $U$ of $x_0$ in $X$ and an open neighborhood $V$ of $y_0$ in $Y$ such that for each $x \in U$, there is a unique $y \in V$ satisfying $(x, y) \in S$. Moreover, this unique local solution $y = f(x)$ depends continuously on $x$.
        \item \textbf{Closed Graph:} The solution set $S$ is a closed subset in the product topology of $X \times Y$;
        \item \textbf{Uniform Compactness:} There exists a compact set \(K\subset Y\) such that \(S\subset X\times K\).
    \end{enumerate}
     Then, for any given point $(x_0, y_0) \in S$, there exists a unique global continuous map $y: X \to Y$ such that $y(x_0) = y_0$, and for all $x \in X$, we have $(x, y(x)) \in S$.
\end{lemma}
\begin{proof}
The proof is divided into five steps.
\paragraph{Step 1 (The projection \(p\) is surjective):} 
We first show that \(p(S) = X\). Since \(X\) is connected, it suffices to show that \(p(S)\) is both open and closed in \(X\), and nonempty. 
\begin{itemize}
    \item \textit{Openness:} By assumption, \(p: S \to X\) is a local homeomorphism. A standard result in topology states that any local homeomorphism is an open map. Therefore, the image \(p(S)\) is an open subset of \(X\).
   \item \textit{Closedness:} We are given that \(S\) is a closed subset of \(X \times Y\), and \(S \subset X \times K\) for some compact set \(K \subset Y\). Hence \(S\) is closed in the subspace topology of \(X \times K\). By Lemma~\ref{lem:tube}, the projection \(\pi_X: X \times K \to X\) is a closed map. Therefore \(p(S) = \pi_X(S)\) is closed in \(X\).
    \item \textit{Nonemptiness:} The set \(S\) is nonempty because it contains the given point \((x_0, y_0)\). Hence, \(p(S)\) is nonempty.
\end{itemize}
Since \(p(S)\) is a nonempty, open, and closed subset of the connected space \(X\), we conclude that \(p(S) = X\). Thus, \(p\) is surjective.

\paragraph{Step 2 (The projection \(p\) is a proper map):}
A continuous map is proper if the preimage of every compact set is compact. Let \(C \subset X\) be an arbitrary compact set. We examine its preimage under \(p\):
\[
p^{-1}(C) = S \cap (C \times Y).
\]
Because \(S \subset X \times K\), we can write:
\[
p^{-1}(C) = S \cap (C \times K).
\]
Since \(C\) and \(K\) are both compact, we know \(C \times K\) is compact. Furthermore, since \(S\) is a closed set in \(X \times Y\), the intersection \(S \cap (C \times K)\) is a closed subset of the compact space \(C \times K\). A closed subset of a compact space is compact, hence \(p^{-1}(C)\) is compact. This proves that \(p\) is a proper map.

\paragraph{Step 3 (The map \(p\) is a covering map):}
From Steps 1 and 2, we have established that \(p: S \to X\) is a surjective proper local homeomorphism. By Proposition \ref{prop:proper_covering}, it follows directly that \(p\) is a covering map.

\paragraph{Step 4 (Constructing the global continuous section \(y(x)\)):}
Since \(p:S\to X\) is a covering map, let \(S_0\) be the connected component of \(S\) that contains the initial point \((x_0,y_0)\). We first show that \(p(S_0)=X\). Since \(X\) is connected and locally path-connected, \(X\) is path-connected. For any \(x\in X\), choose a continuous path \(\gamma:[0,1]\to X\) such that \(\gamma(0)=x_0\) and \(\gamma(1)=x\). By the path lifting property of covering maps, \(\gamma\) has a unique lift \(\widetilde\gamma:[0,1]\to S\) satisfying \(\widetilde\gamma(0)=(x_0,y_0)\). Since \(\widetilde\gamma([0,1])\) is connected and contains \((x_0,y_0)\), it must be contained in the connected component \(S_0\). Therefore, \(\widetilde\gamma(1)\in S_0\) and $p(\widetilde\gamma(1))=\gamma(1)=x.$ Since \(x\in X\) was arbitrary, we obtain \(p(S_0)=X\).

We next show that the restriction \(p|_{S_0}:S_0\to X\) is also a covering map. Fix any \(x\in X\). Since \(p:S\to X\) is a covering map, there exists an evenly covered neighborhood \(U\) of \(x\). By shrinking \(U\) if necessary, we may assume that \(U\) is connected. Write \(p^{-1}(U)=\bigsqcup_{\alpha} V_\alpha\), where each \(p|_{V_\alpha}:V_\alpha\to U\) is a homeomorphism. Since \(U\) is connected, each \(V_\alpha\) is connected. Hence, if \(V_\alpha\cap S_0\neq\emptyset\), then \(V_\alpha\subset S_0\), because \(S_0\) is a connected component of \(S\). Therefore \(S_0\cap p^{-1}(U)\) is a union of some of these sheets \(V_\alpha\), each of which is mapped homeomorphically onto \(U\). Since \(x\in X\) was arbitrary, \(p|_{S_0}:S_0\to X\) is a covering map. 

We now show that this covering is one-sheeted, i.e., each fiber of \(q\coloneq p|_{S_0}:S_0\to X\) consists of exactly one point. Suppose that \(z_0,z_1\in S_0\) satisfy
\(q(z_0)=q(z_1)=x_0\). Since $X$ is locally path-connected and $q: S_0 \to X$ is a local homeomorphism, the space $S_0$ inherits local path-connectedness. Being both connected and locally path-connected, $S_0$ is path-connected. Hence there exists a path
\(\alpha:[0,1]\to S_0\) such that \(\alpha(0)=z_0\) and \(\alpha(1)=z_1\).
Then \(q\circ\alpha\) is a loop in \(X\) based at \(x_0\). Since \(X\) is simply
connected, this loop is homotopic relative to endpoints to the constant loop
at \(x_0\). By the homotopy lifting property for covering maps, the lifts of
these two homotopic paths starting at \(z_0\) have the same endpoint. The lift
of \(q\circ\alpha\) starting at \(z_0\) is \(\alpha\), whose endpoint is \(z_1\), while
the lift of the constant loop starting at \(z_0\) is the constant path at \(z_0\).
Thus \(z_1=z_0\). Therefore every fiber of \(q\) contains at most one point.
Since \(q\) is surjective, every fiber contains exactly one point. Hence \(q\)
is one-sheeted. A one-sheeted covering map is a homeomorphism, so
\(p|_{S_0}:S_0\to X\) is a homeomorphism. Now define
\[
y(x)\coloneq \pi_Y\bigl((p|_{S_0})^{-1}(x)\bigr),
\]
where \(\pi_Y:X\times Y\to Y\) is the projection onto the second component. Since \(p|_{S_0}\) is a homeomorphism and \(\pi_Y\) is continuous, the map \(y:X\to Y\) is continuous. Moreover, because \((p|_{S_0})^{-1}(x_0)=(x_0,y_0)\), we have \(y(x_0)=y_0\). Finally, for every \(x\in X\),
\[
(x,y(x))=(p|_{S_0})^{-1}(x)\in S_0\subset S.
\]
This establishes the existence of the desired global continuous section.

\paragraph{Step 5 (Global uniqueness):}
Suppose that \(\widetilde y:X\to Y\) is another continuous map such that
\(\widetilde y(x_0)=y_0\) and \((x,\widetilde y(x))\in S\) for all \(x\in X\).
Define \(\widetilde f(x)\coloneq (x,\widetilde y(x))\). Then \(\widetilde f:X\to S\)
is continuous. Since \(X\) is connected, \(\widetilde f(X)\) is connected.
Moreover, \(\widetilde f(x_0)=(x_0,y_0)\in S_0\). Since \(S_0\) is the
connected component of \(S\) containing \((x_0,y_0)\), we have
\(\widetilde f(X)\subset S_0\).

Therefore, for every \(x\in X\), the point \(\widetilde f(x)\) lies in \(S_0\)
and satisfies \(p(\widetilde f(x))=x\). Since \(p|_{S_0}:S_0\to X\) is a
homeomorphism, there is only one point in \(S_0\) that projects to \(x\).
Thus \(\widetilde f(x)=(p|_{S_0})^{-1}(x)\) for all \(x\in X\). Taking the
second component gives \(\widetilde y(x)=y(x)\) for all \(x\in X\). This proves
uniqueness.
\end{proof}

\begin{lemma}\label{lem:closedness_minimizers}
    Suppose Assumption~\ref{assumption:general2} holds, and the lower-level feasible set $\mathcal{Y}(x)$ is uniformly bounded for all $x \in \mathcal{X}$. Furthermore, assume that, at every $x \in \mathcal{X}$, LICQ, SCSC, and uniform SOSC hold for the lower-level problem.  Let $\mathcal{K}$ denote the graph of the local minimizers:
    \[
    \mathcal{K} \coloneq  \{ (x, y) \mid x \in \mathcal{X}, y \text{ is a local minimizer at } x \}.
    \]
    Then, $\mathcal{K}$ is a closed set.
\end{lemma}

\begin{proof}
Let $\{(x_k,y_k)\}\subset \mathcal K$ be an arbitrary convergent sequence such that
$(x_k,y_k)\to(\overline x,\overline y)$ as $k\to\infty$. We need to show that
$(\overline x,\overline y)\in\mathcal K$, i.e., $\overline y$ is a local minimizer
at $\overline x$. By continuity of the constraints, $\overline y$ is feasible for
the lower-level problem at $\overline x$. Since each $y_k$ is a local minimizer at $x_k$, it satisfies the KKT conditions.
By LICQ, the associated multiplier $\lambda_k$ is unique. Let
$A_k\coloneq \{i:h_i(x_k,y_k)=0\}$. Since there are only finitely many active index sets,
after passing to a subsequence and relabeling if necessary, we may assume
$A_k=A$ for all $k$. For each $i\in A$, taking the limit in
$h_i(x_k,y_k)=0$ gives $h_i(\overline x,\overline y)=0$, so
$A\subseteq A(\overline x,\overline y)$. Since LICQ holds at
$(\overline x,\overline y)$, the gradients
$\{\nabla_y h_i(\overline x,\overline y):i\in A\}$ are linearly independent.
Thus, for all large $k$, the matrices $\nabla_y h_A(x_k,y_k)$ have uniformly full
row rank. The stationarity condition can be written as
\begin{align}\label{eq:stabilityofactiveset}
    \nabla_y g(x_k,y_k)+\nabla_y h_A(x_k,y_k)^\top\lambda_{k,A}=0 .    
\end{align}
Hence $\{\lambda_{k,A}\}$ is bounded. Passing to a further subsequence if necessary,
we have $\lambda_{k,A}\to\overline\lambda_A\ge 0$. Set
$\overline\lambda_i=0$ for $i\notin A$. Taking limits in the KKT system gives that
$(\overline y,\overline\lambda)$ satisfies the KKT conditions at $\overline x$.
Thus, $\overline y$ is a KKT point at $\overline x$.

Note that SCSC holds at this KKT point. If
$j\in A(\overline x,\overline y)\setminus A$, then by construction
$\overline\lambda_j=0$, contradicting SCSC. Therefore
$A=A(\overline x,\overline y)$, and the active set is eventually constant along
the subsequence. Consequently, the critical cones vary continuously: for any
$d\in\mathcal{C}(\overline x,\overline y)$ with $\|d\|=1$, there exists
$v_k\in\mathcal{C}(x_k,y_k)$ such that $v_k\to d$. Since each $y_k$ is a local minimizer, the uniform SOSC assumption gives $v_k^\top \nabla^2_{yy}L(x_k,y_k,\lambda_k)v_k
\ge \sigma\|v_k\|^2$. Taking the limit and using the continuity of the Hessian yields $d^\top\nabla^2_{yy}L(\overline x,\overline y,\overline\lambda)d
\ge \sigma\|d\|^2=\sigma>0$. Thus, the limit KKT point satisfies SOSC. By the standard second-order sufficient
condition in nonlinear programming, $\overline y$ is a strict local minimizer of
the lower-level problem at $\overline x$. Hence
$(\overline x,\overline y)\in\mathcal K$, and the graph $\mathcal K$ is closed.
\end{proof}

\noindent\textbf{Detailed proof of Theorem~\ref{thm:global_continuation}:}
The proof proceeds in two steps. We first obtain a unique global continuous
continuation by applying Lemma~\ref{lem:global_extension}, and then upgrade its
regularity by the Implicit Function Theorem.

\paragraph{Step 1 (Global continuous continuation):}
Define the graph of local minimizers by
\[
\mathscr K
\coloneq 
\{(x,y): x\in\mathcal X,\ y \text{ is a local minimizer of the lower-level problem at } x\}.
\]
Let \(p:\mathscr K\to\mathcal X\) be the natural projection \(p(x,y)=x\).

We first note that \(p(\mathscr K)=\mathcal X\). Indeed, for every
\(x\in\mathcal X\), the feasible set \(\mathcal Y(x)\) is nonempty, closed, and
bounded, hence compact. Since \(g(x,\cdot)\) is continuous, the lower-level
problem admits a global minimizer, and therefore at least one local minimizer.
Thus every \(x\in\mathcal X\) belongs to \(p(\mathscr K)\).

We now verify the three conditions of Lemma~\ref{lem:global_extension}.

\begin{itemize}
    \item \textit{Condition 1 (Local Homeomorphism).}
    Fix any \((x_0,y_0)\in\mathscr K\), and let $\mathcal J\coloneq \mathcal J(x_0,y_0).$
    Let \(\lambda_0\) be the KKT multiplier associated with \(y_0\), which is unique
    by LICQ. By SCSC, \((\lambda_0)_i>0\) for every \(i\in\mathcal J\), and all
    constraints outside \(\mathcal J\) are strictly inactive.

    Consider the reduced KKT system
    \[
    \Phi_{\mathcal J}(x,y,\lambda_{\mathcal J})
    \coloneq 
    \begin{pmatrix}
    \nabla_y g(x,y)+\nabla_y h_{\mathcal J}(x,y)^\top\lambda_{\mathcal J}\\
    h_{\mathcal J}(x,y)
    \end{pmatrix}
    =0.
    \]
    Under LICQ and uniform SOSC, the Jacobian of
    \(\Phi_{\mathcal J}\) with respect to \((y,\lambda_{\mathcal J})\) is nonsingular
    at \((x_0,y_0,(\lambda_0)_{\mathcal J})\). Therefore, by the Implicit Function
    Theorem, there exist neighborhoods \(U\) of \(x_0\), \(V\) of \(y_0\), and a
    unique smooth map $x\mapsto (y(x),\lambda_{\mathcal J}(x))$
    satisfying the reduced KKT system in \(U\times V\).

    Shrinking \(U\) and \(V\) if necessary, SCSC and continuity imply that this
    branch has the same active index set \(\mathcal J\), remains feasible, and
    satisfies the second-order sufficient condition. Hence \(y(x)\) remains a
    local minimizer for all \(x\in U\).

    We claim that this gives the required local homeomorphism for the graph
    \(\mathscr K\). Suppose \((x,z)\in\mathscr K\cap(U\times V)\). By LICQ, \(z\)
    has a unique KKT multiplier \(\mu\). By the active-set stability argument used
    in Lemma~\ref{lem:closedness_minimizers}, after shrinking \(U\) and \(V\) if
    necessary, any such nearby local minimizer has the same active index set
    \(\mathcal J\). Thus \((z,\mu_{\mathcal J})\) solves the same reduced KKT
    system. By the local uniqueness from the Implicit Function Theorem, we have $(z,\mu_{\mathcal J})=(y(x),\lambda_{\mathcal J}(x))$. Hence \(z=y(x)\). Therefore, $\mathscr K\cap(U\times V)=\{(x,y(x)):x\in U\}$,    which proves that \(p:\mathscr K\to\mathcal X\) is a local homeomorphism.

    \item \textit{Condition 2 (Closed Graph).}
    By Lemma~\ref{lem:closedness_minimizers}, the graph \(\mathscr K\) is closed.

    \item \textit{Condition 3 (Uniform Compactness).}
    Since the feasible sets \(\mathcal Y(x)\) are uniformly bounded, there exists
    a compact set \(K_Y\subset\mathbb R^m\) such that
    $\mathcal Y(x)\subseteq K_Y,\ \forall x\in\mathcal X.$
    Every local minimizer lies in \(\mathcal Y(x)\). Therefore, $\mathscr K\subseteq \mathcal X\times K_Y.$
\end{itemize}

All conditions of Lemma~\ref{lem:global_extension} are satisfied. Hence, for the
given \((x_0,y_0)\in\mathscr K\), there exists a unique global continuous map
\(y^\ast:\mathcal X\to\mathbb R^m\) such that
\[
y^\ast(x_0)=y_0,
\qquad
(x,y^\ast(x))\in\mathscr K,\quad \forall x\in\mathcal X.
\]

\paragraph{Step 2 (Smoothness):}
It remains to show that \(y^\ast\) is smooth. Fix any \(\overline x\in\mathcal X\), and
set
\[
\overline y\coloneq y^\ast(\overline x),
\qquad
\overline{\mathcal J}\coloneq \mathcal J(\overline x,\overline y).
\]
Let \(\overline\lambda\) be the unique KKT multiplier associated with \(\overline y\). By the
same Implicit Function Theorem argument as above, the reduced KKT system with
active index set \(\overline{\mathcal J}\) admits a unique smooth local solution $x\mapsto (\widehat y(x),\widehat\lambda_{\overline{\mathcal J}}(x))$
near \(\overline x\), with \(\widehat y(\overline x)=\overline y\). Since \(y^\ast\) is continuous and consists of local minimizers, the active-set
stability argument used above implies that, after shrinking the neighborhood of
\(\overline x\) if necessary, \(y^\ast(x)\) has the same active index set
\(\overline{\mathcal J}\). Let \(\lambda^\ast(x)\) be its unique KKT multiplier. Then $(y^\ast(x),\lambda^\ast_{\overline{\mathcal J}}(x))$ solves the same reduced KKT system near
\((\overline x,\overline y,\overline\lambda_{\overline{\mathcal J}})\). By the local uniqueness from
the Implicit Function Theorem, $(y^\ast(x),\lambda^\ast_{\overline{\mathcal J}}(x))
=
(\widehat y(x),\widehat\lambda_{\overline{\mathcal J}}(x))$ holds for all \(x\) sufficiently close to \(\overline x\). Hence \(y^\ast=\widehat y\) locally
near \(\overline x\). Since \(\widehat y\) is smooth and \(\overline x\) was arbitrary,
\(y^\ast\) is smooth on \(\mathcal X\).

\subsection{Proof of Theorem~\ref{prop:uniform_quadratic_growth}}\label{proof:prop:uniform_quadratic_growth}
Suppose for contradiction that the uniform local quadratic growth condition does not hold. Then, for any $k \in \mathbb{N}^+$, there exist $x_k \in \mathcal{X}$, a local minimizer $y^*(x_k)$ at $x_k$, and a feasible point $y_k \in \mathcal{Y}(x_k)$ such that $0 < \|y_k - y^*(x_k)\| < \frac{1}{k}$, and\begin{equation}\label{eq:contradiction_growth}g(x_k, y_k) < g(x_k, y^*(x_k)) + \frac{1}{k}|y_k - y^*(x_k)|^2.\end{equation}Since $\mathcal{X}$ is compact and $\mathcal{Y}(x)$ is uniformly bounded, the sequence of pairs $\{ (x_k, y^*(x_k)) \}$ is bounded. We can extract a convergent subsequence (without relabeling) such that $x_k \to \overline{x}$ and $y^*(x_k) \to \overline{y}$. Because $\|y_k - y^*(x_k)\| < \frac{1}{k} \to 0$, we also have $y_k \to \overline{y}$. Let $\mathcal{K}$ denote the graph of the local minimizers. By Lemma~\ref{lem:closedness_minimizers}, $\mathcal{K}$ is a closed set. Since $(x_k, y^*(x_k)) \in \mathcal{K}$ for all $k$ and converges to $(\overline{x}, \overline{y})$, the closedness directly implies $(\overline{x}, \overline{y}) \in \mathcal{K}$. Therefore, $\overline{y}$ is a valid local minimizer at $\overline{x}$. Let $\overline{\lambda}$ be its unique multiplier. By LICQ at $\overline{x}$, the multipliers $\lambda(x_k)$ of $y^*(x_k)$ converge to $\overline{\lambda}$. Define $t_k \coloneq  \|y_k - y^*(x_k)\| > 0$ and the normalized direction $d_k \coloneq  \frac{y_k - y^*(x_k)}{t_k}$. Since $\|d_k\| = 1$, we can pass to a subsequence such that $d_k \to d$ with $\|d\| = 1$. We now show that $d$ belongs to the critical cone $\mathcal{C}(\overline{x}, \overline{y})$. By assumption, SCSC holds at the limit KKT point $\overline{y}$. Thus, $\overline{\lambda}_i > 0$ for all active constraints $i \in \overline{A}$. Since $\lambda_i(x_k) \to \overline{\lambda}_i$, we have $\lambda_i(x_k) > 0$ for sufficiently large $k$. By complementary slackness at $y^*(x_k)$, this forces $h_i(x_k, y^*(x_k)) = 0$ for all $i \in \overline{A}$. Using $y_k \in \mathcal{Y}(x_k)$, we have $h_i(x_k, y_k) \le 0 = h_i(x_k, y^*(x_k))$. Taking the first-order Taylor expansion and dividing by $t_k$ yields $\nabla_y h_i(\overline{x}, \overline{y})^\top d \le 0$. Similarly, expanding \eqref{eq:contradiction_growth} yields $\nabla_y g(\overline{x}, \overline{y})^\top d \le 0$. Combined with the KKT stationarity condition and $\overline{\lambda}_i > 0$, this ensures $\nabla_y h_i(\overline{x}, \overline{y})^\top d = 0$ for all $i \in \overline{A}$. Hence, $d \in \mathcal{C}(\overline{x}, \overline{y})$. Finally, we consider the Lagrangian $L(x_k, y, \lambda(x_k))$. Since $y^*(x_k)$ is a KKT point, $\nabla_y L(x_k, y^*(x_k), \lambda(x_k)) = 0$. Since $y_k \in \mathcal{Y}(x_k)$ and $\lambda(x_k) \ge 0$, we have $L(x_k, y_k, \lambda(x_k)) \le g(x_k, y_k)$. Applying a second-order Taylor expansion to the Lagrangian at $y^*(x_k)$ along direction $d_k$, and utilizing \eqref{eq:contradiction_growth}, we obtain:$$\frac{1}{2} d_k^\top \nabla^2_{yy} L(x_k, \widetilde{y}_k, \lambda(x_k)) d_k \le \frac{g(x_k, y_k) - g(x_k, y^*(x_k))}{t_k^2} < \frac{1}{k},$$where $\widetilde{y}_k$ lies between $y_k$ and $y^*(x_k)$. Taking the limit as $k \to \infty$, the continuity of the Hessian gives $d^\top \nabla^2_{yy} L(\overline{x}, \overline{y}, \overline{\lambda}) d \le 0$. However, as shown above, SOSC holds at $\overline{y}$. For the non-zero critical direction $d \in \mathcal{C}(\overline{x}, \overline{y})$, SOSC requires $d^\top \nabla^2_{yy} L(\overline{x}, \overline{y}, \overline{\lambda}) d > 0$, which is a contradiction. Therefore, the uniform local quadratic growth condition must hold.

\end{document}

%% file: _macros.tex
\usepackage{comment,url,graphicx,relsize}
\usepackage{amssymb,amsfonts,amsmath,amsthm,amscd,dsfont,mathrsfs,mathtools,nicefrac, bm}
\usepackage{float,psfrag,epsfig,color,xcolor,url}
\usepackage{epstopdf,bbm,mathtools,enumitem}
\usepackage{subfigure}
\usepackage[ruled,vlined]{algorithm2e}
\usepackage{tablefootnote}
\graphicspath{{Plots/}}
\usepackage{diagbox}
\usepackage{tikz}
\usetikzlibrary{calc,patterns,angles,quotes}
\usepackage{makecell}
\usepackage{multirow}
\usepackage{booktabs}

\providecommand{\argmin}{\mathop\mathrm{arg min}}

\ifdefined\nonewproofenvironments\else
\ifdefined\ispres\else
\renewenvironment{proof}{\noindent\textbf{Proof.}\hspace*{.3em}}{\qed\\}
\newenvironment{proof-sketch}{\noindent\textbf{Proof Sketch}
  \hspace*{0.em}}{\qed\bigskip\\}
\newenvironment{proof-idea}{\noindent\textbf{Proof Idea}
  \hspace*{0.em}}{\qed\bigskip\\}
\newenvironment{proof-of-lemma}[1][{}]{\noindent\textbf{Proof of Lemma {#1}.}
  \hspace*{0.em}}{\qed\\}
\newenvironment{proof-of-corollary}[1][{}]{\noindent\textbf{Proof of Corollary {#1}.}
  \hspace*{0.em}}{\qed\\}
\newenvironment{proof-of-theorem}[1][{}]{\noindent\textbf{Proof of Theorem {#1}.}
  \hspace*{0.em}}{\qed\\}
\newenvironment{proof-attempt}{\noindent\textbf{Proof Attempt}
  \hspace*{0.em}}{\qed\bigskip\\}

\fi
\newtheorem{theorem}{Theorem}[section]
\newtheorem{lemma}{Lemma}[section]

\newtheorem{proposition}{Proposition}[section]
\newtheorem{assumption}{Assumption}[section]
\newtheorem{remark}{Remark}[section]
\newtheorem{example}{Example}[section]
\newtheorem{counterexample}[theorem]{Counterexample}
\newtheorem{definition}{Definition}[section]

\usetikzlibrary{calc}